\documentclass[11pt]{amsart}
\usepackage{amsbsy,amssymb,amscd,amsfonts,latexsym,amstext,delarray,
amsmath,graphicx} 
\usepackage[pdftex]{color}

\newtheorem{thm}{Theorem}[section]
\newtheorem{prop}[thm]{Proposition}
\newtheorem{cor}[thm]{Corollary}
\newtheorem{lem}[thm]{Lemma}
\newtheorem{defn}[thm]{Definition}
\newtheorem{rem}[thm]{Remark}

\numberwithin{equation}{section}

\newcommand\BE{{\bf E}}

\newcommand\BG{{\bf G}}
\newcommand\BS{{\bf S}}
\newcommand\BV{{\bf V}}

\def\bL{{\mathbb L}}

\def\A{{\mathbb A}}
\def\C{{\mathbb C}}

\renewcommand{\H}{{\mathbb H}}
\def\N{{\mathbb N}}
\renewcommand{\P}{{\mathbb P}}

\def\Z{{\mathbb Z}}
\def\R{{\mathbb R}}

\def\cA{{\mathcal A}}

\def\cC{{\mathcal C}}
\def\cD{{\mathcal D}}
\def\cE{{\mathcal E}}
\def\cF{{\mathcal F}}
\def\cG{{\mathcal G}}
\def\cH{{\mathcal H}}
\def\cI{{\mathcal I}}

\def\cK{{\mathcal K}}
\def\cL{{\mathcal L}}

\def\cN{{\mathcal N}}

\def\cR{{\mathcal R}}
\def\cS{{\mathcal S}}

\def\cU{{\mathcal U}}
\def\cV{{\mathcal V}}

\def\cX{{\mathcal X}}
\def\cY{{\mathcal Y}}
\def\cZ{{\mathcal Z}}

\def\m{{\mathfrak m}}

\title{Feynman integrals and periods in configuration spaces}
\author{\"Ozg\"ur Ceyhan}
\author{Matilde Marcolli}
\address{Facult\'e des Sciences, de la Technologie et de la Communication,
University of Luxembourg, 6 rue Richard Coudenhove-Kalergi, L-1359 Luxembourg}
\email{ozgur.ceyhan@uni.lu}
\address{Department of Mathematics,
Division of Physics, Mathematics and Astronomy, Mail Code 253-37, California
Institute of Technology, 1200 E.~California Blvd. Pasadena, CA 91125, USA}
\email{matilde@caltech.edu}
\begin{document}
\maketitle

\bigskip

{\it To Tanju  Ba\c{s}ar\i r \"Ozkan, who saved Uzay}

\bigskip

\begin{abstract}
We describe differential forms representing Feynman amplitudes in
configuration spaces of Feynman graphs, and regularization and 
evaluation techniques, for suitable chains of integration, that give 
rise to periods of mixed Tate motives.
\end{abstract}

\tableofcontents

\section{Introduction}\label{IntroSec}

This work is a continuation of our investigation \cite{CeyMar} of algebro-geometric
and motivic aspects of Feynman integrals in configuration spaces and their relation
to periods of mixed Tate motives.

\smallskip

Section \ref{FeyampSec} introduces algebraic varieties
$X^{\BV_\Gamma}$ and $Z^{\BV_\Gamma}$ that describe
configuration spaces associated to a Feynman graph $\Gamma$,
and smooth differential forms on these varieties, with singularities
along diagonals, that extend the Feynman rules and Feynman 
propagators in configuration spaces to varieties containing
a dense affine space that represents the ``physical space"
on which the Feynman diagram integration lives. 
We consider three possible variants of the geometry, and
of the corresponding differential form, respectively given by
the forms  \eqref{amplitude1}, \eqref{amplitudeZ} and 
\eqref{amplitudeZ2}, because this will allow us to present
different methods of regularization and integration and
show, in different ways, how the relation between Feynman
integrals and periods of mixed Tate motives arises in the
setting of configuration spaces. 

\smallskip

The geometric setting builds upon our previous work \cite{CeyMar},
and we will be referring to that source for several of the
algebro-geometric arguments we need to use here. In terms 
of the Feynman amplitudes themselves, the main difference 
between the approach followed in this paper and the one of
\cite{CeyMar} is that in our previous work we extended the
Feynman propagator to the configuration space $X^{\BV_\Gamma}$
as an algebraic differential form, which then had singularities not
only along the diagonals, but also along a quadric $Z_\Gamma$
(the configuration space analog of the graph hypersurfaces describing
singularities of Feynman amplitudes in momentum space). In this
paper, we extend the Feynman propagator to a $\cC^\infty$ (non-algebraic)
differential form on $X^{\BV_\Gamma}$, which then has singularities only
along the diagonals. In this way, we gain the fact that the motives involved
are easy to control, and do not leave the class of Tate motives, while we
move the difficulty in obtaining an interpretation of Feynman integrals
as periods to the fact of working with a non-algebraic differential
form. In the various sections of the paper we show different ways in
which one can overcome this problem and end up with evaluations
of a suitably regularized Feynman amplitude that gives periods of
mixed Tate motives (multiple zeta values, by the recent result of
\cite{Brown_mt}).

\smallskip

Section \ref{BMsec} can be read independently of the rest of the
paper (or can be skipped by readers more directly interested in
our main results on periods). Its main purpose is to explain the 
relation between the Feynman amplitudes \eqref{amplitude1},
\eqref{amplitudeZ} and \eqref{amplitudeZ2} and real and
complex Green functions and the Bochner--Martinelli integral
kernel operator. In particular, we show that the form of the
Feynman amplitude leads naturally to a Bochner--Martinelli
type theorem, adapted to the geometry of graph configuration
spaces, which combines the usual Bochner--Martinelli theorem
on complex manifolds \cite{GrHa} and the notion of harmonic
functions and Laplacians on graphs.

\smallskip

In Section \ref{GegenSec}, we consider the first form \eqref{amplitude1}
of the Feynman amplitude and its restriction to the affine space
$\A^{D\BV_\Gamma}(\R)$ in the real locus in the graph configuration 
space $X^{\BV_\Gamma}$, with $X=\P^D(\C)$. We first describe a
setting generalizing a recent construction, by Spencer Bloch, of cycles
in the configuration spaces $X^{\BV_\Gamma}$. Our setting assigns 
a middle-dimensional relative cycle, with boundary on the union of diagonals, 
to any acyclic orientation on the graph.  We then consider the development
of the propagator into Gegenbauer orthogonal polynomials. This is a technique
for the calculation of Feynman diagrams in configuration spaces (known 
as $x$-space technique) that is well established
in the physics literature since at least the early '80s, \cite{CheKaTka}. 
We split these $x$-space integrals into contributions parameterized
by the chains described above. The decompositions of the chains of
integration, induces a corresponding
decomposition of the Feynman integral into an angular and a radial
part.  In the case of dimension $D=4$, and for polygon graphs
these simplify into integrals over simplexes
of polylogarithm functions, which give rise to zeta values. 
For more general graphs, each star of a vertex contributes a
certain combination of integrals of triple integrals of spherical
harmonics. These can be expressed in terms of isoscalar coefficients.
We compute explicitly, in the case $D=4$ the top term of this expansion,
and show that, when pairs of half edges are joined together to form
an edge of the graph, one obtains a combination of nested
sums that can be expressed in terms of Mordell--Tornheim and
Apostol--Vu series. 

\smallskip

Starting with \S \ref{wondSec}, we consider the form \eqref{amplitudeZ}
of the Feynman amplitude, which corresponds to a complexification,
and a doubling of the dimension of the relevant configuration spaces 
$Z^{\BV_\Gamma}\simeq (X\times X)^{\BV_\Gamma}$.
The advantage of passing to this formulation is that one is then
dealing with a closed form, unlike the case of \eqref{amplitude1} 
considered in the previous section, hence cohomological arguments 
become available. We first describe the simple modifications to the 
geometry of configuration spaces, with respect to the results of our 
previous work \cite{CeyMar}, which are needed in order to deal 
with this doubling of dimension. We introduce the wonderful 
compactification $F(X,\Gamma)$ of the configuration space $Z^{\BV_\Gamma}$,
which works pretty much as in the case of the wonderful compactification
$\overline{\rm Conf}_\Gamma(X)$ of $X^{\BV_\Gamma}$ described in
\cite{CeyMar}, \cite{li}, \cite{li2}, with suitable modifications. The purpose of passing to
the wonderful compactification is to realize the complement of the
diagonals in the configuration space explicitly as the complement of
a union of divisors intersecting transversely, so that we can describe
de Rham cohomology classes in terms of representatives that are
algebraic forms with logarithmic poles along the divisors. 
We also discuss, in this section, the properties of the motive of the wonderful
compactification and we verify that, if the underlying variety $X$ is a Tate
motive, the wonderful compactification also is, and so are the unions
and intersections of the divisors obtained in the construction. We also
describe iterated Poincar\'e residues of these forms along the
intersections of divisors corresponding to $\cG_\Gamma$-nests of
biconnected induced subgraphs of the Feynman graph, according to
the general theory of iterated Poincar\'e residues of forms with
logarithmic poles, \cite{AYu} , \cite{ATYu}, \cite{Deligne}, \cite{Griff}. 

\smallskip

In \S \ref{cutoffreg}, we consider a first regularization method for
the Feynman integral based on the amplitude \eqref{amplitudeZ},
which is essentially a cutoff regularization and is expressed in
terms of principal-value currents and the theory of Coleff--Herrera
residue currents, \cite{BGVY}, \cite{CoHe}. We show that, when we
regularize the Feynman integral as a principal value current, the
ambiguity is expressed by residue currents supported on the
intersections of the divisors of the wonderful compactification,
associated to the $\cG_\Gamma$-nests, which are related to
the iterated Poincar\'e residues described in \S \ref{wondSec}.
In particular, when evaluated on algebraic test differential forms 
on these intersections, the currents describing the ambiguities
take values that are periods of mixed Tate motives.

\smallskip

In \S \ref{RegSec}, we introduce a more directly algebro-geometric
method of regularization of both the Feynman amplitude form \eqref{amplitudeZ}
and of the chain of integration, which is based on the {\em deformation
to the normal cone} \cite{Ful}, which we use to separate the chain of integration from
the locus of divergences of the form. We check again that the motive of the
deformation space constructed using the deformation to the normal cone
remains mixed Tate, and we show once again that the regularized Feynman
integral obtained in this way gives rise to a period.

\section{Feynman amplitudes in configuration space}\label{FeyampSec}

In the following we let $X$ be a $D$-dimensional smooth
projective variety, which contains a dense subset
isomorphic to an affine space $\A^D$, whose set
of real points $\A^D(\R)$ we identify with Euclidean 
$D$-dimensional spacetime.
For instance, we can take $X=\P^D$. 
We assume that $D$ is even and we write
$D=2\lambda+2$.

\medskip

\subsection{Feynman graphs}

A {\em graph} $\Gamma$ is a  one-dimensional finite CW-complex. We denote 
the set of vertices of $\Gamma$ by $\BV_\Gamma$ , the set of edges by 
$\BE_\Gamma$, and the boundary map by $\partial_{\Gamma}: \BE_{\Gamma} 
\to (\BV_{\Gamma})^2$. 

When an orientation is chosen on $\Gamma$, we
write $\partial_\Gamma(e)=\{ s(e), t(e) \}$, the source and target vertices of
the oriented edge $e$.
 
A {\it  looping}  edge is an edge for which the two boundary vertices coincide and {\it multiple 
edges} are edges between the same pair of vertices. We assume that  our graphs have 
no  looping edges.

\subsubsection{Induced subgraphs and quotient graphs}
\label{SGdef}

A subgraph $\gamma \subseteq \Gamma$ is called an {\em induced subgraph} if its set 
of vertices $\BE_\gamma$ is equal to $\partial^{-1}_\Gamma((\BV_{\gamma})^2)$, that is, 
$\gamma$ has all the edges of $\Gamma$ on the same set of vertices. We will
denote by  $\BS\BG(\Gamma)$ the set of all induced subgraphs of $\Gamma$ and by
\begin{equation}\label{BSGk}
\BS\BG_k (\Gamma)  = \{\gamma \in  \BS\BG(\Gamma) \mid  |\BV_\gamma|=k \},
\end{equation}
the subset $\BS\BG_k (\Gamma) \subseteq \BS\BG(\Gamma)$ of all    
induced subgraphs with $k$ vertices.  
 
For $\gamma \in \BS\BG(\Gamma)$, we denote by $\Gamma//\gamma$ the graph 
obtained from $\Gamma$ by shrinking each connected component of $\gamma$ to a 
separate  vertex. The quotient graph  $\Gamma//\gamma$ does not have looping
edges since  we consider only induced subgraphs. 

\medskip
\subsection{Feynman amplitude}

When computing Feynman integrals in configuration space, one
considers singular differential forms on $X^{\BV_\Gamma}$, integrated
on the real locus of this space.

\begin{defn}\label{formdef}
The Euclidean massless Feynman amplitude in configuration space
is given by the differential form
\begin{equation}\label{amplitude1}
\omega_\Gamma = \prod_{e\in \BE_\Gamma} 
\frac{1}{\| x_{s(e)} - x_{t(e)} \|^{2\lambda}} 
\bigwedge_{v\in \BV_\Gamma} dx_v . 
\end{equation}
\end{defn}

The form \eqref{amplitude1} defines a $\cC^\infty$-differential form on the configuration space
\begin{equation}\label{ConfGamma}
 {\rm Conf}_\Gamma(X)= X^{\BV_\Gamma} \smallsetminus \cup_{e\in \BE_\Gamma} \Delta_e , 
\end{equation} 
with singularities along the diagonals 
\begin{equation}\label{DeltaeX}
\Delta_e =\{ (x_v\,|\, v\in \BV_\Gamma) \,\, | \,\, x_{s(e)}=x_{t(e)} \}. 
\end{equation}

\smallskip

\begin{rem}\label{multedgerem} {\rm
The diagonals \eqref{DeltaeX} corresponding to edges between the same 
pair of vertices are the same, consistently with the fact that the notion of 
degeneration that defines the diagonals is based on ``collisions of points" and not 
on contraction of the edges connecting them. 
This suggests that we may want to exclude graphs with multiple edges. However, 
multiple edges play a role in the definition of Feynman amplitudes (see
Definition \ref{formdef} above and  \S 4 of \cite{CeyMar})
hence we allow this possibility.  On the other hand,  the 
definition of configuration space is void in the presence of  looping edges, since
the diagonal  $\Delta_e$ associated to a looping edge is the whole space 
$X^{\BV_\Gamma}$,  and its complement is empty.
Thus, our assumption that graphs have no 
looping edges is aimed at excluding this
geometrically trivial case.  }
\end{rem}

\smallskip

\begin{rem}\label{proprem} {\rm
Our choice of $\| x_{s(e)} - x_{t(e)} \|^2$ in the  Feynman propagator  differs from the
customary choice of propagator (see for instance \cite{CeyMar}) where 
$(x_{s(e)} - x_{t(e)})^2$ is used instead, but these two expressions 
agree on the locus $X^{\BV_\Gamma}(\R)$ of real points, which is the
chain of integration of the Feynman amplitude. The latter choice gives a
manifestly algebraic differential form, but at
the cost of introducing a hypersurface $\cZ_\Gamma$ in $X^{\BV_\Gamma}$ 
where the singularities of the form occur, which makes it difficult to control
explicitly the nature of the motive. Our choice here only gives a $\cC^\infty$ 
differential form, but the domain of definition is now simply ${\rm Conf}_\Gamma(X)$, 
whose motivic nature is much easier to understand, see \cite{CeyMar}. }
\end{rem}

\smallskip

Formally (before considering the issue of divergences), the Feynman integral
is obtained by integrating the form \eqref{amplitude1} on the locus
of real points of the configuration space.

\begin{defn}\label{intchaindef}
The chain of integration for the Feynman amplitude is taken to be
the set of real points of this configuration space, 
\begin{equation}\label{sigmaGamma}
\sigma_\Gamma = X(\R)^{\BV_\Gamma}. 
\end{equation}
\end{defn}

\smallskip

The form \eqref{amplitude1} defines a top form on $\sigma_\Gamma$.
There are two sources of divergences in integrating the form 
\eqref{amplitude1} on $\sigma_\Gamma$: the intersection between
the chain of integration and the locus of divergences of $\omega_\Gamma$,
$$ \sigma_\Gamma \cap \cup_e \Delta_e = \cup_e \Delta_e(\R) $$
and the behavior at infinity, on $\Delta_\infty := X\smallsetminus \A^D$.
In physics terminology, these correspond, respectively, to the {\em ultraviolet} 
and {\em infrared} divergences.

\smallskip

We will address these issues in \S \ref{GegenSec} below.

\medskip
\subsection{Variations upon a theme}\label{variousSec}

In addition to the form \eqref{amplitude1} considered above, we will also
consider other variants, which will allow us to discuss different possible
methods to address the question of periods and in particular the
occurrence of multiple zeta values in Feynman integrals in configuration
spaces.

\medskip

\subsubsection{Complexified case}

In this setting, instead of the configuration space $X^{\BV_\Gamma}$ and
its locus of real points $\sigma=X^{\BV_\Gamma}(\R)$, we will work with
a slightly different space, obtained as follows.

As above, let $X$ be a smooth projective variety, and  $Z$ denote the
product $X \times X$. Let $p: Z \to X$, $p:  z=(x,y) \mapsto x$ be the 
projection.  

Given a graph $\Gamma$, the {\it configuration space} $F(X,\Gamma)$ of 
$\Gamma$  in $Z$ is the complement 
\begin{equation} \label{Zconfig}
Z^{\BV_\Gamma} \setminus \bigcup_{e \in \BE_\Gamma} \Delta^{(Z)}_e \cong 
(X \times X) ^{\BV_\Gamma} \setminus \bigcup_{e \in \BE_\Gamma} \Delta^{(Z)}_e,
\end{equation}
in the cartesian product $Z^{\BV_\Gamma}=\{( z_v \mid v \in \BV_\Gamma)\}$ 
of the {\it  diagonals associated to the edges} of $\Gamma$, 
\begin{equation}\label{DeltaeZ}
\Delta^{(Z)}_e \cong \{( z_v \mid v \in \BV_\Gamma) \in Z^{\BV_\Gamma} \mid p( z_{s(e)}) = p( z_{t(e)}) \} .
\end{equation}

The relation between the configuration space $F(X,\Gamma)$ and the
previously considered ${\rm Conf}_\Gamma(X)$ of \eqref{ConfGamma}
is described by the following.

\begin{lem} \label{XZconfLem}
The configuration space $F(X,\Gamma)$ is isomorphic to
\begin{equation}\label{XZconf}
F(X,\Gamma) \simeq {\rm Conf}_\Gamma(X) \times X^{\BV_\Gamma},
\end{equation}
and the diagonals \eqref{DeltaeZ} and related to those of \eqref{DeltaeX}
by $\Delta^{(Z)}_e \cong \Delta_e \times X^{\BV_\Gamma}$.
\end{lem}

\medskip
\subsubsection{Feynman amplitudes in the complexified case}\label{weightsSec}

We assume that the smooth projective variety $Z$ contains a dense
open set isomorphic to affine space $\A^{2D}$, with coordinates
$z = (x_1,\dots,x_D,y_1,\dots,y_D)$.

\begin{defn}\label{FeyamplZcase}
Given a Feynman $\Gamma$ with no looping edges, we define the
corresponding Feynman amplitude (weight) as 
\begin{equation}\label{amplitudeZ}
\omega^{(Z)}_\Gamma = \prod_{e\in \BE_\Gamma} \frac{1}{\| x_{s(e)}- x_{t(e)} \|^{2D-2}} 
\bigwedge_{v \in\BV_\Gamma} dx_v \wedge d\bar x_v ,
\end{equation}
where $\| x_{s(e)} - x_{t(e)} \|=\| p(z)_{s(e)}- p(z)_{t(e)} \|$ and where
the differential forms $dx_{v}$ and $d\bar x_v$ denote, respectively, 
the holomorphic volume form $dx_{v,1} \wedge \cdots \wedge dx_{v,D}$ and 
its conjugate. The chain of integration is given, in this case, by the range 
of the projection 
\begin{equation}\label{sigmaZ}
\sigma^{(Z,y)} = X^{\BV_\Gamma}\times \{ y=(y_v) \} \subset Z^{\BV_\Gamma}=
X^{\BV_\Gamma} \times X^{\BV_\Gamma},
\end{equation}
for a fixed choice of a point $y=(y_v\, |\, v\in \BV_\Gamma)$.
\end{defn}

\medskip

The form \eqref{amplitudeZ} is a closed $\cC^\infty$ differential form
on $F(X,\Gamma)$ of degree $2\dim_\C X^{\BV_\Gamma}= \dim_\C Z^{\BV_\Gamma} 
= 2D|\BV_\Gamma|$, hence it
defines a cohomology class in $H^{2D|\BV_\Gamma|}(F(X,\Gamma))$, 
and it gives a top form on the locus $\sigma^{(Z,y)}$.

\smallskip

In this case, the divergences of $\int_{\sigma^{(Z,y)}} \omega^{(Z)}_\Gamma$
are coming from the union of the diagonals $\cup_{e\in \BE_\Gamma} \Delta_e^{(Z)}$
and from the divisor at infinity 
\begin{equation}\label{DeltainftyZ}
\Delta^{(Z)}_{\infty,\Gamma} := \sigma^{(Z,y)}  \smallsetminus \A^{D|\BV_\Gamma|}(\C)\times \{ y \}.
\end{equation}
We will address the behavior of the integrand near the loci 
$\cup_{e\in \BE_\Gamma} \Delta_e^{(Z)}$ and $\Delta^{(Z)}_{\infty,\Gamma}$ and the appropriate
regularization of the Feynman amplitude and the chain of integration in \S \ref{wondSec} below.
When convenient, we will choose coordinates so as to identify the
affine space $\A^{D|\BV_\Gamma|}(\C)\times \{ y \}\subset \sigma^{(Z,y)}$ with a real affine space 
$\A^{2D|\BV_\Gamma|}(\R)$.

\medskip

\subsubsection{Feynman amplitudes and complex Green forms}\label{weightGreenCsec}
A variant of the amplitude $\omega^{(Z)}_\Gamma$ of \eqref{amplitudeZ}, which we
will discuss briefly in \S \ref{BMsec}, is related to the complex Green forms. In this
setting, instead of \eqref{amplitudeZ} one considers the closely related form
\begin{equation}\label{amplitudeZ2}
\hat\omega_\Gamma = \prod_{e\in \BE_\Gamma} \frac{1}{\| x_{s(e)}- x_{t(e)} \|^{2D-2}} 
\bigwedge_{v \in\BV_\Gamma} \sum_{k=1}^D (-1)^{k-1} dx_{v,[k]} \wedge d\bar x_{v,[k]},
\end{equation}
where the forms $dx_{v,[k]}$ and $d\bar x_{v,[k]}$ denote 
$$ \begin{array}{l} 
dx_{v,[k]} = dx_{v,1}\wedge \cdots \wedge \widehat{dx_{v,k}} \wedge \cdots \wedge dx_{v,D},
\\[3mm]
d\bar x_{v,[k]} = d\bar x_{v,1}\wedge \cdots \wedge \widehat{d\bar x_{v,k}} \wedge \cdots \wedge d\bar x_{v,D}, \end{array} $$
respectively, with the factor $dx_{v,k}$ and $d\bar x_{v,k}$ removed.

\smallskip

Notice how, unlike the form considered in \eqref{amplitudeZ}, which is defined on
the affine $\A^{2D|\BV_\Gamma|}\subset Z^{\BV_\Gamma}$, the form \eqref{amplitudeZ2}
has the same degree of homogeneity $2D-2$ in the numerator and denominator, 
when the graph $\Gamma$ has no multiple edges, hence
it is invariant under rescaling of the coordinates by a common non-zero scalar factor.

\medskip
\subsubsection{Distributional interpretation}\label{distrSec}

In the cases discussed above, the amplitudes defined by
\eqref{amplitude1} and \eqref{amplitudeZ} can be 
given a distributional interpretation, as a pairing of 
a distribution
\begin{equation}\label{distribFey}
\prod_{e\in \BE_\Gamma} \frac{1}{\| x_{s(e)} - x_{t(e)} \|^\alpha},
\end{equation} 
where $\alpha$ is either $2\lambda=D-2$ or $2D-2$, or $2D$, and test
forms given in the various cases, respectively, by
\begin{enumerate}
\item forms $\phi(x_v) \, \bigwedge_{v\in \BV_\Gamma} dx_v$, with $\phi$ a
test function in $\cC^\infty(\A^{D|\BV_\Gamma|}(\R))$;
\item forms $\phi(z_v) \, \bigwedge_{v \in\BV_\Gamma} dx_v \wedge d\bar x_v$,
with $\phi(z_v)=\phi(p(z_v))=\phi(x_v)$ a test function in $\cC^\infty(X^{\BV_\Gamma}(\C))$;
\end{enumerate}

\medskip

\section{Feynman amplitudes and Bochner--Martinelli kernels}\label{BMsec}

We describe the relation between the Feynman amplitude $\omega_\Gamma$
of \eqref{amplitude1} and the Green functions of the Laplacian and we compute
$d\omega_\Gamma$ in terms of an integral kernel associated to the graph
$\Gamma$ and the affine space $\A^{D|\BV_\Gamma|}\subset X^{\BV_\Gamma}$.
We also discuss the relation between the Feynman amplitude $\omega^{(Z)}_\Gamma$ 
of \eqref{amplitudeZ}, in the complexified case, and the Bochner--Martinelli kernel.

\medskip
\subsection{Real Green functions}

The Green function for the real Laplacian on $\A^D(\R)$, with
$D=2\lambda+2$, is given by
\begin{equation}\label{realGreen}
G_\R(x,y) = \frac{1}{\| x-y \|^{2\lambda}}
\end{equation}

Consider then the differential form $\omega= G_\R(x,y) \, dx\wedge dy$. This corresponds
to the Feynman amplitude \eqref{amplitude1} in the case of the graph consisting
of a single edge, with configuration space $(X\times X) \smallsetminus \Delta$,
with $\Delta=\{(x,y)\,|\, x=y \}$ the diagonal. 

\begin{lem}\label{1eomega}
The form $\omega=G_\R(x,y) \, dx\wedge dy$
is not closed. Its differential is given by
\begin{equation}\label{1edomega}
d\omega = - \lambda \sum_{k=1}^D \frac{ (x_k-y_k) }{\| x-y \|^D} \, \left( dx\wedge dy\wedge
d \bar x_k - dx\wedge dy\wedge d\bar y_k \right). 
\end{equation}
\end{lem}

\proof We have
$$ d \omega =\bar\partial \omega = \sum_{k=1}^D \frac{\partial G_\R}{\partial \bar x_k} 
dx\wedge dy\wedge d\bar x_k + \sum_{k=1}^D \frac{\partial G_\R}{\partial \bar y_k} 
dx\wedge dy\wedge d\bar y_k . $$
We then see that
$$ \frac{\partial \| x-y \|^{-2\lambda}}{\partial \bar x_k} =-\lambda \frac{ (x_k-y_k) }{\| x-y \|^D}, \ \  \
\text{ and } \ \ \
 \frac{\partial \| x-y \|^{-2\lambda}}{\partial \bar y_k} = \lambda \frac{ (x_k-y_k) }{\| x-y \|^D}, $$
so that we obtain \eqref{1edomega}.
\endproof

\medskip
\subsection{Feynman amplitudes and integral kernels on graphs}

We first consider the form defining the Feynman amplitude $\omega_\Gamma$
of \eqref{amplitude1}. It is not a closed form. In fact, we compute here explicitly
its differential $d\omega_\Gamma$ in terms of some integral kernels associated
to graphs.

\smallskip

Recall that, given a graph $\Gamma$ and a vertex $v\in \BV_\Gamma$, 
the graph $\Gamma \smallsetminus \{ v \}$ has
\begin{equation}\label{Gminusv}
\BV_{\Gamma\smallsetminus \{ v \}}= \BV_\Gamma \smallsetminus \{ v \}, \ \ \ \text{ and } \ \ \
\BE_{\Gamma\smallsetminus \{ v \}} = \BE_\Gamma \smallsetminus \{ e\in \BE_\Gamma\,|\,
v\in \partial(e)\}, 
\end{equation}
that is, one removes a vertex along with its star of edges.

\smallskip

\begin{defn}\label{kernelGR}
Suppose given a graph $\Gamma$ and a vertex $v\in \BV_\Gamma$. 
We define the differential form
\begin{equation}\label{kappaGR}
\kappa_{\Gamma,v}^{\R}(x) = (-1)^{|\BV_\Gamma|}  \epsilon_v\,\,
\sum_{e:\, v\in \partial(e)} \epsilon_e \sum_{k=1}^D
\frac{(x_{s(e),k}-x_{t(e),k})}{\| x_{s(e)}-x_{t(e)} \|^D} \, dx_v \wedge d\bar x_{v,k}
\end{equation}
in the coordinates $x=(x_v \,|\, v\in \BV_\Gamma)$, where the sign
$\epsilon_e=\pm 1$ is positive or negative according to whether $v=s(e)$ or $v=t(e)$
and the sign $\epsilon_v=\pm 1$ is defined by
$$ \epsilon_v ( \bigwedge_{w\neq v} dx_w )\wedge
dx_v = \bigwedge_{v'\in \BV_\Gamma} dx_{v'}. $$
Given an oriented graph $\Gamma$, we then consider the integral
kernel 
\begin{equation}\label{realBMG}
\cK_{\R,\Gamma}(x)= \lambda \sum_{v\in \BV_\Gamma}  
\omega_{\Gamma\smallsetminus \{ v \}} \wedge \kappa_{\Gamma,v}^\R(x)  ,
\end{equation}
where $\omega_{\Gamma\smallsetminus \{ v \}}$ is the form \eqref{amplitude1} for
the graph \eqref{Gminusv}. 
\end{defn}

\smallskip

\begin{prop}\label{amp1kernel}
The differential $d \omega_\Gamma$ is the integral kernel
$\cK_{\R,\Gamma}$ of \eqref{realBMG}.
\end{prop}

\proof First observe that, for $\omega_\Gamma$ the Feynman amplitude 
of \eqref{amplitude1}, we have $d\omega_\Gamma = \bar\partial \omega_\Gamma$,
that is,
$$ d\omega_\Gamma = \sum_{v \in \BV_\Gamma} \sum_{k=1}^D (-1)^{\BV_\Gamma} 
\frac{\partial}{\partial \bar x_{v,k}} \omega_\Gamma \wedge d\bar x_{v,k}. $$
We have
$$ \partial_{\bar x_{v,k}} \left( \prod_e \frac{1}{\| x_{s(e)}-x_{t(e)} \|^{2\lambda}} \right) = $$
$$ \left(  \prod_{e: v\notin \partial e} \frac{1}{ \| x_{s(e)}-x_{t(e)} \|^{2\lambda}} \right) \cdot
\left( \sum_{e: v=t(e)} \bar\partial_{x_{v,k}} \frac{1}{\| x_{s(e)}-x_v \|^{2\lambda}} \right) $$
$$ - \left(  \prod_{e: v\notin \partial e} \frac{1}{ \| x_{s(e)}-x_{t(e)} \|^{2\lambda}} \right) \cdot
\left( \sum_{e: v=s(e)} \bar\partial_{x_{v,k}} \frac{1}{\| x_{t(e)}-x_v \|^{2\lambda}}  \right) $$
$$ = \lambda \cdot \left(  \prod_{e: v\notin \partial e} \frac{1}{ \| x_{s(e)}-x_{t(e)} \|^{2\lambda}} \right) \cdot $$
$$ \left( \sum_{e: v=s(e)} \frac{(x_{v,k}-x_{t(e),k})}{\| x_v-x_{t(e)} \|^D} -
\sum_{e: v=t(e)} \frac{(x_{s(e),k}-x_{v,k})}{\| x_{s(e)}-x_v \|^D} \right), $$
where we used the fact that, for $z,w\in \A^D$, one has
$$ \frac{\partial}{\partial \bar w_k} \frac{1}{\| z-w \|^{2\lambda}} =
- \frac{\lambda \, (w_k - z_k)}{\| z-w \|^D}. $$
We also introduce the notation
\begin{equation}\label{upsilonv}
\upsilon_v(x)= \prod_{e: v\notin \partial(e)} \frac{1}{\| x_{s(e)} - x_{t(e)} \|^{2\lambda}} ,
\end{equation}
so that we find
$$ d\omega_\Gamma =\lambda (-1)^{\BV_\Gamma}  \sum_{v \in \BV_\Gamma}  
\upsilon_v(x) \bigwedge_{w\neq v} dx_w \wedge  \kappa^\R_{\Gamma,v}(x) . $$
We then identify the term $\upsilon_v(x) \wedge_{w\neq v} dx_w$ with the form
$\omega_{\Gamma \smallsetminus \{ v \}}$.
\endproof

\smallskip

It is easy to see that this recovers \eqref{1edomega} in the case of the graph
consisting of two vertices and a single edge between them.

\medskip
\subsection{Complex Green functions and the Bochner--Martinelli kernel}

On $\A^D(\C) \subset X$ the complex Laplacian 
$$ \Delta = \sum_{k=1}^D \frac{\partial^2}{\partial x_k \partial \bar{x}_k} . $$
has a fundamental solution of the form ($D>1$)
\begin{equation}\label{Green}
G_\C(x,y) = \frac{- (D-2)!}{(2\pi i)^D \| x-y \|^{2D-2}} .
\end{equation}
The Bochner--Martinelli kernel is given by
\begin{equation}\label{BMker}
\cK_\C(x,y) = \frac{(D-1)!}{(2\pi i )^D} \sum_{k=1}^D(-1)^{k-1} \frac{\bar x_k - \bar y_k}{\| x-y \|^{2D}} 
d\bar{x}_{[k]} \wedge dx,
\end{equation}
where we write 
$$dx =dx_1\wedge \cdots \wedge dx_D \ \ \ \text{ and } \ \ \ 
d\bar{x}_{[k]}=d\bar x_1\wedge
\cdots \wedge \widehat{d\bar x_k} \wedge \cdots \wedge d\bar x_D, $$ 
with the $k$-th factor removed. 
The following facts are well known (see \cite{GrHa}, \S 3.2 and \cite{Kyt}):
\begin{equation}\label{BMandGreen}
\cK_\C (x,y)= \sum_{k=1}^D (-1)^{k-1} \frac{\partial G_\C}{\partial x_k} d\bar{x}_{[k]} \wedge dx
\end{equation}
$$ = (-1)^{D-1} \partial_x G_\C \wedge \sum_{k=1}^D d\bar{x}_{[k]} \wedge dx_{[k]}. $$
where $\partial_x$ and $\bar\partial_x$ denote the operators $\partial$ and $\bar\partial$ 
in the variables $x=(x_k)$.
For fixed $y$, the coefficients of $\cK_\C(x,y)$ are harmonic functions 
on $\A^D\smallsetminus \{ y \}$ and $\cK_\C(x,y)$ is closed, $d_x \cK(x,y)=0$.
Moreover, the Bochner--Martinelli integral formula holds: for a bounded domain $\Sigma$
with piecewise smooth boundary $\partial\Sigma$, a function $f\in \cC^2(\overline{\Sigma})$
and for $y\in \Sigma$,
\begin{equation}\label{BMint}
f(y)=\int_{\partial\Sigma} f(x) \cK_\C(x,y) + \int_\Sigma  \Delta(f)(x) \, G_\C(x,y) d\bar x \wedge dx  
- \int_{\partial\Sigma} G_\C(x,y) \, \mu_f(x),
\end{equation}
where $\mu_f(x)$ is the form
\begin{equation}\label{mufx}
 \mu_f(x) = \sum_{k=1}^D (-1)^{D+k-1} \frac{\partial f}{\partial \bar x_k}\, dx_{[k]} \wedge d\bar x.
\end{equation}
The integral \eqref{BMint} vanishes when $y\notin \Sigma$. A related Bochner--Martinelli
integral, which can be derived from \eqref{BMint} (see Thm 1.3 of \cite{Kyt}) is of the form
\begin{equation}\label{BMint2}
f(y) = \int_{\partial\Sigma} f(x)\, \cK_\C(x,y) - \int_\Sigma \bar\partial f \wedge \cK_\C(x,y) ,
\end{equation}
for $y\in \Sigma$ and $f\in \cC^1(\overline{\Sigma})$, with $\bar\partial f= \sum_k \partial_{\bar x_k} f\, d\bar x_k$.

Similarly, one can consider Green forms associated to the Laplacians on $\Omega^{p,q}$
forms and related Bochner--Martinelli kernels, see \cite{TarShla}.

\medskip
\subsection{Feynman amplitude and the Bochner--Martinelli kernel}

We now consider the Feynman amplitude $\hat\omega_\Gamma$ of
\eqref{amplitudeZ2} in the complexified case discussed in \S \ref{weightGreenCsec}.
We first introduce a Bochner--Martinelli kernel for graphs. 

\medskip
\subsubsection{Bochner--Martinelli kernel for graphs}\label{BMgraphs}

We define Bochner--Martinelli kernels for graphs in the following way.

\begin{defn}\label{BMgraphDef}
Suppose given an oriented graph $\Gamma$ and a 
vertex $v\in \BV_\Gamma$. We set 
\begin{equation}\label{KCGammav}
\kappa^\C_{\Gamma,v} =\sum_{e: v\in \partial(e)} \epsilon_e \, \sum_{k=1}^D 
(-1)^{k-1} \frac{(\bar x_{s(e),k}-\bar x_{t(e),k})}{\| x_{s(e)}-x_{t(e)} \|^{2D}} \, 
dx_v \wedge d\bar x_{v,[k]} 
\end{equation}
and
\begin{equation}\label{KCGammavstar}
\kappa^{\C,\ast}_{\Gamma,v} =\sum_{e: v\in \partial(e)} \epsilon_e \, \sum_{k=1}^D 
(-1)^{k-1} \frac{(x_{s(e),k}-x_{t(e),k})}{\| x_{s(e)}-x_{t(e)} \|^{2D}} \, 
dx_{v,[k]} \wedge d\bar x_v, 
\end{equation}
where the sign $\epsilon_e$ is $\pm 1$ depending on whether $v=s(e)$ or $v=t(e)$.
\end{defn}

\medskip
\subsubsection{Bochner--Martinelli integral on graphs}
There is an analog of the classical Bochner--Martinelli integral \eqref{BMint2} 
for the kernel \eqref{KCGammav} of graphs.

We first recall some well known facts about the Laplacian on graphs,
see e.g.~\cite{BeLo}. 
Given a graph $\Gamma$, one defines the exterior differential $\delta$ from
functions on $\BV_\Gamma$ to functions on $\BE_\Gamma$ by
$$ (\delta h)(e) = h(s(e)) - h(t(e)) $$
and the $\delta^*$ operator from functions on edges to functions on
vertices by
$$ (\delta^* \xi)(v) = \sum_{e: v=s(e)}  \xi(e) - \sum_{e: v=t(e)} \xi(e). $$
Thus, the Laplacian $\Delta_\Gamma =\delta^* \delta$ on $\Gamma$ is given by
$$ (\Delta_\Gamma f)(v) = \sum_{e: v=s(e)}  (h(v)-h(t(e)) - \sum_{e: v=t(e)} (h(s(e))- h(v)) $$
$$ = N_v \, h(v) - \sum_{e: v\in \partial(e)} h(v_e), $$
where $N_v$ is the number of vertices connected to $v$ by an edge, and 
$v_e$ is the other endpoint of $e$ (we assume as usual that $\Gamma$
has no looping edges).
Thus, a harmonic function $h$ on a graph is a function on $\BV_\Gamma$ satisfying
\begin{equation}\label{graphharmonic}
h(v) = \frac{1}{N_v} \sum_{e: v\in \partial(e)} h(v_e) .
\end{equation} 

\smallskip

Motivated by the usual notion of graph Laplacian $\Delta_\Gamma$ and
the harmonic condition \eqref{graphharmonic} for graphs recalled here above,
we introduce an operator
\begin{equation}\label{DeltaGv}
(\Delta_{\Gamma,v} f)(x) = \sum_{e: v\in \partial(e)} f(x_{v_e}),
\end{equation}
which assigns to a complex valued function $f$ defined on $\A^D\subset X$
a complex valued function $\Delta_{\Gamma,v} f$ defined on $X^{\BV_\Gamma}$.

\smallskip

We then have the following result.

\begin{prop}\label{BMintgraphProp}
Let $f$ be a complex valued function defined on $\A^{D}\subset X$.
Also suppose given a bounded domain $\Sigma$ with piecewise
smooth boundary $\partial \Sigma$ in $\A^D$ 
and assume that $f$ is $\cC^1$ on $\overline\Sigma$. 
For a given $v\in \BV_\Gamma$ consider the set of 
$x=(x_w)\in \A^{D|\BV_\Gamma|}\subset X^{\BV_\Gamma}$, 
such that $x_{v_e} \in \Sigma$ for all $v_e\neq v$ endpoints of
edges $e$ with $v\in \partial(e)$.
For such $x=(x_w)$ we have
\begin{equation}\label{BMintgraph}
(\Delta_{\Gamma,v} f)(x) = \frac{(D-1)!}{(2\pi i)^D\, N_v} \left( \int_{\partial \Sigma} f(x_v) \kappa_{\Gamma,v}^\C(x)
- \int_\Sigma \bar\partial_{x_v} f(x_v) \wedge \kappa_{\Gamma,v}^\C(x) \right),
\end{equation}
where the integration on $\Sigma$ and $\partial \Sigma$ is in the variable $x_v$ 
and $\Delta_{\Gamma,v} f$ is defined as in \eqref{DeltaGv}.
\end{prop}

\proof We have
$$  \int_{\partial \Sigma} f(x_v) \kappa_{\Gamma,v}^\C(x) = 
 \sum_{e: v\in \partial(e)} \epsilon_e \,  \int_{\partial \Sigma} f(x_v) \sum_{k=1}^D 
(-1)^{k-1} \frac{(\bar x_{s(e),k}-\bar x_{t(e),k})}{\| x_{s(e)}-x_{t(e)} \|^{2D}} \, \eta_v $$
where
$$ \eta_v = dx_v \wedge d\bar x_{v,[k]}. $$
We write the integral as
$$ \sum_{e: v=s(e)} \int_{\partial\Sigma} f(x_v) \sum_{k=1}^D (-1)^{k-1} 
\frac{(\bar x_{v,k}-\bar x_{v_e,k})}{\| x_{v}-x_{v_e} \|^{2D}} \, 
\eta_v $$
$$ - \sum_{e: v=t(e)} \int_{\partial\Sigma} f(x_v) \sum_{k=1}^D (-1)^{k-1} 
\frac{(\bar x_{v_e,k}-\bar x_{v,k})}{\| x_{v}-x_{v_e} \|^{2D}} \,  \eta_v $$
$$ = \sum_{e: v\in \partial(e)} \int_{\partial\Sigma} f(x_v) \sum_{k=1}^D (-1)^{k-1} 
\frac{(\bar x_{v,k}-\bar x_{v_e,k})}{\| x_{v}-x_{v_e} \|^{2D}} \, \eta_v. $$
The case of the integral on $\Sigma$ is analogous. We then apply the
classical result \eqref{BMint2} about the Bochner--Martinelli integral 
and we obtain
$$ \int_{\partial \Sigma} f(x) \kappa_{\Gamma,v}^\C(x) - \int_\Sigma \bar\partial_{x_v} f(x) \wedge \kappa_{\Gamma,v}^\C(x) = \frac{(2\pi i)^D}{(D-1)!} \sum_{e: v\in \partial(e)} f(x_{v_e}). $$
\endproof

\medskip
\subsubsection{Feynman amplitude and Bochner--Martinelli kernel}

The Bochner--Martinelli kernel of graphs defined above is related
to the Feynman amplitude \eqref{amplitudeZ2} by the following.

\begin{prop}\label{amplitudeZ2prop}
Let $\hat\omega_\Gamma$ be the Feynman amplitude \eqref{amplitudeZ2}.
Then
\begin{equation}\label{delamplitudeZ2}
\partial \hat\omega_\Gamma= \sum_{v\in \BV_\Gamma} \epsilon_v\, \,
\hat\omega_{\Gamma\smallsetminus\{ v\}} \wedge \kappa^\C_{\Gamma,v}
\end{equation}
\begin{equation}\label{delabarmplitudeZ2}
\bar\partial \hat\omega_\Gamma= \sum_{v\in \BV_\Gamma} \epsilon_v\, \,
\hat\omega_{\Gamma\smallsetminus\{ v\}} \wedge (-1)^{D-1} \kappa^{\C,\ast}_{\Gamma,v},
\end{equation}
where the sign $\epsilon_v$ is defined by
$$ \epsilon_v \,\, \left( \bigwedge_{w\neq v} \sum_k (-1)^{k-1} dx_{w,[k]} \wedge d\bar x_{w,[k]} \right)\wedge \left( \sum_k (-1)^{k-1} dx_{v,[k]} \wedge d\bar x_{v,[k]} \right) = $$
$$  \bigwedge_{v'\in\BV_\Gamma} \sum_k (-1)^{k-1} dx_{v',[k]} \wedge d\bar x_{v',[k]} . $$
\end{prop}

\proof The argument is analogous to Proposition \ref{amp1kernel}. 
\endproof

%%%%%%%%%%%%

\bigskip
\section{Integration over the real locus}\label{GegenSec}

In this section we consider the Feynman amplitudes $\omega_\Gamma$
defined in \eqref{amplitude1} and the domain $\sigma_\Gamma$ defined
in \eqref{sigmaGamma}. We give an explicit formulation of the integral
in terms of an expansion of the real Green functions $\| x_{s(e)}-x_{t(e)} \|^{-2\lambda}$
in Gegenbauer polynomials, based on a technique well known to physicists
(the $x$-space method, see \cite{CheKaTka}). We consider the integrand restricted
to the real locus $X(\R)^{\BV_\Gamma}$, and express it in polar coordinates, 
separating out an angular integral and a radial integral. We identify 
a natural subdivision of the domain of integration into chains that are
indexed by acyclic orientations of the graph. In the special case of dimension
$D=4$, we express the integrand in terms closely related to  
multiple polylogarithm functions.

Spencer Bloch recently introduced a construction of cycles
in the relative homology $H_*(X^{\BV_\Gamma}, \cup_e \Delta_e)$
of the graph configuration spaces,
that explicitly yield multiple zeta values as periods \cite{Bloch}.
We use here a variant of his construction, which will
have a natural interpretation in terms of the $x$-space method for
the computation of the Feynman amplitudes in configuration spaces.

\smallskip

\subsubsection{Directed acyclic graph structures}

\begin{defn}\label{orientations}
Let $\Gamma$ be a finite graph without looping edges.
Let $\Omega(\Gamma)$ denote the set of edge orientations
on $\Gamma$ such that the resulting directed graph is a directed
acyclic graph. 
\end{defn}

It is well known that all finite graphs without looping edges admit such
orientations. In fact, the number
of possible orientations that give it the structure of a directed acyclic 
graph are given by $(-1)^{\BV_\Gamma} P_\Gamma(-1)$, where
$P_\Gamma(t)$ is the chromatic polynomial of the graph $\Gamma$,
see \cite{Stan}.

\smallskip

The following facts about directed acyclic graphs are also well
known, and we recall them here for later use.

\begin{itemize}

\item Each orientation ${\bf o}\in \Omega(\Gamma)$ determines
a partial ordering on the vertices of the graph $\Gamma$,
by setting $w\geq_{\bf o} v$ whenever there is an oriented
path of edges from $v$ to $w$ in the directed graph $(\Gamma,{\bf o})$. 

\smallskip

\item In every directed acyclic graph there is at least
a vertex with no incoming edges and at least a vertex with no outgoing edges. 

\end{itemize}

\medskip
\subsubsection{Relative cycles from directed acyclic structures}

Given a graph $\Gamma$ we consider the space $X^{\BV_\Gamma}$.
On the dense subset $\A^D(\R) =X(\R) \smallsetminus \Delta_\infty(\R)$ of
the chain of integration $\sigma_\Gamma$, we use polar coordinates
with $x_v =r_v \omega_v$, with $r_v \in \R_+$ and $\omega_v \in S^{D-1}$.

\smallskip

\begin{defn}\label{BlochSigma}
Let ${\bf o} \in \Omega(\Gamma)$ be an acyclic orientation. 
Consider the chain
\begin{equation}\label{SigmaO}
 \Sigma_{{\bf o}} :=\{ (x_v) \in X^{\BV_\Gamma}(\R) \,:\, 
  r_w \geq r_v \text{ whenever } w\geq_{{\bf o}} v \},
\end{equation}
with boundary $\partial \Sigma_{{\bf o}}$ contained in $\cup_{e\in \BE_\Gamma} \Delta_e$.  
It defines a middle dimensional relative homology class 
$$ [\Sigma_{{\bf o}}]\in H_{|\BV_\Gamma|}(X^{\BV_\Gamma}, \cup_{e\in \BE_\Gamma} \Delta_e). $$
\end{defn}

\smallskip

The following simple observation will be useful in the Feynman integral calculation
we describe later in this section.

\begin{lem}\label{SigmaSD}
Let ${\bf o} \in \Omega(\Gamma)$ be an acyclic orientation and $\Sigma_{{\bf o}}$
the chain defined in \eqref{SigmaO}. 
Then $\Sigma_{{\bf o}} \smallsetminus \cup_v \{ r_v=0 \}$ 
is a bundle with fiber $(S^{D-1})^{\BV_\Gamma}$
over a base 
\begin{equation}\label{barSigma}
\overline{\Sigma}_{\bf o} =\{ (r_v)\in (\R_+^*)^{\BV_\Gamma} \,:\,  r_w \geq r_v \text{ whenever } w\geq_{{\bf o}} v \}.
\end{equation}
\end{lem}

\proof This is immediate from the polar coordinate form $x_v = r_v \omega_v$, with
$r_v \in \R^*_+$ and $\omega_v \in S^{D-1}$.
\endproof

\medskip

\subsection{Gegenbauer polynomials and angular integrals}

One of the techniques developed by physicists to compute
Feynman amplitudes, by passing from momentum to configuration space, 
relies on the expansion in Gegenbauer polynomials, see for instance 
\cite{CheKaTka} and the recent \cite{NiStoTo}.

\smallskip

The Gegenbauer polynomials (or ultraspherical polynomials) are
defined through the generating function 
\begin{equation}\label{Geg1}
 \frac{1}{(1-2tx + t^2)^\lambda} = \sum_{n=0}^\infty C_n^{(\lambda)}(x) t^n, 
\end{equation} 
for $|t|<1$. For $\lambda > -1/2$, they satisfy
\begin{equation}\label{Geg2}
 \int_{-1}^1 C_n^{(\lambda)}(x) C_m^{(\lambda)}(x) \, (1-x^2)^{\lambda -1/2}  dx =\delta_{n,m}
\frac{\pi 2^{1-2\lambda} \Gamma(n+2\lambda)}{n! (n+\lambda) \Gamma(\lambda)^2}. 
\end{equation}

\smallskip

We use what is known in the physics literature as the $x$-space method
(see \cite{CheKaTka}) to reformulate the integration involved in the
Feynman amplitude calculation in a way that involves the relative
chains of Definition \ref{BlochSigma}.

\medskip

\begin{thm}\label{IntGeg}
In even dimension $D=2\lambda+2$, 
the integral $\int_{\sigma_\Gamma} \omega_\Gamma$ 
of the form \eqref{amplitude1} on the chain $\sigma_\Gamma$ 
can be rewritten in the form
\begin{equation}\label{amplitudeGeg}
\sum_{{\bf o}\in \Omega(\Gamma)} m_{{\bf o}} \int_{\Sigma_{\bf o}} 
\prod_{e\in \BE_\Gamma} r_{t_{{\bf o}}(e)}^{-2\lambda} \left(\sum_n (\frac{r_{s_{\bf o}(e)}}{r_{t_{\bf o}(e)}})^n  C^{(\lambda)}_n(\omega_{s_{\bf o}(e)}\cdot \omega_{t_{\bf o}(e)}) \right) \, dV ,
\end{equation}
for some positive integers $m_{{\bf o}}$, and with volume element 
$dV=\prod_v d^Dx_v = \prod_v  r_v^{D-1} dr_v\, d\omega_v$.
\end{thm}

\medskip

\proof We write the integral in polar coordinates, with 
$$d\omega=\sin^{D-2}(\phi_1)\sin^{D-3}(\phi_2)\cdots \sin(\phi_{D-2}) 
d\phi_1\cdots d\phi_{D-1}$$ the volume element on the sphere $S^{D-1}$
and $d^Dx_v =r_v^{D-1} dr_v\, d\omega_v$.

\smallskip

In dimension $D=2\lambda +2$, by \eqref{Geg1}, the Newton 
potential has an expansion in Gegenbauer polynomials, so that
\begin{equation}\label{Geg3}
\begin{array}{c}
\displaystyle{ \frac{1}{\| x_{s(e)}-x_{t(e)} \|^{2\lambda}} = \frac{1}{\rho_e^{2\lambda} (1+ (\frac{r_e}{\rho_e})^2 
- 2\frac{r_e}{\rho_e}  \omega_{s(e)}\cdot \omega_{t(e)})^{\lambda}} } \\[4mm]
\displaystyle{ = \rho_e^{-2\lambda} \sum_{n=0}^\infty (\frac{r_e}{\rho_e})^n 
C_n^{(\lambda)}(\omega_{s(e)}\cdot \omega_{t(e)}), }
\end{array}
\end{equation}
where $\rho_e=\max\{ \|x_{s(e)} \|, \| x_{t(e)} \| \}$ and $r_e =\min\{ \| x_{s(e)} \|, \| x_{t(e)} \| \}$
and with $\omega_v\in S^{D-1}$. 

\smallskip

We can subdivide the integration into open sectors where, for each edge, either
$r_{s(e)} < r_{t(e)}$ or the converse holds, so that each term $\rho_e^{-2\lambda}$
is $r_{t(e)}^{-2\lambda}$ (or $r_{s(e)}^{-2\lambda}$) and each term $(r_e/\rho_e)^n$ is
$(r_{s(e)}/r_{t(e)})^n$ (or its reciprocal). In other words, let ${\bf b}$ denote an
assignment of either $r_{s(e)} < r_{t(e)}$ or $r_{s(e)} > r_{t(e)}$ at each
edge, which we write simply as ${\bf b}(r_{s(e)}, r_{t(e)})$, and let 
$$ \bar\cR_{\bf b}=\{ (r_v)\in (\R^*_+)^{{\bf V}_\Gamma} \,|\, {\bf b}(r_{s(e)}, r_{t(e)})\, 
\text{ for } e\in {\bf E}_\Gamma \} $$
and $\cR_{\bf b} = \bar\cR_{\bf b}\times (S^{D-1})^{{\bf V}_\Gamma}$.
Then we identify the domain of integration with
$\cup_{\bf b} \cR_{\bf b}$, up to a set of measure zero. The set $\cR_{\bf b}$
is empty unless the assignment ${\bf b}$ defines a strict partial ordering of the
vertices of $\Gamma$, in which
case ${\bf b}$ determines an acyclic orientation ${\bf o}={\bf o}({\bf b})$ of $\Gamma$,
as described in Definition \ref{orientations}. In this case, then, the chain of integration corresponding to the sector $\cR_{\bf b}$ is the chain $\Sigma_{{\bf o}}$ 
of Definition \ref{BlochSigma}, with $\rho_e=r_{t_{\bf o}(e)}$
and $r_e=r_{s_{\bf o}(e)}$. Thus, the domain of integration 
can be identified with 
$\cup_{{\bf o}\in \Omega(\Gamma)} m_{{\bf o}} \, \Sigma_{{\bf o}}$,
where $m_{{\bf o}}$ is a multiplicity, taking into account the fact that 
different strict partial orderings may define the same acyclic orientation. 
\endproof

The integral  \eqref{amplitudeGeg} can be approached by first considering
a family of angular integrals  
\begin{equation}\label{angint}
\cA_{(n_e)_{e\in \BE_\Gamma}} = \int_{(S^{D-1})^{\BV_\Gamma}} \prod_e C_{n_e}^{(\lambda)}(\omega_{s(e)}\cdot \omega_{t(e)}) \prod_v d\omega_v ,
\end{equation}
labelled by all choices of integers $n_e$ for $e\in \BE_\Gamma$. The evaluation
of these angular integrals will lead to an expression $\cA_{n_e}$ in the $n_e$, so that
one obtains a radial integral
\begin{equation}\label{radAint}
\sum_{{\bf o}\in \Omega(\Gamma)} m_{{\bf o}} \int_{\bar\Sigma_{\bf o}} 
\prod_{e\in \BE_\Gamma} \cF(r_{s_{\bf o}(e)}, r_{t_{\bf o}(e)})\,\, \prod_v  r_v^{D-1} dr_v
\end{equation}
where
\begin{equation}\label{FAre}
\cF(r_{s_{\bf o}(e)}, r_{t_{\bf o}(e)}) = r_{t_{{\bf o}}(e)}^{-2\lambda} 
\sum_{n_e}  \cA_{n_e}\, (\frac{r_{s_{\bf o}(e)}}{r_{t_{\bf o}(e)}})^{n_e} .
\end{equation}

\medskip
\subsection{Polygons and polylogarithms}

We first discuss the very simple example of a polygon graph, where one sees
polylogarithms and zeta values arising in the expression \eqref{FAre}
and its integration on the domains $\bar \Sigma_{\bf o}$. In the following subsections
we will analyze the more general structure of these integrals for more complicated
graphs.

\medskip
\subsubsection{The angular integral for polygons in arbitrary dimension}
The angular integral for polygon graphs has the following explicit expression.

\begin{prop}\label{angpolyGamma}
Let $\Gamma$ be a polygon with $k$ edges. Then the angular integral \eqref{angint}
depends on a single variable $n\in \N$ and is given by
\begin{equation}\label{AintP}
\cA_n = \left(\frac{\lambda 2 \pi^{\lambda+1}}{\Gamma(\lambda +1)(n +\lambda)}\right)^k 
\cdot \dim \cH_n(S^{2\lambda+1}),
\end{equation}
and $\cH_n(S^{2\lambda+1})$ is the space of harmonic functions of degree $n$ on the sphere $S^{2\lambda+1}$.
\end{prop}

\proof 
The angular integral, in this case, is simply given by
$$ \int_{(S^{D-1})^{\BV_\Gamma}} C_{n_1}^{(\lambda)}(\omega_{v_k}\cdot\omega_{v_1})
C_{n_2}^{(\lambda)}(\omega_{v_1}\cdot\omega_{v_2})\cdots C_{n_k}^{(\lambda)}(\omega_{v_{k-1}}\cdot\omega_{v_k}) \prod_{v\in \BV_\Gamma} d\omega_v , $$
which is independent of the orientation.
We then use the fact that the Gegenbauer polynomials satisfy (\cite{BaEr} Vol.2, Lemma 4, \S 11.4)
\begin{equation}\label{GegS}
\int_{S^{D-1}} C_m^{(\lambda)}(\omega_1 \cdot \omega) C_n^{(\lambda)}(\omega\cdot \omega_2) d\omega = \delta_{n,m} \, \frac{\lambda\, Vol(S^{D-1})}{n+\lambda} C_n^{(\lambda)}(\omega_1\cdot \omega_2),
\end{equation}
with $Vol(S^{D-1})=2 \pi^{\lambda+1}/\Gamma(\lambda+1)$.  Thus, we obtain
\begin{equation}\label{Agamma}
 \cA_n =   \left(\frac{\lambda 2 \pi^{\lambda+1}}{\Gamma(\lambda +1)}\right)^{k-1} \frac{1}{(n+\lambda)^{k-1}}   \int_{S^{D-1}} C_n^{(\lambda)}(\omega \cdot \omega) \, d\omega,
\end{equation} 
where $n=n_1=\cdots=n_k$ and where the remaining integral is just
$$ \int_{S^{D-1}} C^{(\lambda)}_n (\omega\cdot \omega) d\omega =
C^{(\lambda)}_n (1)\, Vol(S^{D-1}). $$
The value of $C_n^{(\lambda)}(1)$ can be seen using 
the fact that the Gegenbauer polynomials are related to 
the {\em zonal spherical harmonics} (see \cite{Stein}, \S 4, and \cite{Gallier}, 
\cite{Mori}, \cite{Vilenkin})
$Z_{\omega_1}^{(n)}(\omega_2)$ by
\begin{equation}\label{zonalS}
C^{(\lambda)}_n(\omega_1\cdot \omega_2) = c_{D,n}\, Z_{\omega_1}^{(n)}(\omega_2),
\end{equation}
for $D=2\lambda+2$, with $\omega_1,\omega_2\in S^{D-1}$, 
where the coefficient $c_{D,n}$ is given by
\begin{equation}\label{cDn}
 c_{D,n} = \frac{Vol(S^{D-1})\, (D-2)}{2n +D-2}.
\end{equation} 
In turn, the zonal spherical harmonics are expressed in terms of 
an orthonormal basis $\{ Y_j \}$ of the Hilbert space $\cH_n(S^{D-1})$ 
of spherical harmonics on $S^{D-1}$ of degree $n$, as
\begin{equation}\label{zonalS2}
 Z_{\omega_1}^{(n)}(\omega_2) = \sum_{j=1}^{\dim \cH_n(S^{D-1})} Y_j(\omega_1) \overline{Y_j(\omega_2)}. 
\end{equation}
The dimension of the space $\cH_n(S^{D-1})$ of spherical harmonics is given by
$$ \dim \cH_n(S^{D-1}) = \binom{D-1+n}{n} - \binom{D-3+n}{n-2}. $$
Using \eqref{zonalS2}, we then have
$$ \int_{S^{D-1}} C^{(\lambda)}_n (\omega\cdot \omega) d\omega = c_{D,n}
\int_{S^{D-1}} Z_{\omega}^{(n)}(\omega) d\omega  $$
$$ = c_{D,n}  \sum_{j=1}^{\dim\cH_n(S^{D-1})} \int_{S^{D-1}} | Y_j(\omega) |^2 d\omega
= c_{D,n} \dim\cH_n(S^{D-1}), $$
which gives $C_n^{(\lambda)}(1)= \frac{2\lambda \dim\cH_n(S^{D-1})}{2(n +\lambda)}$.
\endproof

\subsubsection{Polygon amplitudes in dimension four}

We now specialize to the case where $D=4$ and $\lambda=1$ and we
show how one obtains integrals of polylogarithm functions
$$ {\rm Li}_s(z) = \sum_{n=1}^\infty \frac{z^n}{n^s}. $$

\begin{prop}\label{intPolygsprop}
Let $\Gamma$ be a polygon with $k$ edges and let $D=4$. Then the Feynman
amplitude is given by the integral
\begin{equation}\label{intPolygs}
(2 \pi^2)^k 
\sum_{\bf o} m_{\bf o} \int_{\bar\Sigma_{\bf o}}   
{\rm Li}_{k-2}(\prod_i \frac{r^2_{w_i}}{r^2_{v_i}})
\, \prod_v r_v\, dr_v,
\end{equation}
where the vertices $v_i$ and $w_i$ are the sources and tails of the
oriented paths determined by ${\bf o}$.
\end{prop}

\proof We write the terms in the integrand  \eqref{amplitudeGeg} as
\begin{equation}\label{expandintegrand}
 \prod_{e\in \BE_\Gamma} \rho_e^{-2\lambda} \left(\sum_n (\frac{r_e}{\rho_e})^n  \,\,
C^{(\lambda)}_n(\omega_{s(e)}\cdot \omega_{t(e)}) \right) = 
\end{equation}
$$ (\rho_1\cdots \rho_k)^{-2\lambda} \sum_{n_1, \ldots, n_k} (\frac{r_1}{\rho_1})^{n_1}\cdots (\frac{r_k}{\rho_N})^{n_k} \,\,  C^{(\lambda)}_{n_1}(\omega_{s(e_1)}\cdot \omega_{t(e_1)})\cdots 
C_{n_k} ^{(\lambda)}(\omega_{s(e_k)}\cdot \omega_{t(e_k)}) $$
and we perform the angular integral as in \eqref{AintP}.
In the case $D=4$ we have
$$ \dim \cH_n(S^3)= \binom{n+3}{n} - \binom{n+1}{n-2} = (n+1)^2 . $$
Thus, the angular integral of Proposition \ref{angpolyGamma} becomes
\begin{equation}\label{angularD4polyg}
\cA_n = \frac{(2 \pi^2)^k}{(n +1)^{k-2}} .
\end{equation}
We then write the radial integrand as in \eqref{radAint}. An acyclic
orientation ${\bf o}\in \Omega(\Gamma)$, subdivides the polygon
$\Gamma$ into oriented paths $\gamma_i$ such that $s_{\gamma_i}=s_{\gamma_{i-1}}$
and $t_{\gamma_i}=t_{\gamma_{i+1}}$ or $s_{\gamma_i}=s_{\gamma_{i+1}}$
and $t_{\gamma_i}=t_{\gamma_{i-1}}$. Correspondingly, the set of vertices
is subdivided into $\BV_\Gamma =\{ v_i \} \cup \{ w_i \} \cup \{ v \notin \{ v_i,w_i\} \}$
with $v_i$ the sources and $w_i$ the tails of the oriented paths and the remaining
vertices partitioned into internal vertices of each oriented path. We then
have
$$ \cF(r_{s(e)}, r_{t(e)}) = (2 \pi^2)^k (\prod_e r_{t(e)}^{-2})\, \sum_{n\geq 0}  \frac{1}{(n +1)^{k-2}}
(\prod_i \frac{r^2_{v_i}}{r^2_{w_i}})^n $$
$$ = (2 \pi^2)^k (\prod_e r_{t(e)}^{-2}) (\prod_i \frac{r^2_{w_i}}{r^2_{v_i}}) 
\sum_{n\geq 1}  \frac{1}{n^{k-2}}
(\prod_i \frac{r^2_{v_i}}{r^2_{w_i}})^n. $$
Since each $w_i$ is counted twice as target of an edge and the
internal vertices of the oriented paths are counted only once,
we obtain
\begin{equation}\label{radPolygPolylog}
\prod_e \cF(r_{s(e)}, r_{t(e)}) \prod_v r_v^3 dv
=(2 \pi^2)^k  \, \cdot {\rm Li}_{k-2}(\prod_i \frac{r^2_{w_i}}{r^2_{v_i}})
\, \prod_v r_v dr_v.
\end{equation}
\endproof
 
\medskip
\subsubsection{Zeta values} 
After a cutoff regularization, these integrals produce combinations 
of zeta values with coefficients that are rationals combinations of
powers of $2\pi i$.
To see how this happens, we look explicitly at the contribution of an
acyclic orientations of the polygon consisting of just two oriented paths 
$\gamma_1$ and $\gamma_2$ with source $v$ and target  $w$,
respectively with $k_1$ and $k_2$ internal vertices. The other 
summands can be handled similarly. By changing variables
to $t= r^2_v/r^2_w$, $t_i = r^2_{v_i}/r^2_w$ for $v_i$ the internal edges of $\gamma_1$
and $s_i= r^2_{v_i}/r^2_w$ for $v_i$ the internal edges of $\gamma_2$, we obtain
$$ \bigwedge_{v\in \BV_\gamma} r_v\, dr_v = \pm 2^{1-k} r_w^{2k+1} dr_w\, dt\,
\bigwedge_i dt_i \wedge ds_i . $$
After factoring out a divergence along $\Delta_\infty =X \smallsetminus \A^4$, coming
from the integration of the $r_w$ term, which gives a pole along the divisor $\Delta_\infty$, 
one obtains an integral of the form
$$ 2 \pi^{2k} \int_{\bar\Sigma_1\cap \bar\Sigma_2}  {\rm Li}_{k-2}(t) \,\, dt\,  \prod_i dt_i ds_i, $$
where $\bar\Sigma_{\bf o}= \bar\Sigma_1\cap \bar\Sigma_2$ with
$\bar\Sigma_1=\{ (t,t_i,s_i) \,|\, t\leq t_1 \leq \cdots \leq t_{k_1-1}\leq 1\}$ and 
$\bar\Sigma_2=\{ (t,t_i,s_i) \,|\, t\leq s_1 \leq \cdots \leq s_{k_2-1}\leq 1\}$. One can 
use the relation \cite{Freitas}
\begin{equation}\label{Freitas}
 \int x^m {\rm Li}_n(x) \, dx = \frac{1}{m+1} x^{m+1} {\rm Li}_n(x) - \frac{1}{m+1} \int x^m {\rm Li}_{n-1}(x) dx,  
\end{equation} 
to reduce the integral to a combination of zeta values.

\medskip
\subsection{Stars of vertices and isoscalars}

To see how more complicated expressions can arise in the integrands, which
eventually lead to the presence of multiple zeta values, it is convenient to
regard graphs as being built out of stars of vertices pasted together by 
suitably matching the half edges, where by the {\em star of a vertex} we mean 
a single vertex $v$ of valence $k$ with $k$ half-edges $e_j$ attached to it. 
One can then built the Feynman integral by first identifying the contribution 
of the star of a vertex, which is obtained by integrating in the variables $r_v$ and 
$\omega_v$ of the central vertex $v$, and results in a function of variables 
$r_{v_j}$ and $\omega_{v_j}$, for each of the edges $e_j$. One then obtains
the integral for the graph, which gives a number (possibly after a regularization),
by matching the half edges and identifying the corresponding variables and
integrating over them.

\smallskip

We first introduce the analog of the angular integral \eqref{angint} for the
case of a graph with half-edges. While in the case of a usual graph, where
all half-edges are paired to form edges, the angular integral \eqref{angint}
is a number, in the case with open (unpaired) half-edges it is a function
of the variables of the half edges, which we denote by 
$\cA_{\underline{n}}(\underline{\omega})$, where $\underline{n}=(n_1,\ldots, n_\ell)$
and $\underline{\omega}=(\omega_1,\ldots,\omega_\ell)$ are vectors of integers $n_j\in \N$,
for each half-edge $e_j$, and of variables $\omega_j\in S^{D-1}$. We will sometime denote
these variables by $\omega_{v_j}$, where $v_j$ simply denotes the end of the half-edge $e_j$.
We will equivalently use the notation $\cA_{\underline{n}}(\underline{\omega})$ 
or $\cA_{(n_j)}(\omega_{v_j})$.
In the case of the star of a vertex, the angular integral is of the form
\begin{equation}\label{angintHalfstar}
\cA_{\underline{n}}(\underline{\omega})= \int_{S^{D-1}} \prod_j C_{n_j}^{(\lambda)} (\omega_j
\cdot \omega)\, d\omega, \text{ with } \underline{n}=(n_j)_{e_j \in \BE_\Gamma}, \,
\underline{\omega}=(\omega_j)_{e_j \in \BE_\Gamma}.
\end{equation}

\begin{lem}\label{angIntStar}
Let $\Gamma$ be the star of a valence $v$ vertex. Then the angular
integral \eqref{angintHalfstar} is given by the function
$$ \cA_{(n_j)} (\omega_{v_j}) = c_{D,n_1}\cdots c_{D,n_k} \,\, \tilde\cA_{(n_j)} (\omega_{v_j}), $$
\begin{equation}\label{AintStar}
\tilde\cA_{(n_j)} (\omega_{v_j}) = \sum_{\ell_1,\ldots, \ell_k} \overline{ Y_{\ell_1}^{(n_1)}(\omega_1)\cdots
Y_{\ell_k}^{(n_k)} (\omega_k)} \int_{S^{D-1}} Y_{\ell_1}^{(n_1)}(\omega)\cdots
Y_{\ell_k}^{(n_k)} (\omega) \,\, d\omega,
\end{equation}
where $\{ Y^{(n)}_\ell \}_{\ell =1, \ldots, d_n}$ is an orthonormal basis of the space
$\cH_n(S^{D-1})$ of spherical harmonics of degree $n$, and $d_n=\dim \cH_n(S^{D-1})$,
with the coefficients $c_{D,n}$ as in \eqref{cDn}.
\end{lem}

\proof Using the relation \eqref{zonalS}, \eqref{zonalS2} between the Gegenbauer
polynomials and the spherical harmonics, we rewrite the angular integral 
\eqref{angintHalfstar} in the form \eqref{AintStar}.
\endproof

Thus, the evaluation of the angular integrals \eqref{AintStar} for stars of vertices 
relates to the well known problem of evaluating coupling coefficients for
spherical harmonics,
\begin{equation}\label{coupleCoeffs}
\langle Y_{\ell_1}^{(n_1)}, \ldots, Y_{\ell_k}^{(n_k)} \rangle_D := \int_{S^{D-1}} Y_{\ell_1}^{(n_1)}(\omega)\cdots Y_{\ell_k}^{(n_k)} (\omega) \,\, d\omega .
\end{equation}

\smallskip

In the following we will be using the standard labeling of the basis $\{ Y^{(n)}_\ell \}$ where,
for fixed $n$, the indices $\ell$ run over  a set of $(D-2)$-tuples 
$$ (m_{D-2}, m_{D-1}, \ldots, m_2, m_1) \ \ \text{ with } \ \ \ 
n \geq m_{D-2}\geq \cdots \geq m_2 \geq |m_1|. $$
The spherical harmonics $ Y^{(n)}_\ell$ have the symmetry 
\begin{equation}\label{starYnl}
\overline{Y^{(n)}_\ell}=(-1)^{m_1} Y^{(n)}_{\bar \ell}, 
\end{equation}
where, for $\ell=(m_{D-2}, m_{D-1},
\ldots, m_2, m_1)$, one has $$\bar\ell:=(m_{D-2}, m_{D-1},
\ldots, m_2, -m_1).$$

In the simplest case of a tri-valent vertex, these coefficients are also referred to as the
Gaunt coefficients, and have been extensively studied, see for instance \cite{Alis},
\cite{Junk}, \cite{Sebil}. The Gaunt coefficients arising from the integration
of three harmonic functions determine the coefficients of the expansion formula
\begin{equation}\label{Hexpand}
 Y_{\ell_1}^{(n_1)} Y_{\ell_2}^{(n_2)} =\sum_{n,\ell} \cK_{D,n_i,n,\ell_i,\ell} \,\, Y_\ell^{(n)} 
\end{equation} 
that expresses the product of two harmonic functions in terms of a linear combination
of other harmonic functions, with the cases where some of the factors are conjugated 
taken care of by the symmetry \eqref{starYnl}. In the more general case \eqref{coupleCoeffs}
one can therefore repeatedly apply \eqref{Hexpand}, hence we focus here on the example
of the star of a tri-valent vertex.

\medskip

The Gaunt coefficients $\langle  Y_{\ell_1}^{(n_1)}, Y_{\ell_2}^{(n_2)} Y_{\ell_3}^{(n_3)} \rangle_D$
can be computed via Racah's factorization lemma (\cite{Alis}, \cite{Junk}) 
in terms of {\em isoscalar factors} and the Gaunt coefficients for $D-1$, according to
\begin{equation}\label{isoscal}
\langle  Y_{\ell_1}^{(n_1)}, Y_{\ell_2}^{(n_2)}, Y_{\ell_3}^{(n_3)} \rangle_D =\left(\begin{array}{ccc}
n_1 & n_2 & n_3 \\ n_1' & n_2' & n_3' 
\end{array}\right)_{D:D-1} \,\,\, \langle  Y_{\ell_1'}^{(n_1')}, Y_{\ell_2'}^{(n_2')}, Y_{\ell_3'}^{(n_3')} \rangle_{D-1},
\end{equation}
where $\ell_i=(n_i',\ell_i')$ with $n_i'=m_{D-2,i}$ and 
$\ell_i'=(m_{D-3,i},\ldots, m_{1,i})$.  An explicit expression of the isoscalar factors
\begin{equation}\label{isofactor}
\left(\begin{array}{ccc}
n_1 & n_2 & n_3 \\ n_1' & n_2' & n_3' 
\end{array}\right)_{D:D-1} 
\end{equation}
is given in \cite{Alis}, \cite{Junk}. We will discuss this more in detail 
in \S \ref{gluestarSec} and \S \ref{multpolylogsSec} below.

\medskip
\subsection{Gluing two stars along an edge}\label{gluestarSec}

We now consider the effect of patching together two trivalent stars by gluing 
two half edges with matching orientations. 

\begin{lem}\label{ang2starsLem}
Let $\cA_{n,n_1,n_2}(\omega,\omega_1, \omega_2)$ and $\cA_{n',n_3,n_4}(\omega', \omega_3,\omega_4)$ be the angular integrals associated to two trivalent stars, as in Lemma
\ref{angIntStar}. Then the angular integral of the graph obtained by joining the two stars at an
edge is 
\begin{equation}\label{ang2starsjoin}
\cA_{(n_i)_{i=1,\ldots,4}}((\omega_i)_{i=1,\ldots,4})= \sum_{\ell_i} \prod_{i=1}^4
c_{D,n_i} \, \overline{Y_{\ell_i}^{(n_i)}(\omega_i)} \, \cK_{n_i, \ell_i}(n),
\end{equation}
with coefficients $\cK_{\underline{n}, \underline{\ell}}(n)$ given by
\begin{equation}\label{ang2starscoeff}
\cK_{n_i, \ell_i}(n) = c_{D,n}^2 \sum_{\ell=1}^{d_n} 
\langle  Y_{\ell}^{(n)}, Y_{\ell_1}^{(n_1)}, Y_{\ell_2}^{(n_2)} \rangle_D \cdot
\langle  Y_{\ell}^{(n)}, Y_{\ell_3}^{(n_3)}, Y_{\ell_4}^{(n_4)} \rangle_D.
\end{equation}
\end{lem}

\proof The angular integral $\cA_{(n_i)_{i=1,\ldots,4}}((\omega_i)_{i=1,\ldots,4})$ is obtained
by integrating along the variables of the matched half edges,
$$ \cA_{(n_i)_{i=1,\ldots,4}}((\omega_i)_{i=1,\ldots,4}) = c_{D,n} c_{D,n'} (\prod_{i=1}^4
c_{D,n_i} ) \cdot \tilde\cA_{(n_i)_{i=1,\ldots,4}}((\omega_i)_{i=1,\ldots,4}), $$
where $\tilde\cA_{(n_i)_{i=1,\ldots,4}}((\omega_i)_{i=1,\ldots,4})$ is given by
$$ 
\int_{S^{D-1}} d\omega \sum_{\ell,\ell',\ell_i} \overline{Y_{\ell}^{(n)}(\omega)} 
\overline{Y_{\ell'}^{(n')}(\omega)} 
\prod_i \overline{Y_{\ell_i}^{(n_i)}(\omega_i)}
\left( \begin{array}{ccc} n & n_1 & n_2 \\
\ell & \ell_1 & \ell_2 \end{array}\right)_D 
\left( \begin{array}{ccc} n' & n_3 & n_4 \\
\ell' & \ell_3 & \ell_4 \end{array}\right)_D 
$$
 where we used the shorthand notation
 \begin{equation}\label{tricoeffnotation}
 \left( \begin{array}{ccc} n & n_1 & n_2 \\
\ell & \ell_1 & \ell_2 \end{array}\right)_D := \langle  Y_{\ell}^{(n)}, Y_{\ell_1}^{(n_1)}, Y_{\ell_2}^{(n_2)} \rangle_D.
 \end{equation}
Using the orthogonality relations for the spherical harmonics, this gives 
$$ (\prod_{i=1}^4
c_{D,n_i} ) \sum_{\ell_1,\ell_2,\ell_3,\ell_4} 
\overline{Y_{\ell_1}^{(n_1)}(\omega_1)} \overline{Y_{\ell_2}^{(n_2)} (\omega_2)}
\overline{Y_{\ell_3}^{(n_3)}(\omega_3)} \overline{Y_{\ell_4}^{(n_4)}(\omega_4)}\, \,\,
\cK_{\underline{n}, \underline{\ell}}(n), $$
with the coefficients as in \eqref{ang2starscoeff}.
\endproof

The coefficients $\cK_{\underline{n}, \underline{\ell}}(n)$ are usually very
involved to compute explicitly (see (3.3), (3.6) and (4.7) of \cite{Alis}). However,
some terms simplify greatly in the case $D=4$, and that will allow us to
show the occurrence of functions closely related to multiple 
polylogarithm functions in \S \ref{multpolylogsSec} below. 
For later use, we give here the explicit computation, in dimension $D=4$, 
of the coefficients $\cK_{\underline{n}, \underline{\ell}}(n)$,  in the
particular case with $\underline{\ell}=0$.

\begin{prop}\label{KcoeffD4lem}
In the case where $D=4$, the coefficient $\cK_{\underline{n}, \underline{\ell}}(n)$
with $\ell_i=0$ has the form 
\begin{equation}\label{ell0D4Kcoeff}
\cK^{(D=4)}_{\underline{n}, \underline{0}}(n) = (\prod_{i=1}^4 \frac{1}{(n_i+1)^{1/2}}) \, 
\frac{4\pi^4}{(n+1)^3},
\end{equation}
in the range where $n+ n_1 +n_2$ and $n+n_3+n_4$ are even and the
inequalities  $|n_j-n_k|\leq n_i \leq n_j+n_k$ hold for $(n_i,n_j,n_k)$ equal
to $(n,n_1,n_2)$ or $(n,n_3,n_4)$ and transpositions,
and are equal to zero outside of this range.
\end{prop}

\proof We use the fact that 
(\cite{Alis}, (4.9) and \cite{Junk}, (22) and (23))
the coefficients $\langle  Y_{0}^{(n_1)}, Y_{0}^{(n_2)}, Y_{0}^{(n_3)} \rangle_D$
are zero outside the range where 
\begin{equation}\label{rangesum}
\sum_i n_i\ \ \text{ is even and } \ \ \ |n_j-n_k|\leq n_i \leq n_j+n_k,
\end{equation}
while within this range they are given by the expression
\begin{equation}\label{ell03coeff}
\epsilon_D \frac{1}{\Gamma(D/2)} \left(
\frac{(J+D-3)!}{(D-3)!\, \Gamma(J + D/2)} \prod_i \frac{(n_i+\frac{D}{2}-1)\, 
\Gamma(J-n_i+\frac{D}{2}-1)}{d_{n_i}^{(D)} \, (J-n_i)!}\right)^{1/2},
\end{equation}
where $\epsilon_D$ is a sign, $J=\frac{1}{2} \sum_i n_i$, and
$d_{n_i}^{(D)}=\dim \cH_{n_i}(S^{D-1})$. In the particular case where $D=4$, 
the expression \eqref{ell03coeff} reduces to  
$$ \left( \begin{array}{ccc} n_1 & n_2 & n_3 \\
0 & 0 & 0 \end{array}\right)_4 =\epsilon_4\, \prod_i \frac{(n_i+1)^{1/2}}{(d_{n_i}^{(4)})^{1/2}}  =
\epsilon_4\, \prod_i (n_i+1)^{-1/2}, $$
using again the fact that $\dim \cH_n(S^3)=(n+1)^2$.  Thus, we obtain 
$$ \cK^{(D=4)}_{n_i,\ell_i=0}(n) = c_{4,n}^2 \left( \begin{array}{ccc} n & n_1 & n_2 \\
0 & 0 & 0 \end{array}\right)_4 \left( \begin{array}{ccc} n & n_3 & n_4 \\
0 & 0 & 0 \end{array}\right)_4 = $$
$$ =\left( \frac{2 \, Vol(S^3)}{2 (n+1)} \right)^2 \frac{1}{(n+1)} \prod_{i=1}^4 \frac{1}{(n_i+1)^{1/2}} \,  =
\prod_{i=1}^4 (n_i+1)^{-1/2} \, \frac{4\pi^4}{(n+1)^3}. $$
\endproof

\medskip
\subsection{Gluing stars of vertices}\label{multpolylogsSec}
We now consider the full integrand, including the radial variables and again look at the
effect of gluing together two half edges of two trivalent stars. We will see that one can
explicitly identify the leading term in the resulting expression in the integrand with a
function closely related to multiple polylogarithms.

\begin{lem}\label{rintstarlem}
Consider the star of a trivalent vertex, and let $D=4$. After a change of variables 
$t_i=r_{v_i}/r$, with $r=r_v$ for $v$ the central vertex of the star, the
integrand \eqref{radAint}, for an orientation ${\bf o}$, can be written as
an expression $\cI_{\bf o}(r,t_1,t_2,t_3,\omega_1,\omega_2,\omega_3)\, dr dt_1 dt_2 dt_3$ 
of the form
\begin{equation}\label{rintstart}
r^9 \prod_{i=1}^3 t_i^{\alpha_i} 
\sum_{n_1,n_2,n_3} \cA_{(n_1,n_2,n_3)} (\omega_1,\omega_2,\omega_3) t_1^{\epsilon_1 n_1}
t_2^{\epsilon_2 n_2} t_3^{\epsilon_3 n_3}\, dr\, \prod_{i=1}^3 dt_i,
\end{equation}
where $\alpha_i=1$ and $\epsilon_i= 1$ if the half-edge $e_i$ is outgoing 
in the orientation ${\bf o}$  and $\alpha_i=3$ and $\epsilon_i=- 1$ if it is incoming,
and where $\cA_{\underline{n}}(\underline{\omega})=\cA_{(n_j)}(\omega_j)$ 
is the angular integral of Lemma \ref{angIntStar} with $D=4$.
\end{lem}

\proof The integrand of \eqref{radAint}, for the case of a trivalent star, is of the form
\begin{equation}\label{Fvolstar}
 \prod_{i=1}^3 \cF(r_{s(e_i)},r_{t(e_i)}) r^3 dr\, \prod_{i=1}^3 r_{v_i}^3 dr_{v_i}, 
\end{equation} 
with
$$ \prod_{i=1}^3 \cF(r_{s(e_i)},r_{t(e_i)}) =  (\prod_{i=1}^3  r_{t(e_i)}^{-2}) \, \sum_{\underline{n}} \cA_{\underline{n}} (\underline{\omega}) \,  \left( \frac{r_{s(e_1)}}{r_{t(e_1)}} \right)^{n_1} \left( \frac{r_{s(e_2)}}{r_{t(e_2)}} \right)^{n_2}  \left( \frac{r_{s(e_3)}}{r_{t(e_3)}} \right)^{n_3} . $$
When combined with the volume form as in \eqref{Fvolstar}, this can be rewritten as
\begin{equation}\label{Fvolstar2}
 r^{\alpha_0} r_1^{\alpha_1} r_2^{\alpha_2} r_3^{\alpha_3} \sum_{\underline{n}} \cA_{\underline{n}} (\underline{\omega}) \,  \left( \frac{r_1}{r} \right)^{\epsilon_1 n_1} \left( \frac{r_2}{r} \right)^{\epsilon_2 n_2}  \left( \frac{r_3}{r} \right)^{\epsilon_3 n_3}\, dr\, dr_1 dr_2 dr_3, 
\end{equation} 
where the exponents $\alpha_i$ are given by the table
\begin{center}
\begin{tabular}{|c| r r r r|}\hline  & ${\bf o}_0$ & ${\bf o}_1$ & ${\bf o}_2$ & ${\bf o}_3$ \\ \hline
$\alpha_0$ & $-3$ & $3$ & $-1$ & $1$ \\
$\alpha_1$ & $1$ & $3$ & $1$ & $1$ \\
$\alpha_2$ & $1$ & $3$ & $3$ & $1$ \\
$\alpha_3$ & $1$ & $3$ & $3$ & $3$ \\ \hline
\end{tabular}
\end{center} 
where the orientation ${\bf o}_0$ has all the half-edges of the star pointing outward,
${\bf o}_1$ all pointing inward, ${\bf o}_2$ has $e_1$ outward and $e_2$ $e_3$
inward and ${\bf o}_3$ has $e_1$ and $e_2$ outward and $e_3$
inward. All the other cases are obtained by relabeling of indices.
After we change variables to $t_i=r_{v_i}/r$, we obtain
$dr \wedge \wedge_{i=1}^3 dr_i =  r^3 dr \wedge \wedge_{i=1}^3 dt_i$
and \eqref{Fvolstar2} becomes
\begin{equation}\label{Fvolstar3} \cI_{\bf o}  (r,(t_i),\underline{\omega}) =
r^9 t_1^{\alpha_1} t_2^{\alpha_2} t_3^{\alpha_3} \sum_{\underline{n}} \cA_{\underline{n}} (\underline{\omega}) \,  \left( t_1 \right)^{\epsilon_1 n_1} \left( t_2 \right)^{\epsilon_2 n_2}  \left( t_3 \right)^{\epsilon_3 n_3}\, dr\, dt_1 dt_2 dt_3.
\end{equation}
\endproof

We can now perform the gluing of two stars by matching an oriented half-edge of one
trivalent star to an oriented half edge of the other, so that one obtains an oriented edge.
This means integrating 
\begin{equation}\label{IromegaR}
\int_0^\infty \int_{\bar\Sigma}\int_{S^{D-1}} 
\cI_{\bf o}(r,t,t_1,t_2,\omega,\omega_1,\omega_2) 
\cI_{\bf o}(r,t,t_3,t_4,\omega,\omega_3,\omega_4) \, dr\, dt \,d\omega.
\end{equation}
There is an overall divergent factor arising from the integration of \eqref{IromegaR}
in the variable $r$, which can be taken care of by a cutoff regularization. Up to this
divergence, one obtains an integrand 
$\cI(t_1,t_2,t_3,t_4, \omega_1,\omega_2,\omega_3,\omega_4)$, 
as the result of gluing two trivalent stars by matching oriented half-edges to form 
an oriented edge $e$, which is given by
\begin{equation}\label{Iromega}
\cI_{\bf o} (t_i, \omega_i) = \int_{\bar\Sigma}\int_{S^{D-1}} 
\cI_{\bf o}(t,t_1,t_2,\omega,\omega_1,\omega_2) 
\cI_{\bf o}(t,t_3,t_4,\omega,\omega_3,\omega_4) \, dt d\omega,
\end{equation}
where the domain of integration $\bar\Sigma=\bar\Sigma(t_1,t_2,t_3,t_4)$ for the 
variable $t$ is given by $$\bar\Sigma =\cap_{i,j: t(e_i)=s(e), s(e_j)=t(e)} \{ t\,|\, t_i \leq t \leq t_j \}.$$
In the following we write $\underline{t}=(t_1,t_2,t_3,t_4)$ and similarly for
$\underline{\omega}$, $\underline{n}$ and $\underline{\ell}$. 

By combining \eqref{ang2starsjoin} with \eqref{AintStar}, we can
rephrase \eqref{Iromega} in terms of isoscalars. This gives a decomposition of
$\cI_{\bf o} (t_i, \omega_i)$ into a sum of terms of the form
$$
\cI_{\bf o} (t_i, \omega_i)  = \sum_{\underline{n},\underline{\ell}} \cI_{\bf o,\underline{n},\underline{\ell}} (t_i, \omega_i).
$$
We denote by $\cI_{{\bf o},0} (\underline{t}, \underline{\omega})$ the leading term 
\begin{equation}\label{Io0}
\cI_{\bf o,0} (t_i, \omega_i)  = \sum_{\underline{n}} \cI_{\bf o,\underline{n},\underline{0}} (t_i, \omega_i),
\end{equation}
involving only the isoscalars with all $\ell_i=0$.
We have the following result computing the terms $\cI_{{\bf o},0}(t_i, \omega_i)$.

\begin{lem}\label{Iradint2stars}
In the case $D=4$, the integrands $\cI_{{\bf o},0} (\underline{t}, \underline{\omega})$
are explicitly given by
\begin{equation}\label{IromegaD4}
\cI_{{\bf o},0} (\underline{t}, \underline{\omega}) =\sum_{\underline{n}} 
( \prod_{i=1}^4 c_{D,n_i} \overline{Y^{(n_i)}_0 (\omega_i)} \frac{t_i^{\alpha_i+\epsilon_i n_i} \, dt_i}{(n_i+1)^{1/2}} ) \, \int_{\bar\Sigma} t^4\, dt \,\, \sum_n \frac{4\pi^2}{(n+1)^3} t^{\epsilon n} ,
\end{equation}
where the sum over the indices $n$ and $\underline{n}$ is restricted by the constraints
$n+ n_1 +n_2$ and $n+n_3+n_4$ are even and the
inequalities  $|n_j-n_k|\leq n_i \leq n_j+n_k$ hold for $(n_i,n_j,n_k)$ equal
to $(n,n_1,n_2)$ or $(n,n_3,n_4)$ and transpositions. 
\end{lem}

\proof  Using \eqref{Iromega} and \eqref{ang2starsjoin}, \eqref{ang2starscoeff}, \eqref{rintstart},
we obtain for $\cI_{\bf o} (\underline{t}, \underline{\omega})$ the expression
$$ \sum_{\underline{n}} \sum_{\underline{\ell}}
( \prod_{i=1}^4 c_{D,n_i} \overline{Y^{(n_i)}_{\ell_i} (\omega_i)} t_i^{\alpha_i+\epsilon_i n_i} \, dt_i )
 \, \int_{\bar\Sigma} t^4\, dt \,\, \sum_n \cK_{\underline{n},\underline{\ell}}(n)\,\, t^{\epsilon n} . $$
The expression \eqref{IromegaD4} then follows directly from the form
\eqref{ell0D4Kcoeff} of the coefficients $\cK^{(D=4)}_{\underline{n}, \underline{0}}(n)$.
The factor $t^4$ in the integral
comes from the exponents $\alpha=1$ and $\alpha=3$ of the two half edges, which have 
matching orientations.
\endproof

Notice that, without the constraints on the summation range of the indices $n, n_i$,
we would obtain again an integral of the general form \eqref{Freitas},
involving polylogarithm functions ${\rm Li}_s(t^\epsilon)$, 
with $s=3=k-2$ as in the case of polygons analyzed above. 
However, because not all values of $n, n_i$ are allowed
and one needs to impose the constraints of the form \eqref{rangesum}, one obtains
more interesting expressions. We first introduce some notation.

\medskip

\subsubsection{Summation domains and even condition}
In the following we let $\cR$ denote a domain of summation for integers
$(n_1,\ldots, n_k)$. We consider in particular the cases
\begin{equation}\label{Rdomains}
\begin{array}{ll}
\cR = \cR_P^{(k)} & :=\{ (n_1,\ldots, n_k) \,|\, n_i >0, \,\,\, i=1,\ldots, k \}   \\[2mm]
\cR = \cR_{MP}^{(k)} & :=\{ (n_1,\ldots, n_k) \,|\, n_k > \cdots > n_2 > n_1 > 0 \} \\[2mm]
\cR = \cR_T^{(3)} & := \{ (n_1,n_2, n_3) \,|\, n_2 > n_1, \,\,\,\, n_2-n_1 < n_3 < n_2 + n_1 \}.
\end{array}
\end{equation}
We denote by ${\rm Li}^\cR_{s_1,\ldots,s_k}(z_1,\ldots, z_k)$ the associated series
\begin{equation}\label{LiR}
{\rm Li}^\cR_{s_1,\ldots, s_k}(z_1,\ldots, z_k) = \sum_{(n_1,\ldots, n_k)\in \cR} \frac{ z_1^{n_1}\cdots
z_k^{n_k}}{n_1^{s_1} \cdots n_k^{s_k}}.
\end{equation}
In the first two cases of \eqref{Rdomains}, this is, respectively, a product of polylogarithms
${\rm Li}^{\cR_P}_{s_1,\ldots, s_k}(z_1,\ldots, z_k) = \prod_j {\rm Li}_{s_j}(z_j)$ and
and a multiple polylogarithm ${\rm Li}^{\cR_{MP}}_{s_1,\ldots, s_k}(z_1,\ldots, z_k) 
={\rm Li}_{s_1,\ldots, s_k}(z_1,\ldots, z_k)$. We will discuss the third case more in detail
below. We then define
\begin{equation}\label{LiRevenodd}
\begin{array}{ll}
{\rm Li}^{\cR,{\rm even}}_{s_1,\ldots, s_k}(z_1,\ldots, z_k) & := \frac{1}{2} \left(
{\rm Li}^\cR_{s_1,\ldots, s_k}(z_1,\ldots, z_k) + {\rm Li}^\cR_{s_1,\ldots, s_k}(-z_1,\ldots, -z_k) \right)
\\[2mm]
{\rm Li}^{\cR,{\rm odd}}_{s_1,\ldots, s_k}(z_1,\ldots, z_k) & := \frac{1}{2} \left(
{\rm Li}^\cR_{s_1,\ldots, s_k}(z_1,\ldots, z_k) - {\rm Li}^\cR_{s_1,\ldots, s_k}(-z_1,\ldots, -z_k) \right).
\end{array}
\end{equation}
The odd ${\rm Li}^{\cR,{\rm odd}}_{s_1,\ldots, s_k}(z_1,\ldots, z_k)$ is a direct generalization
of the Legendre $\chi$ function, while the even ${\rm Li}^{\cR,{\rm even}}_{s_1,\ldots, s_k}(z_1,\ldots, z_k)$ corresponds to summing only over those indices in $\cR$ whose sum is even,
\begin{equation}\label{LiReven}
{\rm Li}^{\cR,{\rm even}}_{s_1,\ldots, s_k}(z_1,\ldots, z_k) =\sum_{(n_1,\ldots, n_k)\in \cR, \, \sum_i n_i \in 2\N} \frac{ z_1^{n_1}\cdots
z_k^{n_k}}{n_1^{s_1} \cdots n_k^{s_k}}.
\end{equation}
More generally, one can also consider summations of the form
\begin{equation}\label{LiRevenoddi}
{\rm Li}^{\cR,\cE_1,\ldots, \cE_k}_{s_1,\ldots, s_k}(z_1,\ldots, z_k) =\sum_{(n_1,\ldots, n_k)\in \cR, \, n_i \in \cE_i} \frac{ z_1^{n_1}\cdots
z_k^{n_k}}{n_1^{s_1} \cdots n_k^{s_k}},
\end{equation}
where, for each $i=1,\ldots, k$, $\cE_i =2\N$ or $\cE_i=\N \smallsetminus 2\N$, that is,
some of the summation indices are even and some odd.

\medskip
\subsubsection{Matching half-edges}
We now illustrate in one sufficiently simple and explicit case, what
the leading $\ell=0$ term looks like when all the half-edges of stars
are joined together. We look at the case of two stars of trivalent 
vertices with the half edges pairwise joined, that is, the 3-banana
graph (two vertices and three parallel edges between them).

\begin{prop}\label{3banana}
In the case of $D=4$, consider the graph with two vertices and three parallel edges between
them. The $\underline{\ell}=0$ amplitude $\cI_{{\bf o},0}$ is given by
\begin{equation}\label{3banana0}
\cI_{{\bf o},0} =\int_0^1 t^9 \,  (2^6 {\rm Li}^{\cR_{MP},\text{odd},\text{even}}_{6,3}(t,t) + 
2{\rm Li}^{\cR_T,\text{even}}_{3,3,3}(t,t,t))  \, dt .
\end{equation}
\end{prop}

\proof 
There is a unique acyclic orientation of this graph, with the
three edges oriented in the same direction. Thus, there is a single
variable $t\in [0,1]=\bar\Sigma$ in the integrand of $\cI_{{\bf o},0}$,
and the latter has the form
$$ \sum_{(n_1,n_2,n_3)\in \cD} \cK_{n_1,n_2,n_3} t^{n_1+n_2+n_3}, $$
where the coefficients $\cK_{n_1,n_2,n_3}$ are given by
$$ \cK_{n_1,n_2,n_3} = \frac{c_{4,n_1}^2  c_{4,n_2}^2 c_{4,n_3}^4}{(n_1+1) (n_2+1) (n_3+1)}=
\frac{(4 \pi^2)^3}{ (n_1+1)^3 (n_2+1)^3 (n_3+1)^3}, $$
according to Proposition \ref{KcoeffD4lem}, and the fact that 
all the half-edges of the two trivalent stars are matched. The summation domain $\cD$ is
given by
$$ \cD=\{ (n_1,n_2,n_3)\,|\, n_i \geq 0 \,\, |n_j-n_k|\leq n_i \leq n_j+n_k, \,\,\, 
\sum_i n_i \text{ even} \}. $$
We subdivide this into separate domains $\cD=\cD_1 \cup \cD_2\cup \cD_3 \cup \cD_4 \cup\cD_5$, where
$$ 
\begin{array}{l}
\cD_1=\{ n_1=0, \,\, n_2 \geq 0, \,\, n_3=n_2 \} \\
\cD_2=\{ n_2=0, \,\, n_1 >0, \,\, n_3=n_1\} \\
\cD_3=\{ n_1>0, \,\, n_2=n_1, \,\, 0\leq n_3 \leq 2n_1, \,\,\, n_3 \text{ even } \} \\
\cD_4=\{ 0<n_2<n_1, \,\,\, n_1-n_2\leq n_3 \leq n_1+n_2, \,\,\, \sum_i n_i  \text{ even } \} \\
\cD_5=\{ 0<n_1<n_2, \,\,\, n_2-n_1\leq n_3 \leq n_1+n_2, \,\,\, \sum_i n_i  \text{ even } \}.
\end{array}
$$
We have
$$ \sum_{(n_1,n_2,n_3)\in \cD_1} \frac{t^{n_1+n_2+n_3}}{(n_1+1)^3 (n_2+1)^3 (n_3+1)^3} =
t^{-2} \sum_{n\geq 1} \frac{t^{2n}}{n^6} = t^{-2} {\rm Li}_6(t^2), $$
$$ \sum_{(n_1,n_2,n_3)\in \cD_2} \frac{t^{n_1+n_2+n_3}}{(n_1+1)^3 (n_2+1)^3 (n_3+1)^3} =
t^{-2} \sum_{n\geq 2} \frac{t^{2n}}{n^6}= t^{-2} {\rm Li}_6(t^2) -1 $$
$$ \sum_{(n_1,n_2,n_3)\in \cD_3} \frac{t^{n_1+n_2+n_3}}{(n_1+1)^3 (n_2+1)^3 (n_3+1)^3} =
\sum_{n>0, 0\leq \ell \leq n} \frac{t^{2(n+\ell)}}{(2\ell +1)^3 (n+1)^6} $$ 
$$ =-1+2^6 t^{-3} \sum_{n\geq 0,\, 0\leq \ell \leq n}\frac{t^{2n+2+2\ell+1}}{(2\ell +1)^3 (2n+2)^6}$$
$$ = -1+ 2^6 t^{-3} \sum_{0<m_1<m_2 \atop m_1\text{ odd }, \, m_2\text{ even }} 
\frac{t^{m_1+m_2}}{m_1^6 m_2^3} =-1+ 2^6 t^{-3} {\rm Li}^{\cR_{MP},\text{odd},\text{even}}_{6,3}(t,t) 
$$
$$ \sum_{(n_1,n_2,n_3)\in \cD_4} \frac{t^{n_1+n_2+n_3}}{(n_1+1)^3 (n_2+1)^3 (n_3+1)^3} =
t^{-3} {\rm Li}^{\cR_T,\text{even}}_{3,3,3}(t,t,t) +1 - t^{-2} {\rm Li}_6(t^2) $$
$$ \sum_{(n_1,n_2,n_3)\in \cD_5} \frac{t^{n_1+n_2+n_3}}{(n_1+1)^3 (n_2+1)^3 (n_3+1)^3} =
t^{-3} {\rm Li}^{\cR_T,\text{even}}_{3,3,3}(t,t,t) +1 - t^{-2} {\rm Li}_6(t^2), $$
where in the last two cases the term $t^{-2} {\rm Li}_6(t^2) -1$ corresponds to 
the summation over $m_2=1$, $m_1>1$ and $m_3=m_1$ (respectively, 
$m_1=1$, $m_2>1$, $m_3=m_2$), with $m_i=n_1+1$.
The integrand has a factor of $t^4$ for each edge, as in Lemma \ref{Iradint2stars},
which gives a power of $t^{12}$ that combines with the $t^{-3}$ factor in the result
of the sum of the terms above to give the $t^9$ factor in \eqref{3banana0}.
\endproof

For more general graphs, where more vertices and more stars are involved,
one gets summations involving several ``triangular conditions" $|n_j-n_k|\leq n_i\leq n_j+n_k$
around each vertex, and the integrand can correspondingly be expressed in terms
of series with a higher depth. Moreover, notice that we have focused here on the
leading terms  $\cK_{\underline{n},\underline{\ell}=0}^{D=4}(n)$ only. When one
includes all the other terms $\cK_{\underline{n},\underline{\ell}}(n)$ with $\ell_i\neq 0$,
the expressions become much more involved, as these coefficients are expressed in
terms of the isoscalars \eqref{isofactor} and of the standard $3j$-symbols for $SO(3)$,
through the factorization \eqref{isoscal}. The isofactors are known explicitly 
\cite{Alis}, \cite{Junk} so the computation can in principle be carried out in full,
but it becomes much more cumbersome. We will discuss this elsewhere.
 
\medskip

Next we show that the functions ${\rm Li}^{\cR_T}_{s_1,s_2,s_3}(z_1,z_2,z_3)$
that appear in these Feynman amplitude computations can be related, via the
Euler--Maclaurin summation formula, to some well known generalizations of multiple
zeta values and multiple polylogarithms. 

\subsubsection{Mordell--Tornheim and Apostol--Vu series}
We consider two generalizations of the multiple polylogarithm series,
which arise in connection to the Mordell--Tornheim and the Apostol--Vu
multiple series. The Mordell--Tornheim multiple series is given by \cite{Mordell},
\cite{Tornheim}
\begin{equation}\label{MTseries}
\zeta_{MT,k}(s_1,\ldots,s_k;s_{k+1}) = \sum_{(n_1,\ldots,n_k)\in \cR_P^{(k)}}
n_1^{-s_1}\cdots n_k^{-s_k} (n_1+\cdots+n_k)^{-s_{k+1}},
\end{equation}
with an associated multiple polylogarithm-type function 
\begin{equation}\label{MTLi}
{\rm Li}^{MT}_{s_1,\ldots,s_k;s_{k+1}}(z_1,\ldots,z_k;z_{k+1}) = 
\sum_{(n_1,\ldots,n_k)\in\cR_P^{(k)}} \frac{z_1^{n_1} \cdots z_k^{n_k}
z_{k+1}^{(n_1+\cdots+n_k)}}{n_1^{s_1}\cdots n_k^{s_k} (n_1+\cdots+n_k)^{s_{k+1}}}.
\end{equation}
Similarly, the Apostol--Vu multiple series \cite{ApVu} is defined as 
\begin{equation}\label{AVseries}
\zeta_{AV,k}(s_1,\ldots,s_k;s_{k+1}) = \sum_{(n_1,\ldots,n_k)\in \cR_{MP}^{(k)}}
n_1^{-s_1}\cdots n_k^{-s_k} (n_1+\cdots+n_k)^{-s_{k+1}},
\end{equation}
and we consider the associated multiple polylogarithm-type series
\begin{equation}\label{AVLi}
{\rm Li}^{AV}_{s_1,\ldots,s_k;s_{k+1}}(z_1,\ldots,z_k;z_{k+1}) = 
\sum_{(n_1,\ldots,n_k)\in\cR_{MP}^{(k)}} \frac{z_1^{n_1} \cdots z_k^{n_k}
z_{k+1}^{(n_1+\cdots+n_k)}}{n_1^{s_1}\cdots n_k^{s_k} (n_1+\cdots+n_k)^{s_{k+1}}}.
\end{equation}

\subsubsection{Euler--Maclaurin formula}
A  way to understand better the behavior of the functions \eqref{LiReven} with
$\cR=\cR^{(3)}_T$ that appear in this result, is in terms of the Euler--Maclaurin
summation formula. 

\begin{lem}\label{derivativesft}
Let $f(t)=x^t t^{-s}$. Then 
\begin{equation}\label{fktxs}
f^{(k)}(t) = \sum_{j=0}^k (-1)^{k-j} \binom{k}{j} \binom{s+k-j-1}{k-j} (k-j)! t^{-(s+k-j)} x^t \, \log(x)^j.
\end{equation}
\end{lem}

\proof Inductively, we have
$$ f^{(k)}(t) = \sum_{j=0}^k (-1)^{k-j} \binom{k}{j} s (s+1)\cdots (s+k-j-1) \, 
t^{-(s+k-j)}\, x^t \, \log(x)^j , $$
where $s (s+1)\cdots (s+k-j-1) = \binom{s+k-j-1}{k-j} (k-j)!$.
\endproof 

\smallskip

The Euler--Maclaurin summation formula gives
\begin{equation}\label{EMLsummation}
\begin{array}{ll}
\displaystyle{\sum_{n=a}^b f(n)} & = \displaystyle{ \int_a^b f(t)dt + \frac{1}{2} (f(b)+f(a)) } \\[2mm]
& \displaystyle{ + \sum_{k=2}^{N} \frac{b_k}{k!} (f^{(k-1)}(b)-f^{(k-1)}(a)) } \\[2mm]
& \displaystyle{ - \int_a^b \frac{B_N(t-[t])}{N!} f^{(N)}(t)\, dt },
\end{array}
\end{equation}
where $b_k$ are the Bernoulli numbers and $B_k$ the Bernoulli polynomials.
We then have the following result.

\begin{prop}\label{LiTemcl}
Consider the series ${\rm Li}^{\cR}_{s_1,s_2,s_3}(z_1,z_2,z_3)$ defined as in \eqref{LiR},
with $\cR=\cR^{(3)}_T$. When applying the Euler--Maclaurin formula to 
the innermost sum, the summation terms in \eqref{EMLsummation} give rise to terms of
the form
\begin{equation}\label{termsEML1}
\pm F_{j,k}(s_3,z_3)\,\, {\rm Li}^{AV}_{s_1,s_2;s_{3+k-j}}(z_1,z_2;z_3) 
\end{equation}
or
\begin{equation}\label{termsEML2} 
\pm F_{j,k}(s_3,z_3)\,\, {\rm Li}^{MT}_{s_1,s_{3+k-j};s_2}(z_1,z_2;z_3),
\end{equation}
where
\begin{equation}\label{Fjksz}
F_{j,k}(s,z)= \frac{b_k}{k!} \binom{k}{j} \binom{s+k-j-1}{k-j} (k-j)! \,\, \log(z)^j 
\end{equation}
\end{prop}

\proof
For ${\rm Li}^\cR_{s_1,s_2,s_3}(z_1,z_2,z_3)$, with
$\cR=\cR^{(3)}_T$, the summation
\begin{equation}\label{sum3RT}
 \sum_{n_2-n_1< n_3 < n_2+n_1} \frac{z_3^{n_3}}{n_3^{s_3}} 
\end{equation} 
can be expressed, using Lemma \ref{derivativesft},
through the Euler--Maclaurin summation formula \eqref{EMLsummation}.
Up to a sign, each summation term in the right-hand-side of \eqref{EMLsummation} is the
product of a function of $z_3$ of the form $F_{j,k}(s_3,z_3)$, as in \eqref{Fjksz},
and a term of the form
$$ \frac{z_3^{n_2+n_1}}{(n_2+n_1)^{s_3+k-j}} \ \ \ \ \text{ or } \ \ \ \
\frac{z_3^{n_2-n_1}}{(n_2-n_1)^{s_3+k-j}}. $$
When inserted back into the summation on the remaining indices $n_2> n_1$, this
gives summations of the form
\begin{equation}\label{sum1T}
\sum_{n_2> n_1>0} \frac{z_1^{n_1} z_2^{n_2} z_3^{n_1+n_2}}{n_1^{s_1} n_2^{s_2} (n_1+n_2)^{s_3+k-j}}, 
\end{equation}
in the first case, or in the second case, after a change of variables $m=n_2-n_1$, $n=n_1$ 
in the indices
\begin{equation}\label{sum2T}
\sum_{n>0, m>0} \frac{z_1^{n} z_3^{m} z_2^{n+m} }{n^{s_1} m^{s_3+k-j} (n+m)^{s_2}}, 
\end{equation}
which are respectively of the form \eqref{MTLi} and \eqref{AVLi}.
\endproof

%%%%%%%%%%

\bigskip
\section{Wonderful compactifications and the Feynman amplitudes}\label{wondSec}

In this section, we consider the case of the Feynman amplitude \eqref{amplitudeZ}
introduced in \S \ref{weightsSec}. As in Definition \ref{FeyamplZcase}, the
locus of integration is, in this case, the complex variety $X^{\BV_\Gamma}\times \{ y=(y_v) \}$,
for a fixed choice of a point $y=(y_v)$, inside the configuration space 
$Z^{\BV_\Gamma}$ with $Z=X\times X$.

\smallskip

To discuss an appropriate regularization procedure for the Feynman integral
and interpret the result in terms of periods, we need first some basic facts about
the wonderful compactifications of the configuration spaces $F(X,\Gamma)$.

\smallskip

We described in detail in our previous work \cite{CeyMar} the geometry
of the wonderful compactifications of the configuration spaces ${\rm Conf}_\Gamma(X)$.
We recall here the main definitions and statements, 
adapted from ${\rm Conf}_\Gamma(X)$ to $F(X,\Gamma)$. The arguments 
are essentially the same as in \cite{CeyMar}.

\medskip
\subsection{Arrangements of diagonals}\label{wondcompSec1}

A {\em simple arrangement} of subvarieties of a smooth quasi-projective
ambient variety $Y$ is a finite collection of nonsingular closed 
subvarieties $\cS = \{S_i \subset Y, i \in I \}$ such that  
\begin{itemize}
\item all nonempty intersections $\bigcap_{i \in J} S_i$  for $J \subset I$ are 
 in the collection $\cS$.
\item  for any pair $S_i, S_j \in \cS$, the intersection 
$S_i \cap S_j$ is {\em clean}, that is, the tangent bundle of the
intersection is the intersection of the restrictions of the tangent bundles. 
\end{itemize}
 
\medskip 
\subsubsection{Diagonals of  induced subgraphs and their arrangement}
\label{DgDef}

For each induced and not necessarily connected subgraph $\gamma \subseteq \Gamma$, 
the corresponding (poly)diagonal is
\begin{equation}\label{Dgeq}
\Delta_\gamma^{(Z)} = \{ z=(z_v) \in Z^{\BV_\Gamma}
\mid   p(z_{v}) = p(z_{v'}) \ {\text for } \
\{ v,v' \} =\partial_\Gamma(e), \,\, e\in \BE_\gamma \}.
\end{equation} 

We  have the following simple description.

\begin{lem} \label{Diag}
The diagonal $\Delta_\gamma^{(Z)}$  is isomorphic to 
$X^{\BV_{\Gamma//\gamma}}\times  X^{\BV_\Gamma}$,
and has dimension 
\begin{equation}\label{Ddim}
\dim \Delta_\gamma^{(Z)} = \dim X^{\BV_{\Gamma//\gamma}}\times  X^{\BV_\Gamma}=
\dim(X)  (2 |\BV_\Gamma| - |\BV_\gamma| + b_0(\gamma))
\end{equation}
where $b_0(\gamma)$ is the number of connected components of $\gamma$.
\end{lem}

We then observe that the diagonals form an arrangement of subvarieties.
This is the analog of Lemma 5 of \cite{CeyMar}.

\begin{lem}  
\label{SsetDeltas}
For a given graph $\Gamma$, the collection
\begin{equation}\label{cSGamma}
\cS_\Gamma=\{ \Delta_\gamma^{(Z)} \,|\, \gamma \in \BS\BG(\Gamma)  \},
\end{equation}
with $\BS\BG(\Gamma)$ the set of all induced subgraphs of $\Gamma$,
is a simple arrangement of diagonal subvarieties in $Z^{\BV_\Gamma}$.
\end{lem}

\begin{proof}
Let $\gamma_1$ and $\gamma_2$ be a pair of induced subgraphs. 
If $\gamma_1 \cap \gamma_2 = \emptyset$,
then $\gamma = \gamma_1 \cup \gamma_2$ is in $\BS\BG(\Gamma)$, and the corresponding diagonal $\Delta_\gamma^{(Z)}$
is given by the intersection $\Delta_{\gamma_1}^{(Z)} \cap \Delta_{\gamma_2}^{(Z)}$. 
On the other hand, if $\gamma_1 \cap \gamma_2 \ne \emptyset$, we consider  the connected components $\gamma_j$ 
of the union $\gamma$. Then, the intersection 
$\Delta_{\gamma_1}^{(Z)} \cap \Delta_{\gamma_2}^{(Z)}$ can be written
as $\cap_j \Delta_{i(\gamma_j)}^{(Z)}$ where $i(\gamma_j)$  is the 
smallest induced subgraph of $\Gamma$ containing $\gamma_j$. 
All diagonals are smooth and the ideal sheaf of intersection $\Delta_\gamma^{(Z)}$ 
is the direct sum of the ideal sheaves of the intersecting  diagonals 
$\Delta_{\gamma_j}^{(Z)}$. By Lemma 5.1 of  \cite{li2}, 
their intersections are clean.
\end{proof}

\medskip
\subsubsection{Building set of  the arrangements of diagonals}\label{buildsetSec}

A subset $\cG \subset \cS$ is called a {\em building set} of  the simple 
arrangement $\cS$ if for any $S \in \cS$, the minimal elements in 
$\{ G\in \cG \, :\, G \supseteq S \}$ intersect transversely and the 
intersection is $S$.

A $\cG$-building set for the arrangement $\cS_\Gamma$ can be
identified by considering  further combinatorial properties of graphs.
A graph $\Gamma$ is {\it 2-vertex-connected} (or {\it biconnected}) if it cannot be 
disconnected by the removal of a single vertex along with the open star of edges 
around it. The graph consisting of a single edge is assumed to be biconnected.   

We then have the analog of Proposition 1 of \cite{CeyMar}.
 
\begin{prop}  
\label{biconnGset}
For a given graph $\Gamma$, the set
\begin{equation}\label{cGGamma}
\cG_\Gamma =\{ \Delta_\gamma^{(Z)} \,|\, \gamma\subseteq \Gamma 
\ \text{induced, biconnected} \}
\end{equation}
is a $\cG$-building set for the arrangement $\cS_\Gamma$ of \eqref{cSGamma}.
\end{prop}
 
\begin{proof}
The intersection of a bi-connected subgraph of $\Gamma$ with an induced subgraph
is either empty or a union of bi-connected induced subgraphs attached at cut-vertices.  
We decompose induced subgraphs into bi-connected components. The diagonals corresponding
to these bi-connected components are the minimal elements in the collection $\cS_\Gamma$.
For a pair of bi-connected induced subgraphs $\gamma_1, \gamma_2$ with $\gamma = \gamma_1 \cup \gamma_2$, 
we have the following equalities due to Lemma \ref{Diag};
$\dim \Delta_{\gamma_1}^{(Z)} + \dim \Delta_{\gamma_2}^{(Z)} -\dim \Delta_{\gamma}^{(Z)} = 
\dim (X^{\BV_{\Gamma //\gamma_1}}  \times X^{\BV_{\Gamma }}) +
\dim (X^{\BV_{\Gamma //\gamma_2}}  \times X^{\BV_{\Gamma }}) -
\dim (X^{\BV_{\Gamma //\gamma}}  \times X^{\BV_{\Gamma }}) =
2 \dim(X) |\BV_\Gamma| = \dim Z^{\BV_\Gamma}$;
and these guarantee the transversality of the intersection $\Delta_{\gamma_1}^{(Z)} \cap \Delta_{\gamma_2}^{(Z)}$.
\end{proof}

\medskip
\subsection{The wonderful compactifications of arrangements}
\label{wondcompSec}

To be able to analyze the residues of Feynman integrals, we need a compactification
$Z[\Gamma]$ of the configuration space $F(X,\Gamma)$ satisfying
certain properties. In particular, $Z[\Gamma]$ must contain 
$F(X,\Gamma)$ as the top dimensional stratum, and the complement  
$Z[\Gamma] \setminus F(X,\Gamma)$ of this principal
stratum must be a union of transversally intersecting divisors in  $Z[\Gamma]$. 
The transversality is  essential for the use of iterated Poincar\'e residues, which
we will discuss in \S \ref{RegSec} below.

There is a smooth wonderful compactification  $Z[\Gamma]$
of the configuration space $F(X,\Gamma)$ which is a generalization 
of the Fulton--MacPherson compactification \cite{fm}.  The construction
is completely analogous to the construction of the wonderful
compactifications $\overline{\rm Conf}_\Gamma(X)$ considered
in our previous paper \cite{CeyMar}. Again, we illustrate here
briefly what changes in passing from the case of ${\rm Conf}_\Gamma(X)$
to the case of $F(X,\Gamma)$.

\medskip
\subsubsection{The iterated blowup description}\label{blowupSec} 

As in the case of ${\rm Conf}_\Gamma(X)$ (see \S 2.3 of \cite{CeyMar}),
the wonderful compactification $Z[\Gamma]$ is obtained by an 
iterated sequence of blowups.

The following is the direct analog of
Proposition 2 of \cite{CeyMar}.

\smallskip

Let $|\BV_\Gamma|=n$ and  let $\cG_{k,\Gamma}\subseteq \cG_\Gamma$ be the subcollection 
\begin{equation} 
\cG_{k,\Gamma} = \{ \Delta_\gamma^{(Z)} \,|\, \gamma \in \BS\BG_k(\Gamma) \ 
\text{and\ biconnected} \} \ \ \text{for}\ \ k=1,\dots,n-1.
\end{equation}  
Let  $Y_0=Z^{\BV_\Gamma}$ and let 
$Y_k$ be the blowup of $Y_{k-1}$ along the (iterated) dominant transform 
of $\Delta_{\gamma}^{(Z)} \in \cG_{n-k+1,\Gamma} $.  If $\Gamma$ is itself biconnected, 
then $Y_1$ is the blowup of $Y_0$ along the deepest diagonal $\Delta_\Gamma^{(Z)}$, 
and otherwise $Y_1=Y_0$. Similarly, we have $Y_k=Y_{k-1}$ if there are 
no biconnected induced subgraphs with exactly $n-k+1$ vertices.
The resulting sequence of blowups
\begin{equation}\label{Yks}
Y_{n-1} \to \cdots \to Y_2 \to Y_1 \to Z^{\BV_\Gamma}
\end{equation}
does not depend on the order in which the blowups are performed along the (iterated) 
dominant transforms of the diagonals $\Delta_\gamma^{(Z)}$, 
for $\gamma \in \cG_{n-k+1,\Gamma}$, for a fixed $k$. Thus, the 
intermediate varieties $Y_k$ in the sequence \eqref{Yks} are 
all well defined.  The variety $Y_{n-1}$ obtained through this sequence 
of iterated blowups is  called  the wonderful compactification;
\begin{equation}\label{blowConf}
Z[\Gamma] := Y_{n-1}.
\end{equation}
 Note that $Z[\Gamma]$  is a smooth quasi-projective variety
 as can be seen through its iterated blow-up construction, see   \cite{li2}.
 
\medskip
\subsubsection{Divisors and their intersections} 

Recall from  \cite{li2} that a {\it  $\cG_\Gamma$-nest} is 
a collection $\{\gamma_1, \ldots, \gamma_\ell\}$
of biconnected induced subgraphs with the property that any two subgraphs
$\gamma$ and $\gamma'$ in the set satisfy one of the following:
(1)  $\gamma \cap \gamma' = \emptyset$;
(2) $\gamma \cap \gamma' =\{ v \}$, a single vertex;
(3) $\gamma \subseteq \gamma'$ or $\gamma' \subseteq \gamma$.
 
We then have the following analog of Proposition 4 of \cite{CeyMar}.
  
\begin{prop}\label{strataConf}
For a given biconnected induced subgraph $\gamma \subseteq \Gamma$, let $D_\gamma^{(Z)}$
be the divisor obtained as the iterated dominant transform of $\Delta_\gamma^{(Z)}$ 
in the iterated blowup construction \eqref{Yks} of $Z[\Gamma]$. Then
\begin{equation}\label{unionstrata}
Z[\Gamma] \setminus F(X,\Gamma) =
\bigcup_{\Delta_\gamma^{(Z)} \in \cG_\Gamma} D_\gamma^{(Z)}.
\end{equation}
The divisors $D_\gamma^{(Z)}$ have the property that
\begin{equation}\label{intersGnest}
D_{\gamma_1}^{(Z)}\cap \cdots \cap D_{\gamma_\ell}^{(Z)} \neq \emptyset \Leftrightarrow \{ \gamma_1, \ldots, \gamma_\ell \} \ \text{ is a } \cG_\Gamma\text{-nest}.
\end{equation}
\end{prop}

\medskip
\subsection{Motives of wonderful compactifications}\label{sec_motives}

As in the case of the wonderful compactifications $\overline{\rm Conf}_\Gamma(X)$
analyzed in \cite{CeyMar}, one can obtain the explicit formula for the  motive of the compactifications  $Z[\Gamma]$ directly from the  formula  for the  motive of blow-ups
and the iterated construction of \S \ref{blowupSec}.

We first introduce the following notation as in \cite{CeyMar}, \cite{li}. 
Given a $\cG_\Gamma$-nest $\cN$, and a
biconnected induced subgraph $\gamma$ such that $\cN' =\cN \cup \{ \gamma \}$
is still a $\cG_\Gamma$-nest, we set
\begin{equation}\label{rgamma}
r_\gamma= r_{\gamma,\cN} := \dim (\cap_{\gamma'\in \cN: \gamma'\subset \gamma} \Delta_{\gamma'})-\dim \Delta_\gamma ,
\end{equation}
\begin{equation}\label{MN}
M_\cN:=\{ (\mu_\gamma)_{\Delta_\gamma \in \cG_\Gamma} \,:\, 1\leq \mu_\gamma \leq r_\gamma -1 , \,\, \mu_\gamma \in \Z \}, 
\end{equation}
\begin{equation}\label{normmu}
\| \mu \| := \sum_{\Delta_\gamma \in \cG_\Gamma} \mu_\gamma .
\end{equation}

We write here $\m(X)$ for the motive in the
Voevodsky category. This corresponds to the notation $M_{gm}$ of \cite{voe}.

The following result is the analog of Proposition 8 of \cite{CeyMar}, see also
\cite{li} for the formulation in the case of Chow motives.

\begin{prop}  
\label{propMTvoe}
Let $X$ be a smooth projective variety. 
The Voevodsky motive $\m(Z[\Gamma])$ of the wonderful compactification 
is given by
\begin{equation}\label{voeConfGX}
\m(Z[\Gamma]) = \m(Z^{\BV_\Gamma}) \oplus \bigoplus_{\cN \in \cG_\Gamma\text{-nests}}
\,\,\,\,  \bigoplus_{\mu \in M_\cN} \m(X^{\BV_{\Gamma/\delta_{\cN}(\Gamma)}} \times X^{\BV_\Gamma})(\|\mu\|)[2\|\mu\|]
\end{equation}
where $\Gamma/\delta_\cN(\Gamma)$ is the quotient 
$
\Gamma/\delta_{\cN}(\Gamma):=\Gamma // (\gamma_1 \cup \cdots \cup \gamma_r)
$
for the  $\cG_\Gamma$-nest $\cN =\{ \gamma_1, \ldots, \gamma_r \}$.
\end{prop}

\proof
Let $\widetilde{Y} \to Y$ be the blow-up of a smooth scheme $Y$ along a smooth 
closed subscheme $V \subset Y$. Then $\m(\widetilde{Y})$ is canonically isomorphic
to
\begin{equation*}
\m(Y) \oplus  \bigoplus_{k=1}^{\rm{codim}_Y(V)-1} \m(V)(k) [2k],
\end{equation*}
see Proposition 3.5.3 in \cite{voe}. 
The result then follows by applying this blow-up formula for Voevodsky's motives 
to the iterated  blow-up construction given in Section \ref{blowupSec}. 
\endproof

We obtain then from Proposition \ref{propMTvoe} the following simple corollary
(see \S 3.2 of \cite{CeyMar}).

\begin{cor}\label{confMT}
If the motive of the smooth projective variety $X$ is mixed Tate,
then the motive of $Z[\Gamma]$ is  mixed Tate, for all graphs
$\Gamma$. Moreover, the exceptional divisors $D_\gamma^{(Z)}$ associated to
the biconnected induced subgraphs $\gamma \subseteq \Gamma$ and
the intersections $D_{\gamma_1}^{(Z)}\cap \cdots \cap D_{\gamma_\ell}^{(Z)}$ associated 
to the $\cG_\Gamma$-nests $\{\gamma_1,\dots,\gamma_\ell\}$   are also mixed Tate.
\end{cor}

\proof
This is an immediate consequence of the construction of $Z[\Gamma]$  
since the motive of $Z[\Gamma]$ depends upon the motive of $X$ only through
products, Tate twists, sums, and shifts. All these operations preserve the subcategory
of mixed Tate motives. The reason why the intersections $D_{\gamma_1}^{(Z)}\cap \cdots 
\cap D_{\gamma_\ell}^{(Z)}$ are also mixed Tate is because one has an explicit stratification,
as described in \cite{CeyMar} and \cite{li2}, where one has a description of the
intersections of diagonals in terms of configuration spaces of quotient graphs and
by repeated use of the blowup formula for motives.  
\endproof

\begin{rem}\label{DefZrem} {\rm One can also see easily that, if the variety $X$ is defined
over $\Z$, then so is $Z[\Gamma]$ and so are the $D_\gamma^{(Z)}$ and
their unions and intersections. Moreover, all these varieties then satisfy
the unramified criterion of \S 3.5 and Proposition 3.10 of \cite{Gon}.
}\end{rem}

\medskip
\subsection{Feynman amplitude and wonderful compactifications}

We now consider the form $\omega^{(Z)}_\Gamma$ defined as in 
\eqref{amplitudeZ} and discuss its behavior when puled back from
$Z^{\BV_\Gamma}$ to the wonderful compactification $Z[\Gamma]$.

\subsubsection{Loci of divergence}
For massless scalar Euclidean field theories, the pole locus $\{\omega^{(Z)}_\Gamma = \infty\}$ in 
$Z^{\BV_\Gamma}$ is  
\begin{eqnarray}
\cZ_\Gamma := \left\{\prod_{e \in \BE_\Gamma}   \| p(z_{v_{s(e)}}) - p(z_{v_{t(e)}}) \|^2 = 0 \right\}.
\label{eqn_g_hyp}
\end{eqnarray}
This definition can be rephrased as follows.

\begin{lem}\label{divMGamma}
The divergent locus of the density $\omega^{(Z)}_\Gamma$ of \eqref{amplitudeZ} in 
$Z^{\BV_\Gamma}$ is given by the union  $\bigcup_{e \in \BE_\Gamma} \Delta_e^{(Z)}$.
\end{lem}

\medskip
\subsubsection{Order of singularities in the blowups} 

Let $\pi_\gamma^*(\omega_\Gamma^{(Z)})$ denote the pullbacks of the form
$\omega_\Gamma^{(Z)}$ of \eqref{amplitudeZ} to the iterated blowups of $Z^{\BV_\Gamma}$
along the (dominant transforms of) the diagonals $\Delta_\gamma^{(Z)}$, for $\gamma\subset
\Gamma$ a biconnected induced subgraph.
  
\begin{prop}\label{prop_form}
Let $\Gamma$ be a connected graph. Then for every biconnected induced subgraph 
$\gamma \subset \Gamma$, the pullback $\pi_\gamma^*(\omega_\Gamma^{(Z)})$ 
of $\omega_\Gamma^{(Z)}$ to the blowup along the (dominant transform of) 
$\Delta_\gamma^{(Z)}$ has singularities of  order
\begin{equation}\label{ordmore1}
\begin{array}{rl}
{\rm ord}_\infty (\pi_\gamma^*(\omega_\Gamma), D_\gamma^{(Z)}) = &
(D-2) |\BE_\gamma| - 2 D (|\BV_\gamma| -1) +2   \\[2mm]
= & 2D  b_1(\gamma) - (D+2)  |\BE_\gamma| +2
\end{array}
\end{equation}
along the exceptional divisors $D_\gamma^{(Z)}$ in the blowup. Here $b_1(\gamma)$ denotes
the first Betti number of graph $\gamma$.
\end{prop} 

\proof Let $m = D |\BV_\Gamma|$ and $L \subset \A^{2m}$ be the coordinate subspace 
given by the equations $\{x_1 = \cdots = x_k =0\}$, and $\pi: \widetilde\A^{2m} \to \A^{2m}$ 
be  the blowup along $L \subset \A^{2m}$. If one chooses the coordinates
$w_i$ in the blow up, such that,$w_i =x_i$ for $i = k, \dots, 2m$, and
$w_i w_p = x_i$ for $i<k$. The exceptional divisor given by $w_p=0$ in these coordinates. 
Then, one obtains 
\begin{eqnarray*} 
\pi^* (dx_1 \wedge dx_1^* \wedge \cdots \wedge d x_m \wedge dx_m^*)
= 
|w|^{2 (k-1)}dw_1 \wedge dw_1^* \wedge \cdots \wedge d w_d \wedge dw_D^*.
\end{eqnarray*}
This form has a zero of order $2 \cdot ({\rm codim}(L) -1)$ along the exceptional divisor 
of the blowup.

The codimension of  the diagonal $\Delta_\gamma \subset X^{\BV_\Gamma}$ associated to a 
connected subgraph $\gamma \subset \Gamma$ is $D\ (|\BV_\gamma|-1)$. On the other hand, 
the form $\omega_\Gamma^{(Z)}$ has singularity along $\Delta_\gamma^{(Z)}$ 
of order  $(D-2) |\BE_\gamma|$. Hence, 
\begin{eqnarray}
{\rm ord}_\infty(\pi^*_\gamma(\omega_\Gamma^{(Z)}),D_\gamma^{(Z)}) 
&=&  (\text{order of}\ \infty) -
 (\text{order of zeros}) \nonumber \\
&=& (D-2) |\BE_\gamma| - 2 D (|\BV_\gamma| -1) +2. \nonumber
\end{eqnarray}
\endproof

Note that the orders of pole are different from the case of the form
$\omega_\Gamma$ on $\overline{\rm Conf}_\Gamma(X)$, see \S 4.3
of \cite{CeyMar}. Lemma \ref{divMGamma} and Proposition \ref{prop_form} then
imply the following.

\begin{cor}\label{divZGamma}
Let $\pi_\Gamma^*(\omega_\Gamma^{(Z)})$ denote the pullback of
$\omega_\Gamma^{(Z)}$ to the wonderful compactification
$Z[\Gamma]$.  The divergence locus of $\pi_\Gamma^*(\omega_\Gamma^{(Z)})$ in
$Z[\Gamma]$ is given by the union of divisors 
\begin{equation}\label{DivDZ}
 \bigcup_{\Delta_\gamma^{(Z)} \in \cG_\Gamma} D_\gamma^{(Z)}. 
\end{equation}
\end{cor}

\medskip
\subsection{Chain of integration and divergence locus}\label{sigmadivSec}

When pulling back the form $\omega_\Gamma^{(Z)}$ along the
projection $\pi_\Gamma: Z[\Gamma] \to Z^{\BV_\Gamma}$, one
also replaces the chain of integration $\sigma^{(Z,y)}_\Gamma =
X^{\BV_\Gamma}\times \{ y \}$ of \eqref{sigmaZ} with 
$\tilde\sigma^{(Z,y)}_\Gamma\subset Z[\Gamma]$ with
$\pi_\Gamma(\tilde\sigma^{(Z,y)}_\Gamma)=\sigma^{(Z,y)}_\Gamma$,
which gives
\begin{equation}\label{tildesigmaZ}
\tilde\sigma^{(Z,y)}_\Gamma = \overline{\rm Conf}_\Gamma(X) \times \{ y \} \subset Z[\Gamma]=
\overline{\rm Conf}_\Gamma(X) \times X^{\BV_\Gamma}.
\end{equation}

\begin{lem}\label{chaindivLem}
The chain of integration $\tilde\sigma^{(Z,y)}_\Gamma$ of \eqref{tildesigmaZ}
intersects the locus of divergence \eqref{DivDZ} in 
\begin{equation}\label{chaindiv}
 \bigcup_{\Delta_\gamma^{(Z)} \in \cG_\Gamma} D_\gamma \times \{ y \} \subset 
 \overline{\rm Conf}_\Gamma(X) \times \{ y \}.
\end{equation}
\end{lem}

\proof This follows directly from Corollary \ref{divZGamma} and \eqref{tildesigmaZ}.
\endproof

Notice that, since $\cG_\Gamma$-factors intersect transversely (see Proposition 2.8 of
\cite{li2} and Proposition 4 of \cite{CeyMar}), the intersection \eqref{chaindiv}
of the chain of integration $\tilde\sigma^{(Z,y)}_\Gamma$ with the locus of divergence
consists of a union of divisors $D_\gamma$ inside $X^{\BV_\Gamma}$ intersecting
transversely, with $D_{\gamma_1}\cap \cdots \cap D_{\gamma_\ell}\neq
\emptyset$ whenever $\{ \gamma_1,\ldots, \gamma_\ell \}$ form a $\cG_\Gamma$-nest
(see \cite{CeyMar} and \cite{li2}).

\medskip
\subsection{Smooth and algebraic differential forms}\label{algdiffSec}

Consider the restriction of the amplitude $\pi^*_\Gamma(\omega^{(Z)}_\Gamma)$
to the chain $\tilde\sigma^{(Z,y)}_\Gamma$. It is defined on the complement
of the divergence locus, namely on 
\begin{equation}\label{divcomplement}
\tilde\sigma^{(Z,y)}_\Gamma \smallsetminus \left( \bigcup_{\Delta_\gamma^{(Z)} \in \cG_\Gamma} D_\gamma \times \{ y \}\right) 
\simeq \overline{\rm Conf}_\Gamma(X) \smallsetminus \left( \bigcup_{\Delta_\gamma^{(Z)} \in \cG_\Gamma} D_\gamma\right) ,
\end{equation}
which is a copy of ${\rm Conf}_\Gamma(X)$ inside $Z[\Gamma]$. The form
$\pi^*_\Gamma(\omega^{(Z)}_\Gamma)$ is a closed form of top dimension
on this domain.

\smallskip

We recall the following general fact. 
Let $\cX$ be a smooth projective variety of dimension $m$ and let $\cD$ be 
a union of smooth hypersurfaces intersecting transversely 
(strict normal crossings divisor). Let $\cU=\cX\smallsetminus \cD$.

\begin{lem}\label{algformLem}
Let $\omega$ be a $\cC^\infty$ closed differential form of degree $m$ on $\cU$, 
and let $[\omega]$ be the corresponding de Rham cohomology class in $H^m(\cU)$.
Then there exists an algebraic differential form $\eta$ with logarithmic poles along
$\cD$, that is cohomologous, $[\eta]=[\omega]\in H^m(\cU)$, to the given form $\omega$.
\end{lem}

\proof First we use the fact that de Rham cohomology of a smooth quasi-projective 
variety $\cU=\cX\smallsetminus \cD$
can always be computed using algebraic differential forms, \cite{Gro},
\cite{Hart}. Thus, the cohomology class $[\omega]$ in $H^m(\cU)$
can be realized by a form $\alpha \in H^0(\cU, \Omega^m)$,
with $\Omega^m$ the sheaf of algebraic differential forms.  Moreover,
by \S 3.2 of  \cite{DeII}, the algebraic de Rham cohomology $H^*(\cU)$
satisfies
 \begin{equation}\label{logDcohom}
 H^*(\cU) \simeq \H^*(\cX, \Omega_{\cX}^*(\log(\cD))),
 \end{equation}
where $\Omega_{\cX}^*(\log(\cD))$ denotes the sheaf  of meromorphic 
differential forms on $\cX$ with logarithmic poles along $\cD$. Thus, we
can find a form $\eta \in \Omega_{\cX}^*(\log(\cD))$ so that
$[\eta] =[\omega]\in H^m(\cU)$.
\endproof

\smallskip

We then have the following consequence.

\begin{lem}\label{amplitudeZalgDef}
Let $\pi^*_\Gamma(\omega^{(Z)}_\Gamma)$ be the pullback of the Feynman
amplitude \eqref{amplitudeZ} to the wonderful compactification $Z[\Gamma]$.
Then there exists a meromorphic differential form $\eta_\Gamma^{(Z)}$,
the {\em algebraic Feynman amplitude}, on
$\tilde\sigma_\Gamma^{(Z,y)}= \overline{\rm Conf}_\Gamma(X)\times \{ y \}$, with
logarithmic poles along 
$$ \cD_\Gamma = \bigcup_{\Delta_\gamma^{(Z)} \in \cG_\Gamma} D_\gamma \times \{ y \}, $$
such that 
$$ [\eta_\Gamma^{(Z)}]=[\pi^*_\Gamma(\omega^{(Z)}_\Gamma)] \in H^{2D|\BV_\Gamma|}
(\tilde\sigma_\Gamma^{(Z)}\smallsetminus \cD_\Gamma). $$
\end{lem}

\proof This follows directly from Lemma \ref{algformLem} above.
\endproof

\medskip
\subsection{Iterated Poincar\'e residues}\label{iterResSec}

One can associate to the holomorphic differential form $\eta_\Gamma$ with
logarithmic poles along $\cD_\Gamma$ a Poincar\'e residue on each
non-empty intersection of a collection of divisors $D_\gamma^{(Z)}$ that
corresponds to a $\cG_\Gamma$-nest $\cN=\{ \gamma_1, \ldots, \gamma_r \}$.

\medskip

\begin{prop}\label{residuenest}
For every $\cG_\Gamma$-nest $\cN=\{ \gamma_1, \ldots, \gamma_r \}$,
there is a Poincar\'e residue $\cR_{\cN}(\eta_\Gamma)$, which defines
a cohomology class in $H^{2D|\BV_\Gamma|-r}(V_\cN)$, 
on the complete intersection 
$V_\cN^{(Z)}=D_{\gamma_1}^{(Z)} \cap \cdots \cap D^{(Z)}_{\gamma_r}$.
The pairing of $\cR_{\cN}(\eta_\Gamma)$ with 
an $(2D|\BV_\Gamma|-r)$-cycle $\Sigma_\cN$ in $V_\cN^{(Z)}$ is equal to
\begin{equation}\label{residuepair}
\int_{\Sigma_\cN} \cR_{\cN}(\eta_\Gamma) = \frac{1}{(2\pi i)^r} \int_{\cL_\cN(\Sigma_\cN)}
\eta_\Gamma,
\end{equation}
where $\cL_\cN(\Sigma_\cN)$ is the $2D|\BV_\Gamma|$-cycle in $Z[\Gamma]$ given by
an iterated Leray coboundary of $\Sigma_\cN$, which is a $T^r$-torus
bundle over $\Sigma_\cN$. Under
the assumption that the variety $X$ is a mixed Tate motive, the 
integrals \eqref{residuepair} are periods of mixed Tate motives.
\end{prop}

\proof
As shown in Proposition \ref{strataConf}, the divisors $D_\gamma^{(Z)}$ in
$Z[\Gamma]$ have the property that
\begin{equation}\label{nonemptyintersD}
V_\cN^{(Z)}=D_{\gamma_1}^{(Z)} \cap \cdots \cap D_{\gamma_r}^{(Z)} \neq \emptyset
\Leftrightarrow \{ \gamma_1,\ldots, \gamma_r \}
\ \ \text{ is a } \cG_\Gamma-{\text nest} ,
\end{equation}
with transverse intersections.  

Consider the first divisor $D_{\gamma_1}^{(Z)}$ in the $\cG_\Gamma$-nest $\cN$,
and a tubular neighborhood $N_{\Gamma,\gamma_1}=N_{Z[\Gamma]}(D_{\gamma_1}^{(Z)})$ 
of $D_{\gamma_1}^{(Z)}$  in $N_{\Gamma,\gamma_1}$.
This is a unit disk bundle over $D_{\gamma_1}^{(Z)}$ with projection $\pi:N_{\Gamma,\gamma_1}
\to D_{\gamma_1}^{(Z)}$ and with $\sigma: D_{\gamma_1}^{(Z)} \hookrightarrow
N_{\Gamma,\gamma_1}$ the zero section. The Gysin long exact sequence in homology 
gives
$$ \cdots \to 
H_k(N_{\Gamma,\gamma_1}\smallsetminus D_{\gamma_1}^{(Z)}) \stackrel{\iota_*}{\to}
H_k(N_{\Gamma,\gamma_1}) $$
$$ \stackrel{\sigma^{!}}{\to} H_{k-2}(D_{\gamma_1}^{(Z)})
\stackrel{\pi^{!}}{\to} H_{k-1}(N_{\Gamma,\gamma_1}\smallsetminus D_{\gamma_1}^{(Z)})
\to\cdots $$
where $\pi^{!}$ is the Leray coboundary map, which assigns to a chain $\Sigma$ in
$D_{\gamma_1}^{(Z)}$ the homology class in 
$N_{\Gamma,\gamma_1}\smallsetminus D_{\gamma_1}^{(Z)}$ of the
boundary $\partial \pi^{-1}(\Sigma)$ of the disk bundle $\pi^{-1}(\Sigma)$
over $\Sigma$, which is an $S^1$-bundle over $\Sigma$. Its dual is a morphism 
$$ \cR_{\gamma_1} : H^{k+1}(N_{\Gamma,\gamma_1}\smallsetminus D_{\gamma_1}^{(Z)}) \to 
H^k(D_{\gamma_1}^{(Z)}), $$
which is the residue map. The iterated residue map is obtained by considering
the complements $\cU_0= N_{\Gamma,\gamma_1}\smallsetminus D_{\gamma_1}^{(Z)}$ and
$$ \cU_1 = D_{\gamma_1}^{(Z)} \smallsetminus \bigcup_{1< k \leq r} D_{\gamma_k}^{(Z)} , $$
$$ \cU_2 = (D_{\gamma_1}^{(Z)} \cap D_{\gamma_2}^{(Z)}) \smallsetminus \bigcup_{2< k \leq r} 
D_{\gamma_k}^{(Z)}, $$
and so on. One obtains a sequence of maps
$$  H^k(\cU_0) \stackrel{\cR_{\gamma_1}}{\to}  H^{k-1}(\cU_1) \stackrel{\cR_{\gamma_2}}{\to}
H^{k-2}(\cU_2) \to\cdots \stackrel{\cR_{\gamma_r}}{\to} H^{k-r}(V_{\cN}^{(Z)}) . $$ 
The composition $\cR_{\cN}=\cR_{\gamma_r} \circ\cdots \circ \cR_{\gamma_1}$ is the
iterated residue map.  Because the residue map is dual to Leray coboundary, under
the pairing of homology and cohomology one obtains
$$ \langle \cR_{\cN}(\eta), \Sigma \rangle = \langle \eta, \cL_{\cN}(\Sigma) \rangle, $$
where $\cL_{\cN}= \cL_{\gamma_1} \circ\cdots \circ \cL_{\gamma_r}$ is the compositions
of the Leray coboundary maps of the divisors $D_{\gamma_k}^{(Z)}$. The resulting
$\cL_{\cN}(\Sigma)$ is therefore, by construction, a $T^r$-torus bundle over $\Sigma$.
At the level of differential forms, the residue map $\cR_{\gamma_1}$ is given by integration
along the circle fibers of the $S^1$-bundle $\partial\pi^{-1}(\Sigma) \to \Sigma$. Thus, the
pairing $\langle \cR_{\cN}(\eta), \Sigma \rangle$ contains a $2\pi i$ factor, coming from
the integration of a form $df/f$, with $f$ the local defining equation of the hypersurface,
along the circle fibers. This means that, when writing the pairings in terms of differential
forms, one obtains \eqref{residuepair}.
As shown in \cite{CeyMar}, if $\m(X)$ is mixed Tate, 
the divisors $D_\gamma$ and their intersections in
$\overline{\rm Conf}_\Gamma(X)$ are mixed Tate motives, and so are
the $D_\gamma^{(Z)}$ and their intersections $V_{\cN}^{(Z)}$ in $Z[\Gamma]$.
\endproof

\medskip

\begin{cor}\label{cycleVN}
Given a $\cG_\Gamma$-nest $\{ \gamma_1,\ldots,\gamma_r \}$ let
$V_\cN= D_{\gamma_1} \cap \cdots \cap D_{\gamma_r}$ be the
intersection of the corresponding divisors in $\overline{\rm Conf}_\Gamma(X)$.
The residues $\cR_{\cN}(\eta_\Gamma)$ of Proposition \ref{residuenest} 
pair with the $2D|\BV_\Gamma|-r$-dimensional cycles in $V_{\cN}^{(Z)}$
given by $V_{\cN}\times \{ y \}$,
\begin{equation}\label{resVN}
\langle\cR_{\cN}(\eta_\Gamma), V_{\cN} \rangle =\int_{V_\cN \times\{ y \}}
 \cR_{\cN}(\eta_\Gamma) .
\end{equation}
\end{cor}

\proof 
The chain of integration $\tilde\sigma^{(Z,y)}_\Gamma =\overline{\rm Conf}_\Gamma(X)\times
\{ y \}$ intersects the loci $V_\cN^{(Z)}=V_{\cN}\times X^{\BV_\Gamma}$ along $V_{\cN}\times \times \{y \}$, where the 
$V_\cN$ are the intersections $V_{\cN}=D_{\gamma_1}\cap\cdots\cap D_{\gamma_r}$
of the divisors $D_{\gamma_k}$ in $\overline{\rm Conf}_\Gamma(X)$. Thus,
$\Sigma= V_{\cN}\times \{y \}$ defines a $2D|\BV_\Gamma|-r$-dimensional cycle
in $V_\cN^{(Z)}$, which can be paired with the form $\cR_{\cN}(\eta_\Gamma)$ of
degree $2D|\BV_\Gamma|-r$ on  $V_\cN^{(Z)}$.
\endproof

\medskip

\begin{rem}{\rm The reason for passing to the wonderful compactification $Z[\Gamma]$
and pulling back the form $\omega_\Gamma^{(Z)}$ along the projection 
$\pi_\Gamma: Z[\Gamma]\to Z^{\BV_\Gamma}$ is in order to pass to a setting where the
locus of divergence is described by divisors $D_\gamma$ intersecting transversely
in $\overline{\rm Conf}_\Gamma(X)$, while the intersections of the diagonals 
$\Delta_\gamma$ in $X^{\BV_\Gamma}$ can be non-transverse. These transversality
issues are discussed in more detail in \cite{CeyMar}. There is a generalization of the theory
of forms with logarithmic poles and Poincar\'e residues \cite{Saito}, that extends the case
of \cite{Deligne} of normal crossings divisors, but in this more general setting the Poincar\'e
residue gives meromorphic instead of holomorphic forms.}
\end{rem}

\bigskip
\section{Regularization and integration}\label{cutoffreg}

In this section, we describe a regularization of the Feynman
integral 
\begin{equation}\label{intetaG}
 \int_{\tilde\sigma^{(Z,y)}_\Gamma\smallsetminus \cD_\Gamma} \eta_\Gamma^{(Z)} 
\end{equation} 
in distributional terms, using the theory of principal value and residue currents.
This will show that one can express ambiguities in the regularization in terms of the
iterated residues along the intersections of the divisors $D_\gamma^{(Z)}$, described
in \S \ref{iterResSec} above.

\medskip
\subsection{Current-regularized Feynman amplitudes}

We review briefly some well known facts about residue and principal value 
currents and we apply them to the Feynman amplitude regularization.

\medskip
\subsubsection{Residue currents and Mellin transforms}
Recall that, for a single smooth hypersurface defined by an equation $\{ f=0 \}$,
the residue current $[ Z_f ]$, supported on the hypersurface, is defined as
$$ [ Z_f ] = \frac{1}{2\pi i} \bar\partial [\frac{1}{f}] \wedge df :=\frac{1}{2\pi i} \bar\partial \partial
\log | f|^2. $$
This is known as the Poincar\'e--Lelong formula. It can also be seen as a limit
$$ \int_{Z_f} \varphi  =\lim_{\epsilon \to 0} \frac{1}{2\pi i} 
\int_{|f|=\epsilon} \frac{df}{f} \wedge \varphi. $$
A generalization is given by the Coleff--Herrera residue current \cite{CoHe}, associated to
a collection of functions $\{ f_1, \ldots, f_r \}$. Under the assumption that these
define a complete intersection $V=\{ f_1 =\cdots = f_r =0 \}$, the residue current
\begin{equation}\label{cRf}
  \cR_f = \bar\partial [\frac{1}{f_1}]\wedge \cdots \wedge \bar\partial [\frac{1}{f_r}] 
\end{equation}  
is obtained as a limit
$$ \cR_f (\varphi) = \lim_{\delta \to 0} \int_{T_{\epsilon(\delta)}(f)} \frac{\varphi}{f_1 \cdots f_r}, $$
with $T_{\epsilon(\delta)}(f)=\{ |f_k| = \epsilon_k(\delta) \}$, with the limit taken 
over ``admissible paths" $\epsilon(\delta)$, which satisfy the properties
$$ \lim_{\delta \to 0} \epsilon_r(\delta) =0, \ \ \  \lim_{\delta \to 0} \frac{\epsilon_k(\delta)}{(\epsilon_{k+1}(\delta))^\ell} =0, $$
for $k=1,\ldots,r$ and any positive integer $\ell$. The test form $\varphi$ is a $(2n-r)$-form of
type $(n,n-r)$, where $2n$ is the real dimension of the ambient space, and the 
residual current obtained in this way is a $(0,r)$-current. For more details, 
see \S 3 of \cite{BGVY} and \cite{TsYg}. Notice that, while in general one cannot take
products of distributions, the Coleff--Herrera product \eqref{cRf} is well defined for
residue currents, as well as between residue and principal value currents.

\smallskip

Moreover,  the Mellin transform
\begin{equation}\label{MellinRes}
\Gamma^\varphi_f (\lambda)= \int_{\R^r_+} \cI^\varphi_f(\epsilon) \, 
\epsilon^{\lambda-I} \, d\epsilon,
\end{equation}
with
$$ \cI^\varphi_f(\epsilon) = \int_{T_{\epsilon}(f)} \frac{\varphi}{f_1 \cdots f_r} $$
and with
$$ \epsilon^{\lambda-I} \, d\epsilon = \epsilon_1^{\lambda_1-1}\cdots \epsilon_r^{\lambda_r -1}\,
d\epsilon_1 \wedge \cdots \wedge d\epsilon_r, $$
can also be written, as in \cite{BGVY}, \cite{TsYg}, as
\begin{equation}\label{MellinResIntX}
\Gamma^\varphi_f (\lambda)= \frac{1}{(2\pi i)^r} \int_{\cX} |f|^{2(\lambda -I)} 
\overline{df} \wedge \varphi ,
\end{equation}
where the integration is on the ambient variety $\cX$ and where
$$ |f|^{2(\lambda -I)} = |f_1|^{2(\lambda_1-1)} \cdots |f_r|^{2(\lambda_r-1)}, \ \ \ \text{ and }
\ \ \  \overline{df}= \overline{df_1}\wedge \cdots \wedge \overline{df_r}. $$
When $\{ f_1, \ldots, f_r \}$ define a complete intersection, the function
$\lambda_1 \cdots \lambda_r \Gamma^\varphi_f (\lambda)$ is holomorphic in
a neighborhood of $\lambda=0$ and the value at zero
is given by the residue current (\cite{BGVY}, \cite{TsYg})
\begin{equation}\label{MellinResZero}
 \cR_f(\varphi)= \lambda_1 \cdots \lambda_r \Gamma^\varphi_f (\lambda)|_{\lambda=0}.
\end{equation} 
Equivalently, \eqref{MellinResIntX} and \eqref{MellinResZero} can also be written as
\begin{equation}\label{MellinResZero2}
\lim_{\lambda \to 0} \lambda \Gamma^\varphi_f (\lambda)=
\lim_{\lambda \to 0} \frac{1}{(2\pi i)^r} \int_{\cX} 
\frac{\bar\partial |f_r|^{2\lambda_r} \wedge \cdots \wedge \bar\partial 
|f_1|^{2\lambda_1}}{f_r \cdots f_1} \wedge \varphi,
\end{equation}
where the factor $\lambda$ on the left-hand-side stands for
$\lambda_1 \cdots \lambda_r$ as in \eqref{MellinResZero}.

The Poincar\'e--Lelong formula, in this more general case of a complete
intersection defined by a collection $\{ f_1, \ldots, f_r \}$, expresses the
integration current $Z_f$ as
\begin{equation}\label{PoiLel}
[Z_f] = \frac{1}{(2\pi i)^r} \bar\partial [\frac{1}{f_r}] \wedge \cdots \wedge \bar\partial [\frac{1}{f_1}]
\wedge df_1 \wedge \cdots \wedge df_r .
\end{equation}

The correspondence between residue currents and the Poincar\'e residues on
complete intersections, discussed above in \S \ref{iterResSec}, is given  
for instance in Theorem 4.1 of \cite{AT}.

\medskip
\subsubsection{Principal value current}

The principal value current $[1/f]$ of a single holomorphic function $f$ can be
computed as \cite{HeLieb}, \cite{Schw}
\begin{equation}\label{PVcurr}
\langle [\frac{1}{f}] , \phi\rangle 
= \lim_{\epsilon \to 0} \int_{|f|>\epsilon} \frac{\phi \, d\zeta \wedge d\bar\zeta}{f},
\end{equation}
where $\phi$ is a test function. More generally, for $\{ f_1, \ldots, f_r \}$ as above,
the principal value current is given by
\begin{equation}\label{PVcurr2}
\langle [\frac{1}{f}] , \phi\rangle 
= \lim_{\epsilon \to 0} \int_{N_\epsilon(f)} \frac{\phi}{f_r \cdots f_1},
\end{equation}
with $\phi$ a test form and with
\begin{equation}\label{Nf}
N_\epsilon(f)=\{ |f_k|> \epsilon_k \}.
\end{equation}

More generally we will use the following notation.

\begin{defn}\label{PVetaDef}
Given a meromorphic $(p,q)$-form $\eta$ on an $m$-dimensional 
smooth projective variety $\cX$, with poles 
along an effective divisor $\cD=D_1\cup\cdots D_r$, 
where the components $D_k$ are smooth hypersurfaces defined
by equations $f_k=0$, the principal value current $PV(\eta)$ is
defined by
\begin{equation}\label{PVeta}
\langle PV(\eta), \phi \rangle = \lim_{\epsilon \to 0} \int_{N_\epsilon(f)} \eta\wedge \phi, 
\end{equation}
for an $(m-p,m-q)$ test form $\phi$, with $N_\epsilon(f)$ defined as in \eqref{Nf}.
\end{defn}

\smallskip

The following simple Lemma describes the source of ambiguities and 
its relation to residues.

\begin{lem}\label{lemPVandResf}
When the test form $\phi$ is modified to $\phi+\bar\partial\psi$, the principal
value current satisfies
\begin{equation}\label{PVandResf}
\langle [\frac{1}{f}] , \phi +\bar\partial \psi\rangle =\langle [\frac{1}{f}] , \phi \rangle -
\langle \bar\partial [\frac{1}{f}] ,\psi \rangle,
\end{equation}
where $\bar\partial [1/f]$ is the residue current $\cR_f$ of \eqref{cRf}.
\end{lem}

\proof By Stokes theorem, we have
$$ \langle [\frac{1}{f}] , \bar\partial\psi \rangle = \lim_{\epsilon \to 0} \int_{|f|> \epsilon}
\frac{\bar\partial \psi}{f} = \lim_{\epsilon \to 0} - \int_{|f|=\epsilon} \frac{\psi}{f} = 
- \langle \bar\partial [\frac{1}{f}] ,\psi \rangle. $$
\endproof

We now return to the case of the Feynman amplitudes and describe the
corresponding regularization and ambiguities.

\medskip
\subsubsection{Principal value and Feynman amplitude}

We can regularize the Feynman amplitude given by the integral \eqref{intetaG},
interpreted in the distributional sense, as in \S \ref{distrSec}, using the principal
value current.

\begin{defn}\label{regPVintDef}
The principal value regularization of the Feynman amplitude \eqref{intetaG}
is given by the current $PV(\eta^{(Z)}_\Gamma)$ defined as in \eqref{PVeta},
$$ \langle PV(\eta^{(Z)}_\Gamma), \varphi \rangle = \lim_{\epsilon \to 0} \int_{N_\epsilon(f)} \varphi \, \,  \eta , $$
for a test function $\varphi$.
\end{defn}

We can also write the regularized integral in the following form.

\begin{lem}
The regularized integral satisfies
$$  \langle PV(\eta^{(Z)}_\Gamma), \varphi \rangle = \lim_{\lambda\to 0}
\int_{\tilde\sigma^{(Z,y)}_\Gamma} |f_n|^{2\lambda_n} \cdots |f_1|^{2\lambda_1} \eta^{(Z,y)}_\Gamma\,\, \varphi  $$
where $n=n_\Gamma$ is the cardinality $n_\Gamma= \# \cG_\Gamma$ of the building
set $\cG_\Gamma$ and the $f_k$ are the defining equations of the $D_\gamma^{(Z)}$
in $\cG_\Gamma$
\end{lem}

\proof The form $\eta^{(Z)}_\Gamma$ has poles along the divisor $\cD_\Gamma=\cup_{\Delta_\gamma^{(Z)} \in \cG_\Gamma} D_\gamma^{(Z)}$. Thus, if we denote by $f_k$, with
$k=1,\ldots, n$ with $n= \# \cG_\Gamma$ the defining equations of the $D_\gamma^{(Z)}$,
we can write the principal value current in the form
$$  \lim_{\lambda\to 0}
\int_{\tilde\sigma^{(Z,y)}_\Gamma} \frac{|f_n|^{2\lambda_n} \cdots |f_1|^{2\lambda_1}}{f_n\cdots
f_1} h \, \varphi, $$
with $h$ an algebraic from without poles and $\varphi$ is a test function.
\endproof

\medskip
\subsubsection{Pseudomeromorphic currents}

If $\{ f_1 , \ldots, f_n \}$ define a complete intersection $V=\{ f_1=\cdots =f_r=0 \}$ in
a smooth projective variety $\cX$, an {\em elementary pseudomeromorphic current} 
is a current of the form
\begin{equation}\label{pseudomero}
C_{r,n}:= [\frac{1}{f_n}] \wedge \cdots \wedge [\frac{1}{f_{r+1}}] \wedge
\bar\partial[\frac{1}{f_r}] \wedge \cdots \wedge \bar\partial[\frac{1}{f_1}] , 
\end{equation} 
for some $1\leq r \leq n$,
where the products of principal value and residue currents are well defined
Coleff--Herrera products and the resulting current is commuting in the principal
value factors and anticommuting in the residue factors.
These distributions also have a Mellin transform formulation (see \cite{BGVY}) as
\begin{equation}\label{pseudomeroMell}
\langle C_{r,n}, \phi \rangle = \lim_{\lambda \to 0} \frac{1}{(2\pi i)^r} \int_{\cX} \prod_{k=r+1}^n
\frac{|f_k|^{2\lambda_k}}{f_k} \,\, \bigwedge_{j=1}^r \bar\partial\left( \frac{|f_j|^{2\lambda_j}}{f_j} 
\right) \wedge \phi.
\end{equation}

\medskip
\subsection{Ambiguities of regularized Feynman integrals}

We can use the formalism of residue currents recalled above to describe
the ambiguities in the principal value regularization of Feynman amplitudes of
Definition \ref{regPVintDef}.

\medskip
\subsubsection{Feynman amplitude and residue currents}

As in \S \ref{iterResSec}, consider a $\cG_\Gamma$-nest $\{ \gamma_1,\ldots,
\gamma_r \}$ and the associated intersection $V_{\cN}^{(Z)}=D_{\gamma_1}^{(Z)}\cap 
\cdots \cap D_{\gamma_r}^{(Z)}$. Also let $n_\Gamma = \#\cG_\Gamma$ and let
$f_k$, for $k=1,\ldots, n_\Gamma$ be the defining equations for the $D_\gamma^{(Z)}$,
for $\gamma$ ranging over the subgraphs defining the building set $\cG_\Gamma$.
For $\epsilon=(\epsilon_k)$, we define
\begin{equation}\label{sigmaepsilon}
\tilde\sigma^{(Z,y)}_{\Gamma,\epsilon}:= \tilde\sigma^{(Z,y)}_\Gamma \cap N_\epsilon(f),
\end{equation}
with $N_\epsilon(f)$ defined as in \eqref{Nf}. The principal value regularization of
Definition \ref{regPVintDef} can then be written as
$$ \langle PV(\eta^{(Z,y)}_\Gamma), \varphi \rangle= \lim_{\epsilon \to 0} \int_{\tilde\sigma^{(Z,y)}_{\Gamma,\epsilon}} \varphi \, \eta^{(Z,y)}_\Gamma, $$
where the limit is taken over admissible paths. 

Similarly, given a $\cG_\Gamma$-nest $\cN=\{ \gamma_1, \ldots, \gamma_r \}$, 
we introduce the notation
\begin{equation}\label{sigmaepsilonN}
\tilde\sigma^{(Z,y)}_{\Gamma,\cN,\epsilon}:= \tilde\sigma^{(Z,y)}_\Gamma \cap T_{\cN,\epsilon}(f)
\cap N_{\cN,\epsilon}(f),
\end{equation}
where $T_{\cN,\epsilon}(f)=\{ |f_k|=\epsilon_k, \, k=1,\ldots, r \}$ and $N_{\cN,\epsilon}(f)=\{ |f_k|> \epsilon, \, k=r+1, \ldots, n \}$, where we have ordered the $n$ subgraphs $\gamma$ in 
$\cG_\Gamma$ so that the first $r$ belong to the nest $\cN$.

\begin{prop}\label{etaResN}
For a $\cG_\Gamma$-nest $\cN=\{ \gamma_1, \ldots, \gamma_r \}$, as above, the limit
\begin{equation}\label{FeynResCurrN}
\lim_{\epsilon \to 0} \int_{\tilde\sigma^{(Z,y)}_{\Gamma,\cN,\epsilon}} 
\varphi \, \eta^{(Z,y)}_\Gamma
\end{equation}
determines a pseudomeromorphic current, whose residue part is an iterated residue
supported on $V_{\cN}^{(Z)}=D_{\gamma_1}^{(Z)}\cap
\cdots \cap D_{\gamma_r}^{(Z)}$.
\end{prop}

\proof Let $f_k$, with $k=1,\ldots, n$ with $n= \# \cG_\Gamma$, be the 
defining equations of the $D_\gamma^{(Z)}$. We assume the subgraphs in
$\cG_\Gamma$ are ordered so that the first $r$ belong to the given
$\cG_\Gamma$-nest $\cN$. We can then write the current \eqref{FeynResCurrN}
in the form $\langle C_{r,n}, h \varphi \rangle$, where $C_{r,n}$ is the
elementary pseudomeromorphic current of \eqref{pseudomero} and $h$ is
algebraic without poles.
\endproof

\medskip
\subsubsection{Residue currents as ambiguities}

With the same setting as in Proposition \ref{etaResN}, we then have the
following characterization of the ambiguities of the principal value regularization.

\begin{prop}\label{ambigProp}
The ambiguities in the current-regularization $PV(\eta^{(Z,y)}_\Gamma)$
are given by iterated residues supported on the intersections 
$V_{\cN}^{(Z)}= D_{\gamma_1}^{(Z)}\cap \cdots \cap D_{\gamma_r}^{(Z)}$,
of divisors corresponding to $\cG_\Gamma$-nests $\cN=\{ \gamma_1, \ldots, \gamma_r \}$.
\end{prop}

\proof As above, we have
$$ \langle PV(\eta^{(Z,y)}_\Gamma), \varphi \rangle =
\lim_{\epsilon \to 0} \int_{\tilde\sigma^{(Z,y)}_{\Gamma,\epsilon}} \varphi \, \eta^{(Z,y)}_\Gamma 
 = \lim_{\lambda \to 0} \int_{\tilde\sigma^{(Z,y)}_\Gamma} \frac{|f_n|^{2\lambda_n} \cdots
|f_1|^{2\lambda_1}}{f_n \cdots f_1} \, h\, \varphi $$
If we replace the form $h\, \varphi$ with a form $h \varphi + \bar\partial_\cN \psi$, where
$\cN$ is a $\cG_\Gamma$-nest and the notation $\bar\partial_{\cN} \psi$ means a form
$$ \bar\partial_\cN \psi := \psi_n \cdots \psi_{r+1} \, \bar\partial \psi_r \wedge \cdots \wedge
\bar\partial \psi_1, $$
for test functions $\psi_k$, $k=1,\ldots, n$, we obtain a pseudomeromorphic current
$$ \langle PV(\eta^{(Z,y)}_\Gamma),  \bar\partial_\cN \psi \rangle =
 \langle  [\frac{1}{f_n}] \wedge \cdots \wedge [\frac{1}{f_{r+1}}] \wedge
\bar\partial[\frac{1}{f_r}] \wedge \cdots \wedge \bar\partial[\frac{1}{f_1}], \psi \rangle, $$
with $\psi=\psi_n\cdots \psi_1$. Notice then that the residue part
$$ \langle  \bar\partial[\frac{1}{f_r}] \wedge \cdots \wedge \bar\partial[\frac{1}{f_1}], 
\psi\rangle =\cR_{\cN}(\psi)  $$
is an iterated residue supported on $V_{\cN}^{(Z)}= D_{\gamma_1}^{(Z)}\cap
\cdots \cap D_{\gamma_r}^{(Z)}$.
\endproof

By the results of \S \ref{iterResSec}, and the relation between residue currents
and iterated Poincar\'e residues (see \cite{AT}), when evaluated on algebraic
test forms on the varieties $V_{\cN}^{(Z)}$, these ambiguities can be expressed
in terms of periods of mixed Tate motives, that is, by the general result of \cite{Brown_mt},
in terms of multiple zeta values.

\bigskip
\section{Other regularization methods}\label{RegSec}

We now discuss a regularization method for the evaluation of the
Feynman integral
$$ \int_{\tilde\sigma^{(Z,y)}_\Gamma} \pi_\Gamma^*(\omega^{(Z)}_\Gamma) , $$
with the pullback $\pi_\Gamma^*(\omega^{(Z)}_\Gamma)$
to $Z[\Gamma]$ as in Corollary \ref{divZGamma} and the chain of integration
$\tilde\sigma^{(Z,y)}_\Gamma$ as in \eqref{tildesigmaZ},
obtained from the form $\omega^{(Z)}_\Gamma$ and the chain $\sigma^{(Z)}_\Gamma$
of Definition \ref{FeyamplZcase}. The geometric method of regularization we adopt
is based on the {\em deformation to the normal cone}.

\smallskip

A general method of regularization consists of deforming the chain of integration
so that it no longer intersects the locus of divergences. We first describe briefly why this
cannot be done directly within the space $Z[\Gamma]$ considered above, and 
then we introduce a simultaneous deformation of the form
and of the space where integration happens, so that the integral can be regularized
according to the general method mentioned above.

\smallskip

To illustrate where the problem arises, if one tries to deform the chain of integration
away from the locus of divergence in $Z[\Gamma]$, consider the local problem near a point
$z\in D^{(Z)}_\gamma$ in the intersection of $\tilde\sigma^{(Z,y)}_\Gamma$ with
one of the divisors in the divergence locus of the form
$\pi_\Gamma^*(\omega_\Gamma)$. Near this point, the locus of divergence is
a product $D_\gamma \times X^{\BV_\Gamma}$. We look at the intersection
of the integration chain $\tilde\sigma^{(Z,y)}_\Gamma$ with a small tubular
neighborhood $T_\epsilon$ of $D_\gamma \times X^{\BV_\Gamma}$. We have
$$ \tilde\sigma^{(Z,y)}_\Gamma \cap \partial T_\epsilon = \partial \pi_\epsilon^{-1}( D_\gamma ) 
\times \{ y \}, $$
with $\pi_\epsilon: T_\epsilon(D_\gamma) \to D_\gamma$ the projection of the
2-disc bundle and $\partial \pi_\epsilon^{-1}( D_\gamma )$ a circle bundle, 
locally isomorphic  to  $D_\gamma \times S^1$.  Thus, locally,
$\tilde\sigma^{(Z,y)}_\Gamma \cap T_\epsilon$ looks like a ball $B^{2D|\BV_\Gamma|}
\times \{ 0 \}$ inside a ball $B^{4D|\BV_\Gamma|}$. Locally,
we can think of the problem of deforming the chain of integration in a
neighborhood of the divergence locus as the question of deforming a ball
$B^{2D|\BV_\Gamma|}\times \{0 \}$ leaving fixed the boundary $S^{2D|\BV_\Gamma|-1}
\times \{ 0\}$ inside a ball $B^{4D|V_\Gamma|}$ so as to avoid the locus
$\{ 0 \} \times B^{2D|\BV_\gamma|}$ that lies in the divergence locus. However,
one can check that the spheres $S^{2D|\BV_\Gamma|-1}\times \{ 0 \}$
and $\{ 0 \} \times S^{2D|\BV_\gamma|-1}$ are linked inside the sphere 
$S^{4D|\BV_\Gamma|-1}$. This can be seen, for instance, by computing their
Gauss linking integral (see \cite{ShonkVela})
\begin{equation}\label{Gausslink}
{\rm Lk}(M, N)=
\frac{1}{{\rm Vol}(S)} \int_{M\times N} \frac{\Omega_{k,\ell}(\alpha)}{\sin^n(\alpha)} [x,dx,y,dy]
\end{equation}
with $M=S^k\times \{ 0 \}$, $N=\{ 0 \} \times S^{\ell}$, $S=S^n$, and with
$k=\ell = 2D|\BV_\Gamma|-1$ and $n=4D|\BV_\Gamma|-1$, where
$$ \Omega_{k,\ell}(\alpha) := \int_{\theta=\alpha}^\pi \sin^k(\theta-\alpha) \sin^\ell(\theta) d\theta, $$
$$ \alpha(x,y) := {\rm dist}_{S^n}(x,y), \ \ \ x\in M,\, y\in N, $$
$$ [x,dx,y,dy] := \det(x, \frac{\partial x}{\partial s_1}, \ldots, \frac{\partial x}{\partial s_k},
y, \frac{\partial y}{\partial t_1}, \ldots, \frac{\partial y}{\partial t_\ell}) \, ds\, dt, $$
with $x,y$ the embeddings of $S^k$ and $S^\ell$ in $S^n$ and $s,t$ the local
coordinates on $S^k$ and $S^\ell$. Then one can see (\S 4 of \cite{ShonkVela})
that in $S^n$ with $n=k+\ell+1$ the linking number is
${\rm Lk}(S^k\times \{ 0 \},\{ 0 \} \times S^{\ell}) =1$.

\smallskip

This type of problem can be easily avoided by introducing a simultaneous
deformation of the form $\pi_\Gamma^*(\omega_\Gamma)$ and of the
space $Z[\Gamma]$ as we show in the following.
 
\medskip
\subsection{Form regularization}

We first regularize the form $\omega^{(Z)}_\Gamma$ by embedding
the configuration space $Z^{\BV_\Gamma}$ as the fiber over zero
in a one parameter family
$Z^{\BV_\Gamma} \times \P^1$ and using the additional coordinate
$\zeta \in \P^1$ to alter the differential form in a suitable way.

\smallskip

\begin{defn}\label{regomegaZDef}
The regularization of the Feynman amplitude $\omega^{(Z)}_\Gamma$
on the space $Z^{\BV_\Gamma} \times \P^1$ is given by
\begin{equation}\label{regomegaZ}
\tilde\omega^{(Z)}_\Gamma = \prod_{e\in \BE_\Gamma} 
\frac{1}{(\| x_{s(e)}- x_{t(e)} \|^2 + |\zeta|^2)^{D-1}} 
\bigwedge_{v \in\BV_\Gamma} dx_v \wedge d\bar x_v \, \wedge d\zeta \wedge d\bar\zeta,
\end{equation}
where $\zeta$ is the local coordinate on $\P^1$.
\end{defn}

\smallskip

\begin{lem}\label{divtildeomegaZ}
The divergent locus $\{ \tilde\omega^{(Z)}_\Gamma = \infty \}$ of the
form \eqref{regomegaZ} on $Z^{\BV_\Gamma} \times \P^1$ is given by the locus
$\cup_{e\in \BE_\Gamma} \Delta_e^{(Z)}\subset Z^{\BV_\Gamma} \times \{ 0 \}$.
\end{lem}

\proof The locus of divergence is the intersection of 
$\{ \zeta=0 \}$ and the union of the products 
$\Delta_e^{(Z)}\times \P^1=\{ x_{s(e)}-x_{t(e)} =0 \}$.
\endproof

Notice that we have introduced in the form \eqref{regomegaZ} an additional
variable of integration, $d\zeta \wedge d\bar\zeta$. The reason for shifting the
degree of the form will become clear later in this section (see \S \ref{chaindeformSec}
below), where we see that, when using the deformation to the normal cone, the
chain of integration $\sigma^{(Z,y)}_\Gamma$ is also extended by an additional
complex dimension to $\sigma^{(Z,y)}_\Gamma \times \P^1$, of which one then
takes a proper transforms and deforms it inside the deformation to the normal
cone. In terms of the distributional interpretation of the Feynman amplitudes of
\S \ref{distrSec}, the relation between the form \eqref{regomegaZ} and the
original amplitude \eqref{amplitudeZ} can be written as
\begin{equation}\label{amplidelta}
\omega^{(Z)}_\Gamma = \int  \prod_{e\in \BE_\Gamma} 
\frac{\delta(\zeta=0)}{(\| x_{s(e)}- x_{t(e)} \|^2 + |\zeta|^2)^{D-1}} 
\bigwedge_{v \in\BV_\Gamma} dx_v \wedge d\bar x_v \, \wedge d\zeta \wedge d\bar\zeta,
\end{equation}
where the distributional delta constraint can be realized as a limit of normalized
integrations on small tubular neighborhoods of the central fiber $\zeta=0$ in the
trivial fibration $Z^{\BV_\Gamma}\times \P^1$.

\medskip
\subsection{Deformation to the normal cone}

The deformation to the normal cone is the natural algebro-geometric
replacement for tubular neighborhoods in smooth geometry, see \cite{Ful}.
We use it here to extend the configuration space $Z^{\BV_\Gamma}$
to a trivial fibration $Z^{\BV_\Gamma} \times \P^1$ and then replacing the
fiber over $\{ 0 \} \in \P^1$ with the wonderful compactification $Z[\Gamma]$.
This will allow us to simultaneously regularize the form and the chain of
integration. For simplicity we illustrate the construction for the case where the
graph $\Gamma$ is itself biconnected.

\smallskip

\begin{prop}\label{defnormconePro}
Let $\Gamma$ be a biconnected graph. Starting with the product
$Z^{\BV_\Gamma} \times \P^1$, a sequence of blowups along
loci parameterized by the $\Delta_\gamma^{(Z)}\times \{0\}$, with
$\gamma$ induced biconnected subgraphs yields a variety $\cD(Z[\Gamma])$
fibered over $\P^1$ such that the fiber over all points $\zeta\in \P^1$
with $\zeta\neq 0$ is still equal to $Z^{\BV_\Gamma}$, while the
fiber over $\zeta=0$ has a component equal to the wonderful
compactification $Z[\Gamma]$ and other components given
by projectivizations $\P(C\oplus 1)$ with $C$ the normal cone 
of the blowup locus.
\end{prop}

\proof We start with the product $Z^{\BV_\Gamma} \times \P^1$.
We then perform the first blowup of the iterated sequence 
of \S \ref{blowupSec} on the fiber over $\zeta=0$ namely we
blowup the locus $\Delta_\Gamma^{(Z)} \times \{ 0 \}$ , with $\Delta_\Gamma^{(Z)}$ 
the deepest diagonal, inside $Z^{\BV_\Gamma} \times \P^1$. (Note that this
is where we are using the biconnected hypothesis on $\Gamma$, otherwise
the first blowup may be along induced biconnected subgraphs with a smaller
number of vertices.) The blowup ${\rm Bl}_{\Delta_\Gamma^{(Z)} \times \{ 0 \}}(Z^{\BV_\Gamma} \times \P^1)$ is equal to $Z^{\BV_\Gamma} \times (\P^1\smallsetminus \{ 0 \})$ away from
$\zeta =0$, while over the point $\zeta=0$ it has a fiber with two components. One of the
components is isomorphic to the blowup of $Z^{\BV_\Gamma}$ along $\Delta_\Gamma^{(Z)}$,
that is, ${\rm Bl}_{\Delta_\Gamma^{(Z)}}(Z^{\BV_\Gamma})=Y_1$,
with the notation of \S \ref{blowupSec}. The other component is equal to
$\P(C_{Z^{\BV_\Gamma}}(\Delta_\Gamma^{(Z)})\oplus 1)$ where 
$C_{Z^{\BV_\Gamma}}(\Delta_\Gamma^{(Z)})$ is the
normal cone of $\Delta_\Gamma^{(Z)}$ in $Z^{\BV_\Gamma}$. Since
$\Delta_\Gamma^{(Z)}\simeq X \times X^{\BV_\Gamma}$ is smooth, the
normal cone is the normal bundle of $\Delta_\Gamma^{(Z)}$ in $Z^{\BV_\Gamma}$.
The two Cartier divisors $Y_1$ and $\P(C_{Z^{\BV_\Gamma}}(\Delta_\Gamma^{(Z)})\oplus 1)$ meet along $\P(C_{Z^{\BV_\Gamma}}(\Delta_\Gamma^{(Z)}))$.  We can then proceed to
blow up the further loci $\Delta_\gamma^{(Z)}$ with $\gamma \in \cG_{n-1,\Gamma}$ inside
the special fiber $\tilde\pi^{-1}(0)$ in 
${\rm Bl}_{\Delta_\Gamma^{(Z)} \times \{ 0 \}}(Z^{\BV_\Gamma} \times \P^1)$, where
$$ \tilde\pi: {\rm Bl}_{\Delta_\Gamma^{(Z)} \times \{ 0 \}}(Z^{\BV_\Gamma} \times \P^1)
\to Z^{\BV_\Gamma} \times \P^1 $$
is the projection. These loci lie in the intersection of the two components 
of the special fiber $\tilde\pi^{-1}(0)$. Thus, at the next stage we obtain a 
variety that again agrees with $Z^{\BV_\Gamma} \times (\P^1\smallsetminus \{ 0 \})$
away from the central fiber, while over $\zeta=0$ it now has a component equal to
$Y_2$ and further components coming from the normal cone after this additional
blowup. After iterating this process as in \S \ref{blowupSec} we obtain a variety that
has fiber $Z^{\BV_\Gamma}$ over all points $\zeta\neq 0$ and over $\zeta=0$
it has a component equal to the wonderful compactification $Z[\Gamma]$ and
other components coming from normal cones.
\endproof

Notice that one can also realize the iterated blowup of \S \ref{blowupSec} 
as a single blowup over a more complicated locus and perform the deformation
to the normal cone for that single blowup. We proceed as in Proposition \ref{defnormconePro},
as it will be easier in this way to follow the effect that this deformation has on the
motive.  

\smallskip

\begin{figure}
\begin{center}
\includegraphics[scale=0.5]{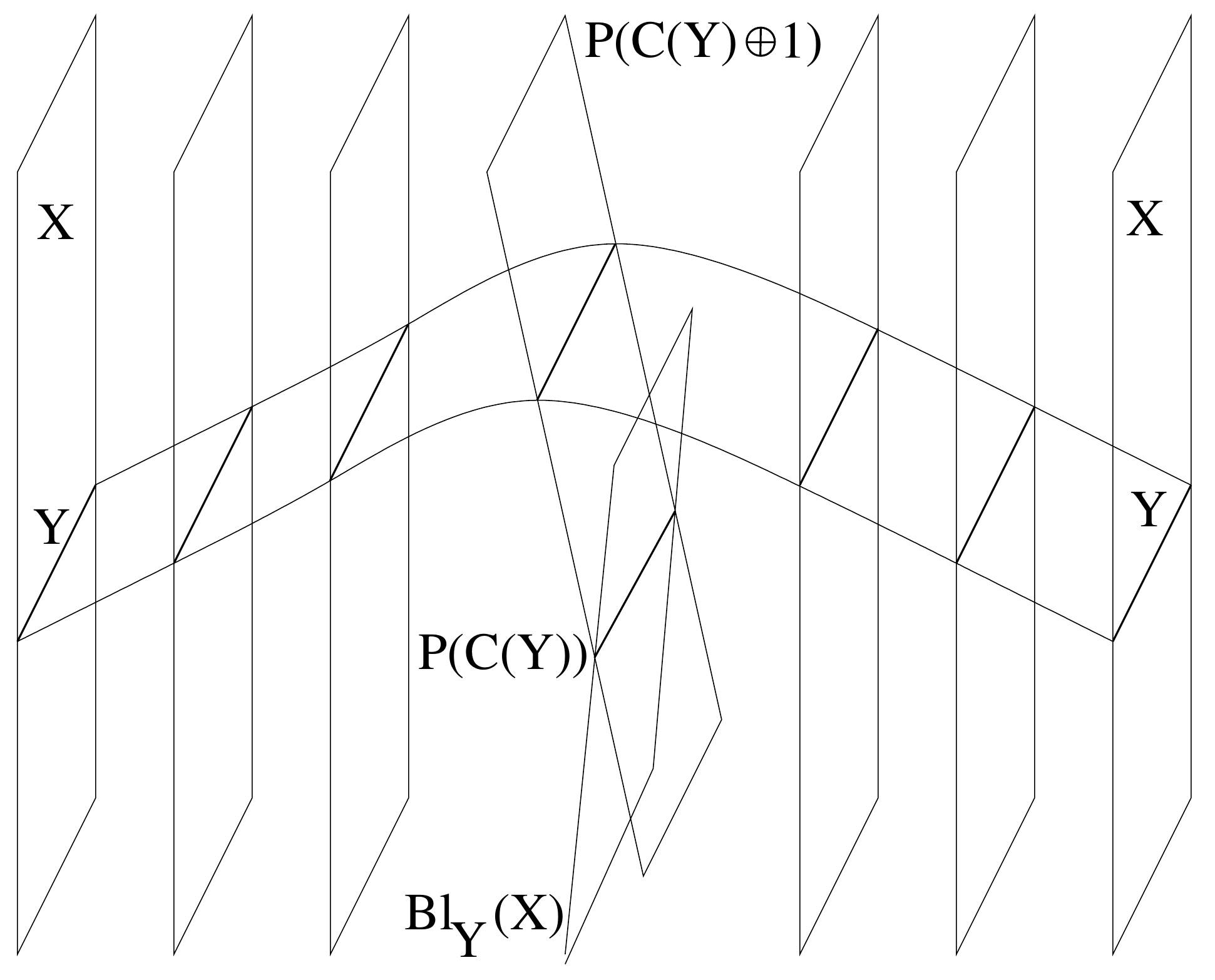}
\caption{Deformation to the normal cone. \label{FigDefCone}}
\end{center}
\end{figure}

The main reason for introducing the deformation to the normal cone,
as we discuss more in detail in \S \ref{chaindeformSec} below, is the
fact that it will provide us with a natural mechanism for deforming 
the chain of integration away from the locus of divergences. The key
idea is depicted in Figure \ref{FigDefCone}, where one considers
a variety $\cX$ and the deformation ${\rm Bl}_{\cY\times \{ 0 \}}(\cX\times \P^1)$.
If $\pi: {\rm Bl}_{\cY\times \{ 0 \}}(\cX\times \P^1) \to \P^1$ denotes the
projection, the special fiber $\pi^{-1}(0)$ has two components, one
given by the blowup ${\rm Bl}_{\cY}(\cX)$ of $\cX$ along $\cY$ and the
other is the normal cone $\P (C_{\cX}(\cY)\oplus 1)$ of $\cY$ inside $\cX$.
The two components meet along $\P(C_{\cX}(\cY))$. As shown in
\S 2.6 of \cite{Ful2}, one can use the deformation to the normal cone
to deform $\cY$ to the zero section of the normal cone. Thus, given
a subvariety $\cZ\subset \cY$ the proper transform $\overline{\cZ\times \P^1}$
in ${\rm Bl}_{\cY\times \{ 0 \}}(\cX\times \P^1)$ gives a copy of $\cZ$ inside
the special fiber $\pi^{-1}(0)$ lying in the normal cone component, see
Figure \ref{FigDefCone}.

\medskip
\subsection{Deformation and the motive}

We check that passing from the space $Z^{\BV_\Gamma}$ to the
deformation $\cD(Z[\Gamma])$ described in Proposition \ref{defnormconePro}
does not alter the nature of the motive. 

\smallskip

It is easy to see that this is the case at the level of virtual motives,
that is, classes in the Grothendieck ring of varieties.

\begin{prop}\label{GroclassD}
If the class $[X]$ in the Grothendieck ring of varieties $K_0(\cV)$ is
a virtual mixed Tate motive, that is, it lies in the subring $\Z[\bL]$ 
generated by the Lefschetz motive $\bL=[\A^1]$, then the class of
$\cD(Z[\Gamma])$ is also in $\Z[\bL]$.
\end{prop}

\proof As shown in Proposition \ref{defnormconePro}, the space 
$\cD(Z[\Gamma])$ is a fibration over $\P^1$, which is a trivial
fibration over $\P^1\smallsetminus \{ 0 \}$ with fiber $Z^{\BV_\Gamma}$.
By inclusion-exclusion, we can write the class $[\cD(Z[\Gamma])]$ in
$K_0(\cV)$ as a sum of the class of the fibration over $\P^1\smallsetminus \{ 0 \}$,
which is
$$ [Z^{\BV_\Gamma} \times (\P^1\smallsetminus \{ 0 \})] = [X]^{2\BV_\Gamma}\, \bL $$
and the class $[\pi^{-1}(0)]$ of the fiber over $\zeta=0$, with
$\pi: \cD(Z[\Gamma]) \to \P^1$ the fibration. The component $ [X]^{2\BV_\Gamma} \bL$
is in $\Z[\bL]$  if the class $[X]\in \Z[\bL]$ as we are assuming, so we need to
check that the class $[\pi^{-1}(0)]$ is also in $\Z[\bL]$. The locus $\pi^{-1}(0)$
is constructed in a sequence of steps as shown in Proposition \ref{defnormconePro}.
At the first step, we are dealing with the deformation to the normal cone
${\rm Bl}_{\Delta^{(Z)}_\Gamma}(Z^{\BV_\Gamma}\times \P^1)$ and the fiber
over zero is the union of $Y_1$ and $\P(C_{Z^{\BV_\Gamma}}(\Delta_\Gamma^{(Z)})\oplus 1)$,
intersecting along $\P(C_{Z^{\BV_\Gamma}}(\Delta_\Gamma^{(Z)}))$. Since 
$\Delta_\Gamma^{(Z)}\simeq X \times X^{\BV_\Gamma}$ is smooth and a Tate motive, 
$\P(C_{Z^{\BV_\Gamma}}(\Delta_\Gamma^{(Z)})\oplus 1)$ is a projective bundle
over a Tate motive so it is itself a Tate motive. So is 
$\P(C_{Z^{\BV_\Gamma}}(\Delta_\Gamma^{(Z)}))$, for the same reason. 
So is also $Y_1$ because of the blowup formula for Grothendieck classes \cite{Bitt},
$$ [Y_1] = [Z^{\BV_\Gamma}] + \sum_{k=1}^{{\rm codim}(\Delta_\Gamma^{(Z)}\times\{0\})-1} [\Delta_\Gamma^{(Z)}]\, \bL^k . $$
By the inclusion-exclusion relations in the Grothendieck ring, it then follows
that if the two components of the fiber over zero are in $\Z[\bL]$ and the class of 
their intersection also is, then so is also the class of the union, which is the
class of the fiber itself.  At the next step the fiber over zero is blown up again,
this time along the (dominant transforms of) 
$\Delta_\gamma^{(Z)}$ with $\gamma \in \cG_{n-1,\Gamma}$. Each of these
is a blowup of a variety whose class is a virtual mixed Tate motive along
a locus whose class is also a virtual mixed Tate motive, hence repeated
application of the blowup formula in the Grothendieck ring and an
argument analogous to the one used in the first step shows that the
Grothendieck  class of the fiber over zero is also in $\Z[\bL]$.
\endproof

\smallskip

We can then, with a similar technique, improve the result from the
level of Grothendieck classes to the level of motives.

\begin{prop}\label{defmixedtate}
If the motive $\m(X)$ of the variety $X$ is mixed Tate, then the motive $\m(\cD(Z[\Gamma])$
of the deformation $\cD(Z[\Gamma])$ is also mixed Tate.
\end{prop}

\proof As in the case of the Grothendieck classes, it suffices to check that, 
at each step in the construction of
$\cD(Z[\Gamma])$, the result remains inside the category of mixed
Tate motives. It is clear that, if $\m(X)$ is mixed Tate, then $\m(Z)$, $\m(Z^{\BV_\Gamma})$
and $\m(Z^{\BV_\Gamma}\times \P^1)$ also are. At the next step, we use the
blowup formula for Voevodsky motives (Proposition 3.5.3 of \cite{voe}) and we
obtain
$$ \m({\rm Bl}_{\Delta_\Gamma^{(Z)}\times \{0\}}(Z^{\BV_\Gamma}\times \P^1))
= $$ $$ \m(Z^{\BV_\Gamma}\times \P^1)) \oplus 
\bigoplus_{k=1}^{{\rm codim}(\Delta_\Gamma^{(Z)}\times\{0\})-1} 
\m(\Delta_\Gamma^{(Z)})(k) [2k]. $$
This implies that $\m({\rm Bl}_{\Delta_\Gamma^{(Z)}\times \{0\}}(Z^{\BV_\Gamma}\times \P^1))$
is mixed Tate if $\m(X)$ is. The successive steps are again obtained by blowing
up loci $\Delta_\gamma^{(Z)}$ whose motive $\m(\Delta_\gamma^{(Z)})$ is mixed Tate,
inside a variety whose motive is mixed Tate by the previous step, hence repeated
application of the blowup formula for motives yields the result.
\endproof

The analog of Remark \ref{DefZrem} also holds for the motive $\m(\cD(Z[\Gamma])$.

\medskip
\subsection{Form regularization on the deformation}

Let $\tilde\omega_\Gamma^{(Z)}$ be the regularized form
defined in \eqref{regomegaZ}. In order to allow room for a
regularization of the chain of integration, we pull it back to
the deformation to the normal cone described above.

\begin{defn}\label{tildeomegaZconeDef}
The regularization of the form $\omega_\Gamma^{(Z)}$ on the 
deformation space $\cD(Z[\Gamma])$ is the pullback
\begin{equation}\label{tildeomegaZcone}
\tilde\pi_\Gamma^*(\tilde\omega_\Gamma^{(Z)}),
\end{equation}
where $\tilde\pi_\Gamma: \cD(Z[\Gamma]) \to Z^{\BV_\Gamma}\times \P^1$
is the projection and $\tilde\omega_\Gamma^{(Z)}$ is the regularization
of \eqref{regomegaZ}.
\end{defn}

The locus of divergence $\{ \tilde\pi_\Gamma^*(\tilde\omega_\Gamma^{(Z)}) =\infty \}$
inside the deformation space $\cD(Z[\Gamma])$ is then given by the following.

\begin{lem}\label{divomegaZDeform}
The locus of divergence of the regularized Feynman amplitude 
$\tilde\pi_\Gamma^*(\tilde\omega_\Gamma^{(Z)})$ on the space
 $\cD(Z[\Gamma])$ is a union of divisors inside the central fiber,
\begin{equation}\label{divergenceDiv}
\bigcup_{\Delta_\gamma^{(Z)}\in \cG_\Gamma} D_\gamma^{(Z)} \subset \pi^{-1}(0),
\end{equation}
where $\pi: \cD(Z[\Gamma]) \to \P^1$ is the projection of the fibration.
\end{lem}

\proof When pulling back the regularized form $\tilde\omega_\Gamma^{(Z)}$
from $Z^{\BV_\Gamma}\times \P^1$ to $\cD(Z[\Gamma])$, the poles of
$\tilde\omega_\Gamma^{(Z)}$ along the diagonals $\Delta^{(Z)}_\gamma \times \{0 \}$
yield (as in Proposition \ref{prop_form}
and Corollary \ref{divZGamma}) poles along the divisors $D_\gamma^{(Z)}$, 
contained in the central fiber $\pi^{-1}(0)$ at $\zeta=0$ of $\cD(Z[\Gamma])$.
\endproof

\medskip
\subsection{Deformation of the chain of integration}\label{chaindeformSec}

We now describe a regularization of the chain of integration, based on the
deformation to the normal cone. 

\begin{prop}\label{defchainD}
The proper transform of the chain $\sigma^{(Z,y)}_\Gamma \times \P^1$
inside $\cD(Z[\Gamma])$ gives a deformation of the chain of integration,
which does not intersect the locus of divergences of the form
$\tilde\pi_\Gamma^*(\tilde\omega_\Gamma^{(Z)})$.
\end{prop}

\proof
Consider the chain $\sigma^{(Z,y)}_\Gamma 
=X^{\BV_\Gamma}\times \{  y \}$ of \eqref{sigmaZ}, inside $Z^{\BV_\Gamma}$. 
Extend it to a chain $\sigma^{(Z,y)}_\Gamma \times \P^1$ inside 
$Z^{\BV_\Gamma}\times \P^1$. Let $\overline{\sigma^{(Z,y)}_\Gamma \times \P^1}$
denote the proper transform in the blowup $\cD(Z[\Gamma])$. Then, as 
illustrated in Figure \ref{FigDefCone}, we obtain a deformation of $\sigma^{(Z,y)}_\Gamma$
inside the normal cone component of the special fiber $\pi^{-1}(0)$ in $\cD(Z[\Gamma])$
that is separated from the intersection with the component given by the
blowup $Z[\Gamma]$.
\endproof

\medskip
\subsection{Regularized integral}

Using the deformation of the chain of integration and of the form, one
can regularize the Feynman integral by 
\begin{equation}\label{regintD}
\int_{\Sigma^{(Z,y)}_\Gamma} \tilde\pi_\Gamma^*(\tilde\omega_\Gamma^{(Z)}),
\end{equation}
where $\Sigma^{(Z,y)}_\Gamma$ denotes the $(2\BV_\Gamma +2)$-chain
on $\cD(Z[\Gamma])$ obtained as in Proposition \ref{defchainD}.
As in \eqref{amplidelta}, one also has a corresponding integral on the intersection
of the deformed chain $\Sigma^{(Z,y)}$ with the central fiber, which we can write as
$$ \int_{\Sigma^{(Z,y)}_\Gamma} \delta(\pi^{-1}(0)) \,\,
\tilde\pi_\Gamma^*(\tilde\omega_\Gamma^{(Z)}). $$

\medskip
\subsubsection{Behavior at infinity}

The regularization \eqref{regintD} described above avoids divergences along the divisors
$D_\gamma^{(Z)}$ in $Z[\Gamma]$. It remains to check the behavior at infinity,
both in the $\P^1$-direction added in the deformation construction, and along
the locus $\cD_\infty$ in $\cD(Z[\Gamma])$ defined, in the intersection of
each fiber $\pi^{-1}(\zeta)$ with the chain of integration $\Sigma^{(Z,y)}_\Gamma$, 
by $\Delta_{\Gamma,\infty} :=X^{\BV_\Gamma} \smallsetminus \A^{D\BV_\Gamma}$.

\begin{prop}\label{regatinfty}
The integral \eqref{regintD} is convergent at infinity when $D>2$.
\end{prop}

\proof
For the behavior of \eqref{regintD} when $\zeta \to \infty$ in $\P^1$,
we see that the form behaves like $r^{-2D+2}\, r dr$, where $r=|\zeta|$.
This gives a convergent integral for $2D-3>1$.
For the behavior at  $\Delta_{\Gamma,\infty}$, consider first the
case where a single radial coordinate $r_v =| x_v| \to \infty$. 
In polar coordinates, we then have a radial integral
$r_v^{-(2D-2) \BE_{\Gamma,v}}\, r^{D-1} dr$, where 
$\BE_{\Gamma,v}=\{ e\in \BE_\Gamma \,|\, v\in \partial(e)\}$ 
is the valence $\upsilon(v)$ of the vertex $v$.
This gives a convergent integral when $(2D-2)\upsilon(v) -D+1 >1$.
Since $\upsilon(v) \geq 1$ and $2D-2\geq 0$, 
we have $(2D-2)\BE_{\Gamma,v} -D+1 \geq D-1$, so the
condition is satisfied whenever $D> 2$. 
More generally, one can have several $r_v\to \infty$. The
strongest constraint comes from the case that behaves like
$r^{-(2D-2) \sum_v \upsilon(v)} r^{D|\BV_\Gamma|-1}$. In
this case the convergence condition is given by
$(2D-2)\upsilon_\Gamma - D |\BV_\Gamma| >0$,
where $\upsilon_\Gamma =\sum_{v\in \BV_\Gamma} \upsilon(v)$.
Again we have $\upsilon_\Gamma \geq |\BV_\Gamma|$, 
and we obtain
$$ (2D-2)\upsilon_\Gamma - D |\BV_\Gamma| \geq 
(D-2) |\BV_\Gamma| > 0, $$
whenever $D>2$. In this case the
condition for convergence at $|\zeta|\to \infty$ is also satisfied.
\endproof

\bigskip

\noindent {\bf Acknowledgments.}  
Parts of this work have been carried out during visits of the first author to 
the California Institute of Technology, the Institut des Hautes \'Etudes  
Scientifiques and the Max Planck Institut f\"ur Mathematik. We thank
these institutions for their support. 
The second author acknowledges support from NSF grants 
DMS-0901221, DMS-1007207, DMS-1201512, and PHY-1205440.
The authors thank Paolo Aluffi and Spencer Bloch for many useful
conversations.

\smallskip

The first author's son, Uzay, was diagnosed with neuroblastoma, at the time
when we were in the early stages of this project. His doctor, Tanju  Ba\c{s}ar\i r \"Ozkan,
not only saved Uzay with her exceptional professional skills, but also
gave constant personal support, so that \"O.C. could return to work and continue 
the project.  This paper is dedicated to her, to express the author's deepest 
gratitude: {\it Ellerin dert g\"ormesin Tanju abla.}


\begin{thebibliography}{10}

\bibitem{AYu} L.A.~Aizenberg, A.P.~Yuzhakov, {\em Integral representations in multidimensional
complex analysis}, Transl. Amer. Math. Soc. Vol.58, 1980.

\bibitem{ATYu} L.A.~Aizenberg, A.K.~Tsikh, A.P.~Yuzhakov, {\em Multidimensional residues
and applications}, in ``Several Complex Variables, II", Encyclopedia of Mathematical
Sciences, Vol.8,  Springer Verlag, 1994.

\bibitem{AT} A.G.~Aleksandrov, A.K.~Tsikh, {\em Multi-logarithmic differential forms
on complete intersections}, J. Siberian Fed. Univ. Math and Phys. 2 (2008) 105--124.

\bibitem{Alis} S.~Ali\v{s}auskas, {\em Coupling coefficients of $SO(n)$ and integrals
involving Jacobi and Gegenbauer polynomials}, J. Phys. A, 35 (2002) 7323--7345.

\bibitem{ApVu} T.M.~Apostol, T.H.~Vu, {\em Dirichlet series related to the Riemann
zeta function}, J. Number Theory, 19 (1984) 85--102.

\bibitem{BaEr} H.~Bateman, A.~Erdelyi, {\em Higher transcendental functions}, Vol.2,
McGraw-Hill, 1953.

\bibitem{BeLo} I.~Benjamini, L.~Lov\'asz, {\em Harmonic and analytic functions on
graphs}, J. Geom. 76 (2003) 3--15.

\bibitem{BGVY} C.A.~Berenstein, R.~Gay, A.~Vidras, A.~Yger, {\em Residue
currents and Bezout identities}, Progress in Math. Vol.114, Birk\"auser, 1993.

\bibitem{Bitt} F.~Bittner, {\em The universal Euler characteristic for varieties 
of characteristic zero}, Compos. Math. 140 (2004), no. 4, 1011--1032.

\bibitem{Bloch} S.~Bloch, lecture at Caltech, May 2012.

\bibitem{Brown_mt} F.~Brown, {\em Mixed Tate motives over $\Z$},
Ann. of Math. 175 (2012) N.2, 949--976.

\bibitem{CeyMar} \"O. Ceyhan, M. Marcolli {\it Feynman integrals and 
motives of configuration spaces}, Communications
in Mathematical Physics, Vol.313 (2012) N.1, 35--70

\bibitem{CheKaTka} K.G.~Chetyrkin, A.L.~Kataev, F.V.~Tkachov, 
{\em New approach to evaluation of multiloop Feynman integrals:
the Gegenbauer polynomial $x$-space technique}, Nuclear Phys. B 174
(1980) 345--377.

\bibitem{CoHe} N.~Coleff, M.~Herrera, {\em Les courants r\'esidus associ\'es \`a 
une forme m\'eromorphe}, Lect. Notes Math. 633, Springer, 1978.

\bibitem{Deligne} P.~Deligne, {\em \'Equations diff\'erentielles \`a points singuliers r\'eguliers}, Lecture Notes in Math., 163, Springer, Berlin, 1970.

\bibitem{DeII} P.~Deligne, {\em Th\'eorie de Hodge. II}, Inst. Hautes \'Etudes Sci. 
Publ. Math. No. 40 (1971), 5--57.

\bibitem{Freitas} P.~Freitas, {\em Integrals of polylogarithmic functions, recurrence
relations, and associated Euler sums}, Mathematics of Computation, Vol.~74 (2005)
N.~251, 1425--1440.

\bibitem{Ful} W.~Fulton, {\em Intersection theory}, Second Edition, Springer, 1998.

\bibitem{Ful2} W.~Fulton, {\em Introduction to Intersection Theory in Algebraic Geometry},
American Mathematical Soc., 1984.

\bibitem{fm}  W.~Fulton, R.~MacPherson, {\it A compactification of configuration spaces.} 
Ann. of Math. (2) 139 (1994), no. 1, 183--225.

%\bibitem{GaMi} T.~Gallai, A.N.~Milgram, 
%{\em Verallgemeinerung eines graphentheoretischen Satzes von R\'edei}, 
%Acta Sci. Math. (Szeged) 21 (1960) 181--186. 

\bibitem{Gallier} J.~Gallier, {\em Notes on spherical harmonics and linear
representations of Lie groups}, preprint, 2009.

\bibitem{Gon} A.B.~Goncharov, {\em Periods and mixed motives}, arXiv.math/0202154v2.

\bibitem{Griff} Ph.~Griffiths, {\em On the periods of certain rational integrals. I, II}, Ann. of Math. 
(2) 90 (1969),  460--495; 496--541.

\bibitem{GrHa} Ph.~Griffiths, J.~Harris, {\em Principles of algebraic geometry}, Wiley, 1994.

\bibitem{Gro} A.~Grothendieck, {\em On the de Rham cohomology of algebraic varieties}, 
Inst. Hautes \'Etudes Sci. Publ. Math. No. 29 (1966) 95--103.

\bibitem{Hart} R.~Hartshorne, {\em On the De Rham cohomology of algebraic varieties}, 
Inst. Hautes \'Etudes Sci. Publ. Math. No. 45 (1975), 5--99.

\bibitem{HeLieb} M.~Herrera, D.~Lieberman, {\em Residues and principal values on
complex space}, Math. Ann. 194 (1971) 259--294.

\bibitem{Junk} G.~Junker, {\em Explicit evaluation of coupling coefficients for the
most degenerate representations of $SO(n)$}, J. Phys. A, 26 (1993) 1649--1661.

\bibitem{Kyt} A.M.~Kytmanov, {\em The Bochner--Martinelli integral and its
applications}, Birkh\"auser, 1995.

\bibitem{li} L.~Li, {\it Chow Motive of FultonÐMacPherson Configuration Spaces and Wonderful 
Compactifications.} Michigan Math. J. 58 (2009).

\bibitem{li2} L.~Li, {\em Wonderful compactification of an arrangement of subvarieties}. Michigan Math. J. 58 (2009), no. 2, 535--563.

\bibitem{Mordell} L.J.~Mordell, {\em On the evaluation of some multiple
series}, J. London Math. Soc. 33 (1958) 368--371.

\bibitem{Mori} M.~Morimoto, {\em Analytic functionals on the sphere}, Translations
of Mathematical Monographs, Vol.178, AMS, 1998.

\bibitem{NiStoTo} N.M.~Nikolov, R.~Stora, I.~Todorov, {\em Configuration space
renormalization of massless QFT as an extension problem for associate homogeneous
distributions}, IHES preprints 2011: IHES/P/11/07.

\bibitem{Saito} K.~Saito, {\em Theory of logarithmic differential forms and logarithmic
vector fields},  J. Fac. Sci. Univ. Tokyo Sect. IA Math. 27 (1980), no. 2, 265--291. 

\bibitem{Sebil} D.~S\'ebilleau, {\em On the computation of the integrated products
of three spherical harmonics}, J. Phys. A, 31 (1998) 7157--7168.

\bibitem{ShonkVela} C.~Shonkwiler, D.S.~Vela-Vick, {\em Higher-dimensional
linking integrals}, Proc. Amer. Math. Soc. 139 (2010) N.4, 1511--1519.

\bibitem{Schw} L.~Schwartz, {\em Division par une fonction holomorphe sur une vari\'et\'e
analytique complexe}, Summa Brasil. Math. 3 (1955) 181--209.

\bibitem{Stan} R.~Stanley, {\em Acyclic orientations of graphs}, Discrete Math. 5 (1973)
171--178.

\bibitem{Stein} E.~Stein, G.~Weiss, {\em Introduction to Fourier analysis on Euclidean
spaces}, Princeton, 1971.

\bibitem{TarShla} N.N.~Tarkhanov, A.A.~Shlapunov, {\em Green's formulas in
complex analysis}, Journal of Mathematical Sciences, Vol.120 (2004) N.6, 1868--1900.

\bibitem{Tornheim} L.~Tornheim, {\em Harmonic double series}, Amer. J. Math.
72 (1950) 303--314.

\bibitem{TsYg} A.~Tsikh, A.~Yger, {\em Residue currents. Complex analysis},
J. Math. Sci. (N.Y.) 120 (2004) N. 6, 1916--1971. 

\bibitem{Vilenkin} N.J.~Vilenkin, {\em Special functions and the theory of group
representations}, Translations
of Mathematical Monographs, Vol.22, AMS, 1968.

\bibitem{voe} V.~Voevodsky, {\it Triangulated categories of motives over a field} in ÒCycles, 
transfer and motivic homology theories, pp. 188Ð238, Annals of Mathematical Studies, Vol. 
143, Princeton, 2000.

\bibitem{Wald} M.~Waldschmidt, {\em Multiple polylogarithms: an introduction}, in
``Number theory and discrete mathematics (Chandigarh, 2000)", pp.1--12, 
Trends Math., Birkh\"auser, 2002.


\end{thebibliography}
\end{document}